\documentclass[11pt,reqno]{amsart}
\usepackage{amsmath,amsfonts,amsthm,color,amssymb, mathrsfs}

\usepackage{pxfonts}
\makeatletter
\@tfor\@tempa:=
\alpha \beta \gamma \delta \epsilon \zeta \eta \theta \iota \kappa \lambda \mu
\nu \xi \pi \rho \sigma \tau \upsilon \phi \chi \psi \omega \varepsilon
\vartheta \varpi \varrho \varsigma \varphi  
\do {% for some reason, must proceed in two steps
	\count255 \numexpr\@tempa+"7000\relax
	\expandafter\mathchardef\@tempa \count255 }
\makeatother

\usepackage{xcolor}
\usepackage[T1]{fontenc}
\usepackage{wasysym}
\usepackage[normalem]{ulem}
\usepackage{stmaryrd} 

\usepackage[left=1in, right=1in, top=1.1in,bottom=1.1in]{geometry}
\usepackage{enumitem}
\usepackage{hyperref}
\hypersetup{
	colorlinks   = true,
	citecolor    = blue,
	linkcolor=blue
}
\allowdisplaybreaks
\usepackage{csquotes}
\usepackage{graphicx}
\usepackage{pstricks}
\usepackage{lmodern}
\newtheorem{assumption}{Assumption}

\newtheorem{thm}{Theorem}[section]

\newtheorem{lem}[thm]{Lemma}

\def \a{{\alpha}}

\def \d{{\mathrm{d}}}

\newcommand{\Om}{\Omega}

\newcommand{\cF}{\mathcal{F}}

\newcommand{\PP}{\mathbb{P}}

\newcommand{\E}{\mathbb{E}}

\newcommand{\bean}{\begin{eqnarray*}}
	\newcommand{\eean}{\end{eqnarray*}}

\newcommand{\EE}{\mathbb{E}}

\newcounter{bean}
\newcommand{\benuma}{\setlength{\labelwidth}{.25in}
	
	\begin{list}
		{(\alph{bean})}{\usecounter{bean}}}
	\newcommand{\eenuma}{\end{list}}

\begin{document}
	
	\title[Limit distributions for SVE with fractional kernels]{Limit  error distributions of   Milstein scheme for  stochastic Volterra equations with singular kernels }
	\author[S. Liu]{Shanqi Liu}
	\address{School of Mathematical Science, Nanjing Normal University, Nanjing 210023, China}
	\email{shanqiliumath@126.com}
	
	\author[Y. Hu]{Yaozhong Hu}
	\address{Department of Mathematical and Statistical Sciences, University of Alberta, Edmonton T6G 2G1, Canada}
	\email{yaozhong@ualberta.ca}
	
	\author[H. Gao]{Hongjun Gao}
	\address{School of Mathematics, Southeast University, Nanjing 211189, China}
	\email{hjgao@seu.edu.cn (Corresponding author)}
	\vspace{-2cm}
	\maketitle
	\vspace{-1cm}
	\begin{abstract}
		For  stochastic Volterra equations driven by standard Brownian and with  singular  kernels    $K(u)=u^{H-\frac{1}{2}}/\Gamma(H+1/2), H\in (0,1/2)$, it is known that the   Milstein scheme has a convergence  rate of $n^{-2H}$.  In this paper, we show that this rate is optimal.  Moreover, we show that the error normalized by $n^{-2H}$   converge stably in law to the (nonzero) solution of a certain linear Volterra equation of random coefficients with the same fractional kernel.
		
		\smallskip
		
	\end{abstract}

	\section{Introduction and main result}
	%	Denote by $\mathbb{R}^{d}$ the $d$-dimensional Euclidean space. 
	Let $(\Omega,\mathcal{F},(\mathcal{F}_t)_{t\ge0},\mathbb{P})$ to be a  complete filtered probability space $(\Omega,\mathcal{F},(\mathcal{F}_t)_{t\ge0},\mathbb{P})$, where $\mathcal{F}_t$ is a nondecreasing family of sub-$\sigma$ fields of $\mathcal{F}$ satisfying the usual conditions. 
	Motivated by an attempt to solve the fractional order 
	stochastic differential equation (e.g.\cite{kilbas}),   the following stochastic Volterra equation (SVE for short) on $d$-dimensional Euclidean space $\mathbb{R}^{d}$ has been studied recently:  
	\begin{align}\label{Volterra eq}
		X_{t}=X_0+\int_{0}^{t}K(t-s)b(X_{s})\d s+\int_{0}^{t}K(t-s)\sigma(X_{s})\d W_s,\quad 0\le t\le T\,, 
	\end{align}
	where $X_0\in\mathbb{R}^d,K(u)=u^{H-\frac{1}{2}}/\Gamma(H+1/2), H\in (0,1/2)$, $
	W=(W_t\,, t\ge 0)$
	is an $\mathrm{m}$-dimensional Brownian motion defined on  
	$(\Om, \cF, \PP)$, and the coefficients $b:\mathbb{R}^d\to\mathbb{R}^d$, $\sigma:\mathbb{R}^d\to\mathbb{R}^d\times \mathbb{R}^m$ are continuous satisfying some conditions that we are going to specify  late in this section  before we state our main result.  
	
	%As natural extension of (deterministic) Volterra equations, 
	Since the explicit solution of SVEs with singular kernels  is  rarely known  one has to rely on numerical methods for simulations of these equations. Various time-discrete numerical approximation schemes for \eqref{Volterra eq} have been investigated  in recent years. 
	%The numerical schemes of
	%	(regular) SVEs have received attention only quite recently. 
	An elementary and yet  widely used numerical method for \eqref{Volterra eq} is the following Euler-Maruyama    scheme, studied by Zhang    \cite{ZX1},    Li et al.\cite{LHH2} and Richard et al.\cite{RTY}.  To describe this scheme concisely    let us   we consider only uniform partitions of the interval $[0,T]$: $\pi_n: 0=t_0<t_1<\cdots<t_{[nT]+1}=T\wedge \frac{[nT]+1}{n}$,  where $ t_i=\frac{i}{n},i=0,1,\cdots,[nT]$ (we shall consider this type of partitions throughout the paper). For every positive integer $n$, Euler-Maruyama    approximation $X ^{  e, n}$  of the  solution is given by 
	\begin{align}\label{Euler scheme}
		X ^{  e, n}_t  =&X_0+\int_{0}^{t}K(t-\frac{[ns]}{n})b(X ^{  e, n} _{\frac{[ns]}{n}})\d s+\int_{0}^{t}K(t-\frac{[ns]}{n})\sigma(X ^{  e, n}_{\frac{[ns]}{n}})\d W_s,
	\end{align}
	where $[a]$ denotes the largest integer which is less than or equal to $a$.  	      It was proved by \cite{LHH2} and \cite{RTY}
	that there exists a positive constant $C$, independent  of  $n$ such that 
	\begin{equation}
		\sup_{t\in [0,T]}\E[|X_t- X^{e,n}_t|^p]^{1/p}\leq C n^{-H}. \label{e.1.3} 
	\end{equation} 
	In other word, the  Euler-Maruyama    scheme   for  \eqref{Volterra eq} 
	has  a rate of convergence $H$. 
	A variant of the  Euler-Maruyama scheme \eqref{Euler scheme} is  the following discretized version of \eqref{Volterra eq}:
	\begin{align}\label{Euler type scheme}
		X^{em,n}_t=&X_0+\int_{0}^{t}K(t-s)b(X^{em,n}_{\frac{[ns]}{n}})\d s+\int_{0}^{t}K(t-s)\sigma(X^{em,n}_{\frac{[ns]}{n}})\d W_s\,. 
	\end{align}
	For this scheme it would also  expected that \eqref{e.1.3} still holds true. 
	In fact, it  does and we have  
	\begin{equation}
		\sup_{t\in [0,T]}\E[|X_t- X^{em,n}_t|^p]^{1/p}\leq C n^{-H}. \nonumber
	\end{equation}   
	Moreover,  it is proved both in   \cite{FU} and \cite{NS}, independently,  that 
	%this rate of convergence $n^{-H}$ is sharp for this scheme, in the sense that the process
	if we denote $V^n=n^{H}[X^{em,n}_t-X_t]$,
	then    as $n$ tends to infinity the process $V^n$ converges stably in law to the solution $V=(V^1,\cdots,V^d)^{\top}$ of  the following linear  SVE:
	\begin{align}
		V_t^i= \sum_{k=1}^{d}&\int_0^t K(t-s)  \partial_k b^i(X_s)V^k_s \mathrm{d} s+\sum_{j=1}^m \sum_{k=1}^d \int_0^t K(t-s) \partial_k \sigma_j^i(X_s)V^k_s\d W^j_s\nonumber \\
		& +\frac{1}{\Gamma(2H+2)\sin \pi H}\sum_{j=1}^m \sum_{k=1}^d\sum_{k=1}^m \int_0^t K(t-s) \partial_k \sigma_j^i\left(X_s\right)\sigma^k_l(X_{s})\d B^{l,j}_s,\nonumber
	\end{align}
	where $B$ is an $m^2$-dimensional Brownian motion independent of $W$. 
	%	
	%		study for the Euler scheme and proposed a Milstein discretization scheme improving the rate of strong convergence.  

	While the limit error distribution  of  the Euler-Maruyama approximation of  SVE is now known, 
	there is no  result on the  limit error distribution  of  the Milstein  approximation of  SVE.
	We wish to extend the results \cite{FU,NS} concerning the asymptotic error distributions of Euler scheme \eqref{Euler type scheme} to Milstein scheme. The Milstein scheme   for \eqref{Volterra eq} is  proposed in   \cite{LHH2,RTY}  as follows: 
	\begin{align}\label{Milstein scheme}
		X_t^n=&X_0+\int_{0}^{t}K(t-s)\cdot\Big(b(X^n_{\frac{[ns]}{n}})+\nabla b(X^n_{\frac{[ns]}{n}})\cdot  M_s^{n}\Big)\d s\nonumber\\&+\int_{0}^{t}K(t-s)\cdot\Big(\sigma(X^n_{\frac{[ns]}{n}})+\nabla \sigma(X^n_{\frac{[ns]}{n}})\cdot  M_s^{n}\Big)\d W_s,
	\end{align}
	where
	\begin{equation}
		M_s^n :=M_s^{1,n}+M_s^{2,n},
	\end{equation}
	with 
	\begin{align}
		M_s^{1,n} :=&\int_{0}^{\frac{[ns]}{n}}\Big(K(s-r)-K(\frac{[ns]}{n}-r)\Big)\sigma(X^n_{\frac{[nr]}{n}})\d W_r \,, \nonumber\\ M_s^{2,n}:=&\int_{\frac{[ns]}{n}}^{s}K(s-r)\sigma(X^n_{\frac{[nr]}{n}})\d W_r,\nonumber
	\end{align}
	and
	$$\nabla b(\cdot)\cdot M:=\big(\langle \nabla b^k(\cdot), M \rangle \big)_{1\leq k\leq d}\quad \text{and}\quad \nabla \sigma(\cdot)\cdot M:=\big(\langle \nabla \sigma^k_j(\cdot), M \rangle \big)_{1\leq k\leq d,1\leq j \leq m}.$$
	It has also  been  shown in   \cite{LHH2,RTY}  
	%to have 
	%a better convergence rate than \eqref{Euler type scheme}.
	that the rate of convergence for this scheme is   $n^{-2H}$, a significant improvement of the Euler-Maruyama scheme. More precisely, it has been proved that 
	\begin{equation}
		\sup_{t\in [0,T]}\E[|X_t- X^{ n}_t|^p]^{1/p}\leq C n^{-2H}. \nonumber
	\end{equation} 
	It is interesting to know whether  this rate is sharp or not and if yes, what is the limit error 
	distribution.  
	The main objective of this paper is to answer these two questions. The  key point is to indentify the expression of limit Volterra equation. The main difficulty is to give the precise expression of the constant appear in this limit equation. We shall apply fractional stochastic calculus and limit theorem of the integral of fractional parts to overcome this difficulty.
	
	To introduce our main result of this paper, we  make the following assumptions about  the fractional kernel 
	and the coefficients in SVE \eqref{Volterra eq}.

	\begin{assumption}\label{assumption 2.2}
		For any $x,y\in\mathbb{R}^d$% and $t\in [0,T]$
		$$
		|b(x)-b(y)\|+\|\sigma(x)-\sigma(y)\|\leq L_1|x-y|,			
		$$
		%	 $$
		%	 	|b(x)|+\|\sigma(x)\|\leq L_1\big(1+|x|\big),
		%$$ 
		where $L_1>0$ is a positive constant.
	\end{assumption}
	
	\begin{assumption}\label{assumption 2.3} \quad  The gradients  $\nabla b(\cdot)$ and $\nabla \sigma(\cdot)$ exist and for any $x,y\in\mathbb{R}^d$%  and $t \in [0,T]$
		$$\|\nabla b(x)-\nabla b(y)\|\leq L_2|x-y|,$$
		$$\|\nabla \sigma(x)-\nabla \sigma(y)\|\leq L_2|x-y|,$$
		for some positive constant $L_2>0$.
	\end{assumption}
	
	Our main result is
	the following theorem. 
	\begin{thm}\label{t.main}  Assume that the derivatives  of the functions $b$ and $\sigma$ up to second orders are bounded (hence Assumptions \ref{assumption 2.2} and \ref{assumption 2.3} hold). 
		Let $\epsilon \in(0, H)$. Then the process $U^n=n^{2H}(X-X^n)$  converges stably in law in $\mathcal{C}_0^{H-\epsilon}$ 
	%	\footnote{what is $\mathcal{C}_0^{H-\epsilon}$} 
	to the unique solution  process $U=\left(U^1, \ldots, U^d\right)^{\top}$ 
	%which is the unique continuous solution
	 of the following  linear stochastic Volterra equation of random coefficients:  
		\begin{align}\label{limit Volterra eq}
			U_t^i=&\sum_{k=1}^{d}\int_0^t K(t-s) \partial_k b^i(X_s)U^k_s \mathrm{d} s+\sum_{j=1}^m \sum_{k=1}^d \int_0^t K(t-s) \partial_k \sigma_j^i(X_s)U^k_s\d W^j_s\nonumber \\
			& +%\sqrt
			{C_M}\sum_{j=1}^m \sum_{i=1}^{m}\sum_{i'=1}^{m}\sum_{l=1}^d\int_0^{t} K(t-s)\partial_k \sigma_j^i\left(X_s\right) \sigma_i^l\left(X_s\right)\partial_l\sigma^{k}_{i'}(X_{s}) \mathrm{d} B_s^{i,i',j}\,, 
			\\
			&\qquad\quad  t \in[0, T], i=1, \ldots, d,\nonumber
		\end{align}
		where $\mathcal{C}_0^{\lambda}$  denotes the set of all  $\mathbb{R}^d$-valued $\lambda$-H$\mathrm{\ddot{o}}$lder continuous functions on $[0,T]$ vanishing at $t=0$, and $B$ is an $m^3$-dimensional standard Brownian motion, independent of the original Brownian motion $W$, $C_M$ is a 
		certain constant specific to the Milstein scheme. 
		% and defined on some extension of $(\Omega, \mathscr{F}, \mathrm{P})$.
	\end{thm}
	% converges stably in law to the solution $U=(U^1,\cdots,U^d)^{\top}$ of the SVE:
	%\begin{align}
	%	U_t^i=&\sum_{k=1}^{d}\int_0^t K(t-s) \partial_k b^i(X_s)U^k_s \mathrm{d} s+\sum_{j=1}^m \sum_{k=1}^d \int_0^t K(t-s) \partial_k \sigma_j^i(X_s)U^k_s\d W^j_s\nonumber \\
	%	& +\sqrt{C_M}\sum_{j=1}^m \sum_{i=1}^{m}\sum_{i'=1}^{m}\sum_{l=1}^d\int_0^{t} \partial_k \sigma_j^i\left(X_s\right) \sigma_i^l\left(X_s\right)\partial_l\sigma^{k}_{i'}(X_{s}) \mathrm{d} B_s^{i,i',j},\nonumber
	%\end{align}
	%where $B$ is an $m^3$-dimensional Brownian motion independent of $W$.  
	%Due to the expression of $M^n$, the constant $C_M$ is very complicated to compute. 
	%compared with the result in Euler scheme. 

	Let us also mention that in the case of classical  stochastic differential equation, namely, when the Hurst parameter  $H=1/2$ the asymptotic behaviors of the
	normalized error process for Euler-Maruyama  scheme  have
	been studied in \cite{Ja2,KP0} and the references therein. The asymptotic behavior of the
	normalized error process for the continuous time Milstein scheme %\cite{Mil}  
	was studied   in \cite{Yan}. We also mention that still in the absence of the Volterra kernel, when  the driven Brownian motion  is fractional Brownian motion, there are also a number of results on the rate of convergence and asymptotic error distributions of Euler-type
	numerical schemes, we refer   to \cite{DNT,FR,HLN,HLN2,LT} and the references therein.
	
	Here is how our paper is structured: In Section \ref{sec 2}, we collect and prove some technical lemmas. Finally in Section \ref{sec 3}, we prove the main result, including the uniqueness of the solution to \eqref{limit Volterra eq}.

	\section{Technical lemmas}
	%{Preparation}
	\label{sec 2}

	The following lemmas will be used several times to evaluate integrals with the convolution kernel.
	\begin{lem}\cite{LHG1}\label{basic lemma} Let $K(u)=u^{H-1/2}/\Gamma(H+1/2)$.  Then 
			\[
			\left\{	\begin{split}
				&\int_{0}^{h}K(t)\d t=O(h^{H+1/2})\,, \quad \int_{0}^{T}\big(K(t+h)-K(t)\big)\d t=O(h^{H+1/2})\,,\\ 
				& \int_{0}^{h}K(t)^2\d t =O(h^{2H})\,, \quad \int_{0}^{T}\big(K(t+h)-K(t)\big)^2\d t =O(h^{2H})\,, 
			\end{split}\right. 
			\]
			where the notation $A=O(B)$ for two quantities  $A$ and $B$ means that there is a constant $C$ such that
			$A\le CB$. 
		
		Moreover, for any adapted $\mathbb{R}^{d}$-valued process $Y$ and $\mathbb{R}^{d\times m}$-valued process $Z$,  the following inequalities hold 
		
		(i) For $p\geq 2$ and $t\in [0,T]$,
		$$\E\Big[\Big|\int_{0}^{t}K(t-s)Y_s\d s\Big|^p\Big]\leq C\int_{0}^{t}K(t-s)\cdot \E[|Y_s|^p]\d s\,.$$
		
		(ii) For $p\geq 2$ and $t\in [0,T]$,
		$$\E\Big[\Big|\int_{0}^{t}K(t-s)Z_s\d W_s\Big|^p\Big]\leq C\int_{0}^{t}K(t-s)^2\cdot E[\|Z_s\|^p]\d s\,.$$
		
		(iii) For $p\geq 1, t\in [0,T]$ and $h\geq 0$ with $t+h\leq T$,
		$$\E\Big[\Big|\int_{0}^{t}(K(t+h-s)-K(t-s))Y_s\d s\Big|^p\Big]+\E\Big[\Big|\int_{t}^{t+h}K(t+h-s)Y_s\d s\Big|^p\Big]\leq C h^{(H+1/2)p}\sup_{t\in[0,T]}\E[|Y_t|^p]. $$
		
		(iv) For $p\geq 2, t\in [0,T]$ and $h\geq 0$ with $t+h\leq T$,
		$$\E\Big[\Big|\int_{0}^{t}(K(t+h-s)-K(t-s))Z_s\d W_s\Big|^p\Big]+E\Big[\Big|\int_{t}^{t+h}K(t+h-s)Z_s\d W_s\Big|^p\Big]\leq C h^{H p}\sup_{t\in[0,T]}\E[\|Z_t\|^p],$$
		where $C$ depends only on $K,p$ and $T$.
	\end{lem}
	We also need some estimates stronger than Lemma \ref{basic lemma}, $(iii)$ and $(iv)$.
	\begin{lem}\label{basic lemma-2}
		The following inequalities hold for any adapted $\mathbb{R}^{d}$-valued process $Y$ and $\mathbb{R}^{d\times m}$-valued process $Z$. 
		
		(i) For $p\geq 1, t\in [0,T]$ and $h\geq 0$ with $t+h\leq T$,
		\begin{align}
			&\E\Big[\Big|\int_{0}^{t}(K(t+h-s)-K(t-s))Y_s\d s\Big|^p\Big]+\E\Big[\Big|\int_{t}^{t+h}K(t+h-s)Y_s\d s\Big|^p\Big]\nonumber\\&\quad\leq C h^{(H+1/2)p}\Big[\sup_{s\in[0,t]}\E[|Y_s|^p]+\sup_{s\in[t,t+h]}\E[|Y_s|^p]\Big]\,.\nonumber
		\end{align}
		(ii)   For $p\geq 2, t\in [0,T]$ and $h\geq 0$ with $t+h\leq T$,
		\begin{align}
			&\E\Big[\Big|\int_{0}^{t}(K(t+h-s)-K(t-s))Z_s\d s\Big|^p\Big]+\E\Big[\Big|\int_{t}^{t+h}K(t+h-s)Z_s\d s\Big|^p\Big]\nonumber\\&\quad\leq C h^{Hp}\Big[\sup_{s\in[0,t]}\E[\|Z_s\|^p]+\sup_{s\in[t,t+h]}\E[\|Z_s\|^p]\Big],\nonumber
		\end{align}
		where $C$ depends only on $K,p$ and $T$.
	\end{lem}
	\begin{proof}
		For $(i)$, by Minkowski's integral inequality we see 
		\begin{align}
			\E\Big[&\Big|\int_{0}^{t}(K(t+h-s)-K(t-s))Y_s\d s\Big|^p\Big]+\E\Big[\Big|\int_{t}^{t+h}K(t+h-s)Y_s\d s\Big|^p\Big]\nonumber\\&\leq \Big(\int_{0}^{t}\Big|K(t+h-s)-K(t-s)\Big|\cdot\E[|Y_s|^p]^{\frac{1}{p}}\d s\Big)^{p}+\Big(\int_{t}^{t+h}K(t+h-s)\cdot\E[|Y_s|^p]^{\frac{1}{p}}\d s\Big)^{p}\nonumber\\&\leq \sup_{s\in[0,t]}E[|Y_s|^p]\Big(\int_{0}^{t}\Big|K(t+h-s)-K(t-s)\Big|\d s\Big)^{p}+\sup_{s\in[t,t+h]}E[|Y_s|^p]\Big(\int_{t}^{t+h}K(t+h-s)\d s\Big)^{p}\nonumber\\&\leq C h^{(H+1/2)p}\Big[\sup_{s\in[0,t]}\E[|Y_s|^p]+\sup_{s\in[t,t+h]}\E[|Y_s|^p]\Big].\nonumber
		\end{align}
		For $(ii)$, take $p\geq 2$. By the Burkholder-Davis-Gundy (BDG)  inequality and Minkowski's inequality, we obtain
		\begin{align}
			\E\Big[&\Big|\int_{0}^{t}(K(t+h-s)-K(t-s))Z_s\d W_s\Big|^p\Big]+\E\Big[\Big|\int_{t}^{t+h}K(t+h-s)Z_s\d W_s\Big|^p\Big]\nonumber\\&\leq C\E\Big(\int_{0}^{t}\big|K(t+h-s)-K(t-s)\big|^2\cdot\|Z_s\|^2\d s\Big)^{\frac{p}{2}}+C\E\Big(\int_{t}^{t+h}K(t+h-s)^2\cdot\|Z_s\|^2\d s\Big)^{\frac{p}{2}}\nonumber\\&\leq C\Big(\int_{0}^{t}\big|K(t+h-s)-K(t-s)\big|^2\cdot \E[\|Z_s\|^p]^{\frac{2}{p}}\d s\Big)^{\frac{p}{2}}+C\Big(\int_{t}^{t+h}K(t+h-s)^2\cdot\E[\|Z_s\|^p]^{\frac{2}{p}}\d s\Big)^{\frac{p}{2}}\nonumber\\&\leq C \sup_{s\in[0,t]}\E[\|Z_s\|^p]\Big(\int_{0}^{t}|K(t+h-s)-K(t-s)|^2\d s\Big)^{\frac{p}{2}}+C \sup_{s\in[t,t+h]}\E[\|Z_s\|^p]\Big(\int_{t}^{t+h}K(t+h-s)^2\d s\Big)^{\frac{p}{2}}\nonumber\\&\leq C h^{H p}\Big[\sup_{s\in[0,t]}\E[\|Z_s\|^p]+\sup_{s\in[t,t+h]}\E[\|Z_t\|^p]\Big].\nonumber
		\end{align}
		This proves the lemma. 
	\end{proof}
	Let us recall the integrability and the continuity  of the solution $X_t$ proved in \cite{ALP}.
	\begin{lem}[\cite{ALP}]\label{bound of X}  Let $p\geq 1$. 
		Under   Assumption \ref{assumption 2.2},  we have 
		$$\sup_{t\in[0,T]}\E[|X_t|^p]\leq C,$$
		where $C$ is a constant  depending only  on $|X_0|,K,L_1,p$ and $T$.
	\end{lem}
	\begin{lem}[\cite{ALP}]\label{continuous of X}
		Let $p>H^{-1}$. Then
		$$\E[|X_t-X_s|^p]\leq C|t-s|^{Hp}, \quad t,s\in [0,T]$$
		and $X$ admits a   $\mathrm{H\ddot{o}lder}$ continuous version  on $[0,T]$ of any order $\a<H $. Denoting this version again by $X$, one has
		$$\E\Big[\Big(\sup_{0\leq s\leq t\leq T}\frac{|X_t-X_s|}{|t-s|^{\a}}\Big)^p\Big]\leq C_{\a}$$
		for all $\a\in[0,H)$, where $C_{\a}$ is a constant. As a consequence, we can regard $X-X_0$ as a $\mathcal{C}^{\alpha}_{0}$ valued random variable for any $\alpha<H$.
	\end{lem}
	Here we give the boundedness of $X^n$.
	\begin{lem}\label{bound of X n}
		Assume the derivatives of the functions $b$ and $\sigma$ are bounded (hence Assumption \ref{assumption 2.2} holds), then   for any $p\geq 1$  there exists a constant $C$ such that
		$$\sup_n \sup_{t\in [0,T]}\E[|X^n_t|^p]\leq C,$$
		where $C$ is a constant that only depends on $|X_0|,K,L_1,p,T$ and the derivatives of $b$ and $\sigma$.
	\end{lem}
	\begin{proof}
		Let $\tau_m=\inf\{t\geq 0||X_t^n|\geq m\}\wedge T$ and observe that for $p>2$
		\begin{align}
			|X_t^n|^p\mathbb{I}_{\{t<\tau_m\}}&\leq \Big|X_0+\int_{0}^{t}K(t-s)\cdot\Big(b(X^n_{\frac{[ns]}{n}}\mathbb{I}_{\{s<\tau_m\}})+\nabla b(X^n_{\frac{[ns]}{n}}\mathbb{I}_{\{s<\tau_m\}})\cdot  M_s^{n}\mathbb{I}_{\{s<\tau_m\}}\Big)\d s\nonumber\\&\quad+\int_{0}^{t}K(t-s)\cdot\Big(\sigma(X^n_{\frac{[ns]}{n}}\mathbb{I}_{\{s<\tau_m\}})+\nabla \sigma(X^n_{\frac{[ns]}{n}}\mathbb{I}_{\{s<\tau_m\}})\cdot  M_s^{n}\mathbb{I}_{\{s<\tau_m\}}\Big)\d W_s\Big|^p\nonumber\\&\leq \Big|X_0+\int_{0}^{t}K(t-s)b(X^n_{\frac{[ns]}{n}}\mathbb{I}_{\{s<\tau_m\}})\d s+\int_{0}^{t}K(t-s)\sigma(X^n_{\frac{[ns]}{n}}\mathbb{I}_{\{s<\tau_m\}})\d W_s\nonumber\\&\quad+\int_{0}^{t}K(t-s)\nabla b(X^n_{\frac{[ns]}{n}}\mathbb{I}_{\{s<\tau_m\}})\int_{0}^{\frac{[ns]}{n}}\Big(K(s-r)-K(\frac{[ns]}{n}-r)\Big)\sigma(X^n_{\frac{[nr]}{n}}\mathbb{I}_{\{r<\tau_m\}})\d W_r\d s\nonumber\\&\quad+\int_{0}^{t}K(t-s)\nabla b(X^n_{\frac{[ns]}{n}}\mathbb{I}_{\{s<\tau_m\}})\int_{\frac{[ns]}{n}}^{s}K(s-r)\sigma(X^n_{\frac{[nr]}{n}}\mathbb{I}_{\{r<\tau_m\}})\d W_r\d s\nonumber\\&\quad+\int_{0}^{t}K(t-s)\nabla \sigma(X^n_{\frac{[ns]}{n}}\mathbb{I}_{\{s<\tau_m\}})\int_{0}^{\frac{[ns]}{n}}\Big(K(s-r)-K(\frac{[ns]}{n}-r)\Big)\sigma(X^n_{\frac{[nr]}{n}}\mathbb{I}_{\{r<\tau_m\}})\d W_r\d W_s\nonumber\\&\quad+\int_{0}^{t}K(t-s)\nabla \sigma(X^n_{\frac{[ns]}{n}}\mathbb{I}_{\{s<\tau_m\}})\int_{\frac{[ns]}{n}}^{s}K(s-r)\sigma(X^n_{\frac{[nr]}{n}}\mathbb{I}_{\{r<\tau_m\}})\d W_r\d W_s\Big|^p.\nonumber
		\end{align}
		Taking the expectation on both sides we have
		\begin{align}
			&\E[|X_t^n|^p\mathbb{I}_{\{t<\tau_m\}}]\nonumber\\&\leq 7^{p-1}|X_0|^p+7^{p-1}\E\Big[\Big|\int_{0}^{t}K(t-s)b(X^n_{\frac{[ns]}{n}}\mathbb{I}_{\{s<\tau_m\}})\d s\Big|^p\Big]+7^{p-1}\E\Big[\Big|\int_{0}^{t}K(t-s)\sigma(X^n_{\frac{[ns]}{n}}\mathbb{I}_{\{s<\tau_m\}})\d W_s\Big|^p\Big]\nonumber\\ &\quad+7^{p-1}\E\Big[\Big|\int_{0}^{t}K(t-s)\nabla b(X^n_{\frac{[ns]}{n}}\mathbb{I}_{\{s<\tau_m\}})\int_{0}^{\frac{[ns]}{n}}\Big(K(s-r)-K(\frac{[ns]}{n}-r)\Big)\sigma(X^n_{\frac{[nr]}{n}}\mathbb{I}_{\{r<\tau_m\}})\d W_r\d s\Big|^p\Big]\nonumber\\&\quad+7^{p-1}\E\Big[\Big|\int_{0}^{t}K(t-s)\nabla b(X^n_{\frac{[ns]}{n}}\mathbb{I}_{\{s<\tau_m\}})\int_{\frac{[ns]}{n}}^{s}K(s-r)\sigma(X^n_{\frac{[nr]}{n}}\mathbb{I}_{\{r<\tau_m\}})\d W_r\d s\Big|^p\Big]\nonumber\\&\quad+7^{p-1}\E\Big[\Big|\int_{0}^{t}K(t-s)\nabla \sigma(X^n_{\frac{[ns]}{n}}\mathbb{I}_{\{s<\tau_m\}})\int_{0}^{\frac{[ns]}{n}}\Big(K(s-r)-K(\frac{[ns]}{n}-r)\Big)\sigma(X^n_{\frac{[nr]}{n}}\mathbb{I}_{\{r<\tau_m\}})\d W_r\d W_s\Big|^p\Big]\nonumber\\&\quad+7^{p-1}\E\Big[\Big|\int_{0}^{t}K(t-s)\nabla \sigma(X^n_{\frac{[ns]}{n}}\mathbb{I}_{\{s<\tau_m\}})\int_{\frac{[ns]}{n}}^{s}K(s-r)\sigma(X^n_{\frac{[nr]}{n}}\mathbb{I}_{\{r<\tau_m\}})\d W_r\d W_s\Big|^p\Big].\nonumber
		\end{align}
		Since the derivatives of $b$ and $\sigma$ are bounded,  from Lemma \ref{basic lemma}, $(i)$ and $(ii)$, Lemma \ref{basic lemma-2}, Fubini's theorem 
		it follows  that
		\begin{align}
			\E[|X_t^n|^p\mathbb{I}_{\{t<\tau_m\}}]&\leq 7^{p-1}|X_0|^p+C\int_{0}^{t}K(t-s)\E[|b(X^n_{\frac{[ns]}{n}}\mathbb{I}_{\{s<\tau_m\}})|^p]\d s\nonumber\\&\quad+C\int_{0}^{t}K(t-s)^2\E[|\sigma(X^n_{\frac{[ns]}{n}}\mathbb{I}_{\{s<\tau_m\}})|^p]\d s\nonumber\\&\quad+Cn^{-(H+1/2)p}\int_{0}^{t}K(t-s)\sup_{r\in[0,s]}\E[|\sigma(X^n_{\frac{[nr]}{n}}\mathbb{I}_{\{r<\tau_m\}})|^p]\d s\nonumber\\&\quad+Cn^{-Hp}\int_{0}^{t}K(t-s)^2\sup_{r\in[0,s]}\E[|\sigma(X^n_{\frac{[nr]}{n}}\mathbb{I}_{\{r<\tau_m\}})|^p]\d s\nonumber\\&\leq C_1+C_2\int_{0}^{t}\hat{K}(t-s)\sup_{r\in[0,s]}\E[|X^n_r\mathbb{I}_{\{r<\tau_m\}}|^p]\d s\nonumber,
		\end{align}
		where $\hat{K}(t-s):=C\big(K(t-s)+K(t-s)^2\big)$. Putting $f_m(t)=\sup_{r\in[0,t]}\E[|X^n_r\mathbb{I}_{\{r<\tau_m\}}|^p]$, we have that
		$$f_m(t)\leq C_1+\int_{0}^{t}\hat{K}(t-s) f_m(s)\d s.$$
		Applying Volterra type Gronwall's inequality (\cite[Theorem 2.2]{ZX2}) we obtain
		$$f_m(t)\leq C_1+\int_{0}^{t}C_1\hat{R}(t-s)(t-s) \d s<\infty,$$
		where $\hat{R}(t-s)$ denotes the resolvent of $\hat{K}$ and the last inequality is based on the boundedness of $\hat{R}(t-s)(t-s)$ (\cite[Lemma 2.1]{ZX2}).
		
		By Fatou's lemma, we have
		\begin{align}
			\E[|X^n_t|^p]\leq \E\Big[\liminf_{m \rightarrow \infty}|X_t^n|^p\mathbb{I}_{\{t<\tau_m\}}\Big]\leq \liminf_{m \rightarrow \infty}\E\big[|X_t^n|^p\mathbb{I}_{\{t<\tau_m\}}\big]\leq \liminf_{m \rightarrow \infty}f_m(t)<\infty,\nonumber
		\end{align}
		for all $t\in [0,T]$, which completes the proof.
	\end{proof}
	By above lemma, we further give the boundness of $M^n$.
	\begin{lem}\label{bound of M n}
		Assume the derivatives of the functions $b$ and $\sigma$ are bounded (hence Assumption \ref{assumption 2.2} holds), then for any $p\geq 1$  there exists a constant $C$ such that
		$$\sup_{t\in [0,T]}\E[|M^n_t|^p]\leq Cn^{-Hp},$$
		where $C$ is a constant that only depends on $|X_0|,K,L_1,p,T$ and the derivatives of $b$ and $\sigma$.
	\end{lem}
	\begin{proof}
		By Lemma \ref{basic lemma}-(ii) and (iv),   Assumption \ref{assumption 2.2} and Lemma \ref{bound of X n}, we have for $p\geq2$
		\begin{align}
			\E[|M_t^n|^p]&\leq \E\Big[\Big|\int_{0}^{\frac{[nt]}{n}}\Big(K(t-r)-K(\frac{[nt]}{n}-r)\Big)\sigma(X^n_{\frac{[nr]}{n}})\d W_r\Big|^p\Big]\nonumber\\&\quad+\E\Big[\Big|\int_{\frac{[nt]}{n}}^{t}K(t-r)\sigma(X^n_{\frac{[nr]}{n}})\d W_r\Big|^p\Big]\nonumber\\&\leq Cn^{-Hp}\sup_{t\in [0,T]}\E[|\sigma(X^n_{\frac{[nt]}{n}})|^p]\leq Cn^{-Hp}.\nonumber
		\end{align}
		This gives the lemma. 
%		{\red 	Lyapunov inequality gives desired result.} \footnote{where to use 
%			Lyapunov inequality? It looks like the lemma  has been proved?}
	\end{proof}
	We now state the continuity of $X^n$.
	\begin{lem}\label{continuous of X n}
		Let $p>H^{-1}$. Then
		$$\sup_n \E[|X^n_t-X^n_s|^p]\leq C|t-s|^{H p}, \quad t,s\in [0,T]$$
		and $X^n$ admits a   $\mathrm{H\ddot{o}lder}$ continuous version   of any order $\a<H$
		on $[0,T]$. Denoting this version again by $X^n$, one has
		$$\E\Big[\Big(\sup_{0\leq s\leq t\leq T}\frac{|X^n_t-X^n_s|}{|t-s|^{\a}}\Big)^p\Big]\leq C_{\a}$$
		for all $\a\in[0,H)$, where $C_{\a}$ is a constant. As a consequence, we can think of  $X^n$ as a $\mathcal{C}^{\alpha}_{0}$
	%	\footnote{what is $\mathcal{C}^{\alpha}_{0}$} 
	valued random variable for any $\alpha<H$.
	\end{lem}

	\begin{proof}
		For any $0\leq s<t\leq T$, we first rewrite $X_t^n-X_s^n$ as
		\begin{align}
			&X_t^n-X_s^n\nonumber\\&=\int_{0}^{s}\big(K(t-u)-K(s-u)\big)\cdot\Big(b(X^n_{\frac{[nu]}{n}})+\nabla b(X^n_{\frac{[nu]}{n}})\cdot  M_u^{n}\Big)\d u\nonumber\\&\quad+\int_{0}^{s}\big(K(t-u)-K(s-u)\big)\cdot\Big(\sigma(X^n_{\frac{[nu]}{n}})+\nabla \sigma(X^n_{\frac{[nu]}{n}})\cdot  M_u^{n}\Big)\d W_u\nonumber\\&\quad+\int_{s}^{t}K(t-u)\cdot\Big(b(X^n_{\frac{[nu]}{n}})+\nabla b(X^n_{\frac{[nu]}{n}})\cdot  M_u^{n}\Big)\d u\nonumber\\&\quad+\int_{s}^{t}K(t-u)\cdot\Big(\sigma(X^n_{\frac{[nu]}{n}})+\nabla \sigma(X^n_{\frac{[nu]}{n}})\cdot  M_u^{n}\Big)\d W_u\nonumber\\&:=I^n_1(t,s)+I^n_2(t,s)+I^n_3(t,s)+I^n_4(t,s).\nonumber
		\end{align}
		%For the terms $I^n_1(t,s)$ and $I^n_3(t,s)$, b
		By Lemma \ref{basic lemma}, (iii), the 
		boundedness of the derivatives of the functions $b$ and $\sigma$, Lemmas \ref{bound of X n} and   \ref{bound of M n}, we have
		\begin{align}
			&\E[|I^n_1(t,s)|^p+|I^n_3(t,s)|^p]\nonumber\\&\leq C |t-s|^{(H+1/2)p}\sup_{t\in [0,T]}\E\Big[\Big|b(X^n_{\frac{[nt]}{n}})+\nabla b(X^n_{\frac{[nt]}{n}})\cdot  M_t^{n}\Big|^p\Big]\nonumber\\&\leq C|t-s|^{(H+1/2)p}\sup_{t\in [0,T]}\E\Big[b(X^n_{\frac{[nt]}{n}})+M_t^{n}\Big]\leq C|t-s|^{Hp}.\nonumber
		\end{align}
		Similarly, %for the terms $I^n_2(t,s)$ and $I^n_4(t,s)$, 
		by Lemma \ref{basic lemma}, (iv),  the 
		boundedness of the derivatives of the functions $b$ and $\sigma$,  Lemmas \ref{bound of X n} and   \ref{bound of M n}, we have
		\begin{align}
			&\E[|I^n_2(t,s)|^p+|I^n_4(t,s)|^p]\nonumber\\&\leq C |t-s|^{Hp}\sup_{t\in [0,T]}\E\Big[\sigma(X^n_{\frac{[nt]}{n}})+\nabla b(X^n_{\frac{[nt]}{n}})\cdot  M_u^{n}\Big]\nonumber\\&\leq C|t-s|^{Hp}\sup_{t\in [0,T]}\E\Big[\sigma(X^n_{\frac{[nt]}{n}})+M_t^{n}\Big]\leq C|t-s|^{Hp}.\nonumber
		\end{align}
		Finally, the proof is complete by Kolmogorov's continuity criterion (e.g. Revuz and Yor \cite[Theorem I.2.1 ]{RY}).
	\end{proof}
	We can now state the convergence rate of Milstein scheme for $X-X^n$. The reader is referred to, for instance, \cite[Theorem 2.9]{LHH2}.
	\begin{lem}[\cite{LHH2}]\label{rate X-X n}
		Assume the derivatives  of the functions $b$ and $\sigma$ up to second orders   are bounded (hence Assumptions \ref{assumption 2.2} and \ref{assumption 2.3} hold), then for any $p\geq 1$
		$$\sup_{t\in [0,T]}\E[|X_t-X^n_t|^p]\leq C n^{-2H p},$$
		where $C$ is a positive constant which does not depend on $n$.
	\end{lem}
	From these estimates, an application of the Garsia-Rodemich-Rumsey inequality gives the following result as in Section $4.3.2$ of Richard et al. \cite{RTY}.
	\begin{lem}\label{strong X-X n}
		Assume the derivatives of the functions $b$ and $\sigma$ are bounded (hence Assumption \ref{assumption 2.2} holds).  For all $p\geq 1$  and any $\epsilon>0$,   there exists a constant $C>0$ which does not depend on $n$ such that
		$$\E\Big[\sup_{t\in [0,T]}|X_t-X_t^n|^p\Big]\leq Cn^{-p(2H-\epsilon)}.$$
	\end{lem}
	The following lemma plays a  key role in  finding   the limit of   fractional integrals appeared in our future computations. One can refer to \cite[Lemma C.2]{FU}.
	\begin{lem}[\cite{FU}] \label{limit distribution of stochastic integral}
		Assume that $g\in L^2(0,1)$. Let $H^{(n)}$ and $H$ be stochastic processes on $[0,T]$ such that
		$$\E\Big[\int_{0}^{T}|H^{(n)}_s-H_s|^2\d s\Big]\to 0$$
		with $H$ being almost surely continuous. Then for all $t\in [0,T]$,
		$$\int_{0}^{t}H^{(n)}_s g(ns-[ns])\d s\xrightarrow[\text { in } L^2]{n \rightarrow \infty} \int_0^1 g(r) \mathrm{d} r \int_0^t H_s \mathrm{d} s.$$
	\end{lem}
	Now let us introduce some notations:
	\begin{align}\label{basic types}
		\delta_{(n,u)}=u-\frac{[nu]}{n},\quad  G=\Gamma(H+1/2)^2, 
		\quad \mu(r,y)=(r+y)^{H-1/2}-r^{H-1/2}.
	\end{align}
	The following bounds and limits will also used in this sequel.
	\begin{lem}[\cite{FU}] \label{small time estimate}
		For $s\in [0,T]$, we have
		$$(i)\quad n^{2H}\int_{\frac{[ns]}{n}}^{s}\Big|K(s-u)\Big|^2\d u=\frac{1}{2HG}\delta_{(n,s)}^{2H}\leq C,$$
		and
		$$(ii)\quad n^{2H}\int_{0}^{\frac{[ns]}{n}}\Big|K(s-u)-K(\frac{[ns]}{n}-u)\Big|^2\d u\leq C,$$
		where $C$ does not depend on $n$.
	\end{lem}
	\begin{lem}\label{2-small time estimate}
		For $s\in [0,T]$, we have
		$$\quad n^{2H}\int_{\frac{[ns]}{n}}^{s}\Big(K(s-u)-K(\frac{[ns]}{n}-u)\Big)^2\d u\leq C,$$
		where $C$ does not depend on $n$.
	\end{lem}
	\begin{proof}
	First,  we make the change of variables $z=\frac{[ns]}{n}-u$ and then $r=z/\delta_{(n,s)}$. Since the integral $\int_{0}^{1}|\mu(r,1)|^2\d r<\infty$ by \cite{Mis} Theorem $1.3.1$ and Lemma $A.0.1$,    we have 
		\begin{align}\label{transfer of K-2}
			&n^{2H}\int_{\frac{[ns]}{n}}^{s}\Big(K(s-u)-K(\frac{[ns]}{n}-u)\Big)^2\d u\nonumber\\&\quad=\frac{n^{2H}}{G}\int_{0}^{\delta_{(n,s)}}\Big|\mu(z,\delta_{(n,s)})\Big|^2\d z=\frac{n^{2H}\delta_{(n,s)}^{2H}}{G}\int_{0}^{1}|\mu(r,1)|^2\d r\leq C,
		\end{align}
		where $C$ does not depend on $n$.
	\end{proof}
	\begin{lem}\label{small time estimate 2}
	For $u\in [0,T]$, we have
	$$\lim_{n\to\infty}\int_{-\infty}^{-[nu]}\Big((z+[nu]-nu)^{H-1/2}-([z]+[nu]-nu)^{H-1/2}\Big)^2\d z=0.$$
\end{lem}
\begin{proof}
	Since 
	$$z-1\leq [z]\leq z,\quad -1\leq [nu]-nu\leq 0, $$
	we have
	\begin{align}
		(z+[nu]-nu)^{H-1/2}-([z]+[nu]-nu)^{H-1/2}\leq 0\nonumber
	\end{align}
	and
	\begin{align}\label{small estimate}
		&\Big((z+[nu]-nu)^{H-1/2}-([z]+[nu]-nu)^{H-1/2}\Big)^2\nonumber\\&\quad\leq \Big((z-1)^{H-1/2}-z^{H-1/2}\Big)^2\nonumber\\&\quad=\big(\frac{1}{H-1/2}\big)^2\Big(\int_{0}^{1}(z-u)^{H-3/2}\d u\Big)^2\leq \big(\frac{1}{H-1/2}\big)^2 (z-1)^{2H-3}.
	\end{align}
	Therefore, we obtain 
	\begin{align}
		0&\leq \int_{-\infty}^{-[nu]}\Big((z+[nu]-nu)^{H-1/2}-([z]+[nu]-nu)^{H-1/2}\Big)^2\d z\nonumber\\&\leq \int_{-\infty}^{-[nu]} \big(\frac{1}{H-1/2}\big)^2 (z-1)^{2H-3}\d z \to 0\quad \text{as $n\to\infty$}.\nonumber
	\end{align}
	Applying the dominated convergence theorem yields  the desired result.
\end{proof}
	\begin{lem}\label{3-small time estimate}
		For $u\in [0,T]$, we have
		$$\quad n^{2H}\int_{0}^{u}\Big(K(s-u)-K(\frac{[ns]}{n}-u)\Big)^2\d s\leq C,$$
		where $C$ is independent of    $n$.
	\end{lem}
	\begin{proof}
		Using the change of variable, $z=ns-[nu]$,  we see 
		\begin{align}\label{tranfer of K}
			\int_{0}^{u}\Big(K(s-u)-K(\frac{[ns]}{n}-u)\Big)^2\d s=\frac{n^{-2H}}{G}\int_{-[nu]}^{nu-[nu]}\nu(z,nu-[nu])^2\d z,
		\end{align}
		where
		\begin{align}\label{def of nu}
			\nu(z,r):=(z-r)^{H-1/2}-([z]-r)^{H-1/2}.
		\end{align}
		By \eqref{small estimate} the integral $\int_{-\infty}^{\infty }\nu(z,nu-[nu])^2\d z$ is finite, proving the lemma. 
	\end{proof}
	\begin{lem}\cite{FU}\label{esti of A}
		For $v\leq s$, let
		$$A_n(v,s)=n^{2H}\int_{0}^{\frac{[nv]}{n}}\Big(K(s-u)-K(\frac{[ns]}{n}-u)\Big)\Big(K(v-u)-K(\frac{[nv]}{n}-u)\Big)\d u.$$  
		Then $\sup_{0\leq v\leq s\leq T}\sup_{n}|A_n(v,s)|<\infty$ and for $v<s$, $\lim_{n\to\infty} A_n(v,s)=0.$
	\end{lem}	
	\begin{lem}\label{esti of B}
		For $\frac{[ns]}{n}\leq v\leq s$, let
		$$B_n(v,s)=n^{2H}\int_{\frac{[ns]}{n}}^{v}K(s-u)K(v-u)\d u.$$  
		Then $\sup_{0\leq \frac{[ns]}{n}\leq v\leq s\leq T}\sup_{n}|B_n(v,s)|<\infty$ and for $v<s$, $\lim_{n\to\infty} B_n(v,s)=0.$
	\end{lem}	
	\begin{proof}
		By the Cauchy-Schwarz inequality and Lemma \ref{small time estimate} we have
		\begin{align}
			B_n(v,s)&\leq n^{2H}\Big(\int_{\frac{[ns]}{n}}^{v}K(s-u)^2\d u\Big)^{1/2}\Big(\int_{\frac{[ns]}{n}}^{v}K(v-u)^2\d u\Big)^{1/2}\nonumber\\&\leq n^{2H}\Big(\int_{\frac{[ns]}{n}}^{s}K(s-u)^2\d u\Big)^{1/2}\Big(\int_{\frac{[ns]}{n}}^{v}K(v-u)^2\d u\Big)^{1/2}\leq C.\nonumber
		\end{align}
		Especially, for $v<s$, by the change of variable $v-u=\frac{z}{n}$, we have
		\begin{align}
			0\leq B_n(v,s)&=\frac{1}{G}\int_{0}^{nv-[ns]}(ns-nv+z)^{H-1/2}z^{H-1/2}\d z\nonumber\\&\leq  C\int_{0}^{1}(ns-nv+z)^{H-1/2}z^{H-1/2}\d z\to 0,\nonumber
		\end{align}
		because of $(ns-nv+z)^{H-1/2}\to 0.$
	\end{proof}
	For the proof of the next two lemmas, one can find in   \cite[pp 5086-5087 and pp 5089-5090]{FU}.
	\begin{lem}\cite{FU}\label{esti of D 1}
		For any $i,i'=1,\cdots,m;l,l'=1,\cdots,d$
		$$  \lim_{n\to\infty}\sup_{t\in[0,T]}\Big\|\int_{0}^{t}D^{i, i',l,l'}_{1,u}\d u\Big\|_{L^2}=0,$$
		where $D^{i, i',l,l'}_{1,u}$ is defined in \eqref{def of D 1}.
	\end{lem}	
	\begin{lem}\cite{FU}\label{esti of D 2}
		For any $i,i'=1,\cdots,m;l,l'=1,\cdots,d$
		$$  \lim_{n\to\infty}\sup_{u\in[0,T]}\|\int_{0}^{t}D^{i, i',l,l'}_{2,u}\d u\|_{L^2}=0,$$
		where $D^{i ,i',l,l'}_{2,u}$ is  defined in \eqref{def of D 2}.
	\end{lem}

	\section{Proof of the main result}\label{sec 3} 
	After the preparation we are going to present the proof of  the theorem, which  is sophisticated and we divide the proof into several subsections.  
	\subsection{Decomposition of  normalized error process $U^n$}
	Combining \eqref{Volterra eq} and \eqref{Milstein scheme}, we can rewrite $U_{\cdot}^{n}:=n^{2H}(X_{\cdot}-X^n_{\cdot})=\{U^{n,1},\cdots,U^{n,d}\}^{\top}$ as
	\begin{align}\label{dcp of U n}
		U_{t}^n&=n^{2H}\Big[\int_{0}^{t}K(t-s)b(X_{s})\d s-\int_{0}^{t}K(t-s)\cdot\Big(b(X^n_{\frac{[ns]}{n}})+\nabla b(X^n_{\frac{[ns]}{n}})\cdot  M_s^{n}\Big)\d s\nonumber\\&\quad+\int_{0}^{t}K(t-s)\sigma(X_{s})\d W_s-\int_{0}^{t}K(t-s)\cdot\Big(\sigma(X^n_{\frac{[ns]}{n}})+\nabla \sigma(X^n_{\frac{[ns]}{n}})\cdot  M_s^{n}\Big)\d W_s\Big]\nonumber\\&=\int_{0}^{t}K(t-s)\nabla b(X^n_{s})U^n_s\d s+\int_{0}^{t}K(t-s)\nabla \sigma(X^n_{s})U^n_s\d W_s+\mathcal{R}^{n,1}_{b}+\mathcal{R}^{n,1}_{\sigma}\nonumber\\&\quad+n^{2H}\int_{0}^{t}K(t-s)\cdot\Big(b(X^n_s)-b(X^n_{\frac{[ns]}{n}})-\nabla b(X^n_{\frac{[ns]}{n}})\cdot  M_s^{n}\Big)\d s\nonumber\\&\quad+n^{2H}\int_{0}^{t}K(t-s)\cdot\Big(\sigma(X^n_s)-\sigma(X^n_{\frac{[ns]}{n}})-\nabla \sigma(X^n_{\frac{[ns]}{n}})\cdot  M_s^{n}\Big)\d W_s\nonumber\\&=\int_{0}^{t}K(t-s)\nabla b(X^n_{s})U^n_s\d s+\int_{0}^{t}K(t-s)\nabla \sigma(X^n_{s})U^n_s\d W_s+\mathcal{R}^{n,1}_{b}+\mathcal{R}^{n,1}_{\sigma}\nonumber\\&\quad+n^{2H}\int_{0}^{t}K(t-s)\cdot\nabla b(X^n_{\frac{[ns]}{n}})\Big(X^n_s-X^n_{\frac{[ns]}{n}}-M_s^{n}\Big)\d s+\mathcal{R}^{n,2}_{b}\nonumber\\&\quad+n^{2H}\int_{0}^{t}K(t-s)\cdot\nabla \sigma(X^n_{\frac{[ns]}{n}})\Big(X^n_s-X^n_{\frac{[ns]}{n}}-M_s^{n}\Big)\d W_s+\mathcal{R}^{n,2}_{\sigma}\nonumber\\&=\int_{0}^{t}K(t-s)\nabla b(X^n_{s})U^n_s\d s+\int_{0}^{t}K(t-s)\nabla \sigma(X^n_{s})U^n_s\d W_s\nonumber\\&\quad+n^{2H}\int_{0}^{t}K(t-s)\cdot\nabla b(X^n_{\frac{[ns]}{n}})\Big(X^n_s-X^n_{\frac{[ns]}{n}}-M_s^{n}\Big)\d s\nonumber\\&\quad+\int_{0}^{t}K(t-s)\cdot\nabla \sigma(X^n_{\frac{[ns]}{n}})\d V^n_s+\mathcal{R}^{n},
	\end{align}
	where $\mathcal{R}^{n,1}_{b}$, $\mathcal{R}^{n,2}_{b}$, $\mathcal{R}^{n,1}_{\sigma}$ and $\mathcal{R}^{n,2}_{\sigma}$ 
	%\footnote{need to write them down} 
	are the Taylor's remainder terms, given in more details by \eqref{e.r.1}-\eqref{e.r.4} when we take care of them,  and $V^n$ and $\mathcal{R}^{n}$ are defined by
	\begin{align}
		V^n_t&:=n^{2H}\int_{0}^{t}\big(X^n_s-X^n_{\frac{[ns]}{n}}-M_s^{n}\big)\d W_s,\nonumber\\\mathcal{R}^{n}&:=\mathcal{R}^{n,1}_{b}+\mathcal{R}^{n,2}_{b}+\mathcal{R}^{n,1}_{\sigma}+\mathcal{R}^{n,2}_{\sigma}.\nonumber
	\end{align}
	\subsection{The limit of the quadratic variation of $V^n$}
	Denote
	\begin{align} 
		M_u^{1,n,l}:=&\sum_{i=1}^{m}\int_{0}^{\frac{[nu]}{n}}\Big(K(u-v)-K(\frac{[nu]}{n}-v)\Big)\sigma^l_i(X^n_{\frac{[nv]}{n}})\d W^i_v\label{M u 1}\,,  \\ 
		M_u^{2,n,l}:=&\sum_{i=1}^{m}\sigma^l_i(X^n_{\frac{[nu]}{n}})\int_{\frac{[nu]}{n}}^{u}K(u-v)\d W^i_v \label{M u 2}\,, 
	\end{align} 
	and 
	\begin{align}\label{def of M u}
		M_u^{n,l}:=M_u^{1,n,l}+M_u^{2,n,l}\,. 
	\end{align} 
	We write  $V^n_{\cdot}=\{V^{n,k,j}\}$    as
	\begin{align}\label{def of V}
		V^{n,k,j}_{\cdot}=n^{2H}\int_{0}^{\cdot}\big[X^n_s-X^n_{\frac{[ns]}{n}}-M_s^{n}\big]^{k}\d W_s^j, \quad 1\leq k\leq d, 1\leq j\leq m.
	\end{align}
	
	We first find  the limit of covariation $\langle V^{n,k_1,j}, V^{n,k_2,j} \rangle_t$. Recall the expression of the solution $X^n$ of \eqref{Milstein scheme}, we can rewrite $\big[X^n_s-X^n_{\frac{[ns]}{n}}-M_s^{n}\big]^{k}$ as follows:
	\begin{align}\label{different parts of Z}
		&\big[X^n_s-X^n_{\frac{[ns]}{n}}-M_s^{n}\big]^{k}\nonumber\\
		&=\int_{0}^{\frac{[ns]}{n}}\Big(K(s-u)-K(\frac{[ns]}{n}-u)\Big)b^k(X^n_{\frac{[nu]}{n}})\d u+b^k(X^n_\frac{[ns]}{n})\int_{\frac{[ns]}{n}}^{s}K(s-u)\d u\nonumber\\
		&\quad+\sum_{l=1}^{d}\int_{0}^{\frac{[ns]}{n}}\Big(K(s-u)-K(\frac{[ns]}{n}-u)\Big)\partial_l b^k(X^n_{\frac{[nu]}{n}})M_u^{n,l}\d u\nonumber\\
		&\quad+ \sum_{l=1}^{d}\partial_l b^k(X^n_\frac{[ns]}{n})\int_{\frac{[ns]}{n}}^{s}K(s-u)M_u^{n,l}\d u\nonumber\\
		&\quad+\sum_{j=1}^{m}\sum_{l=1}^{d}\int_{0}^{\frac{[ns]}{n}}\Big(K(s-u)-K(\frac{[ns]}{n}-u)\Big)\partial_l \sigma^k_j(X^n_{\frac{[nu]}{n}})M_u^{n,l}\d W_u^j\nonumber\\
		&\quad+\sum_{j=1}^{m}\sum_{l=1}^{d}\partial_l \sigma^{k}_j(X^n_\frac{[ns]}{n})\int_{\frac{[ns]}{n}}^{s}K(s-u)M_u^{n,l}\d W_u^j\nonumber\\
		& :=\mathcal{A}^{n,k}_{1,s}+\mathcal{A}^{n,k}_{2,s},
	\end{align}
	where $\mathcal{A}^{n,k}_{1,s}$ and $\mathcal{A}^{n,k}_{2,s}$ denote the sums of the first  four and the last  two terms, respectively.     
	We have
	\begin{align}\label{cov of V}
		&\langle V^{n,k_1,j}, V^{n,k_2,j} \rangle_t\nonumber\\
		&=n^{4H}\int_{0}^{t}\big[X^n_s-X^n_{\frac{[ns]}{n}}-M_s^{n}\big]^{k_1}\big[X^n_s-X^n_{\frac{[ns]}{n}}-M_s^{n}\big]^{k_2}\d s\nonumber\\
		&=      \int_{0}^{t}n^{4H}\mathcal{A}^{n,k_1}_{1,s}\mathcal{A}^{n,k_2}_{1,s}\d s 
		+ \int_{0}^{t}n^{4H}\mathcal{A}^{n,k_1}_{1,s}\mathcal{A}^{n,k_2}_{2,s}\d s\nonumber\\
		&\quad+\int_{0}^{t}n^{4H}\mathcal{A}^{n,k_1}_{2,s}\mathcal{A}^{n,k_2}_{1,s}\d s+\int_{0}^{t}n^{4H}\mathcal{A}^{n,k_1}_{2,s}\mathcal{A}^{n,k_2}_{2,s}\d s.
	\end{align} 
	\begin{lem}\label{est of A-1,2}
		For any $k=1,\cdots,d,$
		\begin{itemize}
			\item[(i)] 
			$ \lim_{n\to\infty}\sup_{s\in[0,T]}\|n^{2H}\mathcal{A}^{n,k}_{1,s}\|_{L^2}=0 $.
			\item[(ii)] $\sup_{n \ge 1} \sup_{s\in[0,T]}\|n^{2H}\mathcal{A}^{n,k}_{2,s}\|_{L^2}<\infty.$ 
			\item[(iii)]  For any $k_1,k_2$ and $t\in [0,T]$  we have as ${n \rightarrow \infty}$
			\begin{align}
				n^{4H}&\int_{0}^{t}\mathcal{A}^{n,k_1}_{2,s}\mathcal{A}^{n,k_2}_{2,s}ds
				\stackrel{L^2}{\rightarrow }  C_M^2\sum_{i,j=1}^{m}\sum_{l,l'=1}^{d}\int_{0}^{t}\sigma^l_i(X_{s})\sigma^{l'}_{i}(X_{s})\partial_l\sigma^{k_1}_j(X_{s})\partial_{l'}\sigma^{k_2}_{j}(X_{s})\d s. \nonumber
			\end{align}
		\end{itemize} 	
		%	where
		%	$$C_M:=C^{(1)}_G+C^{(2)}_G+C^{(3)}_G+C^{(4)}_G.$$
	\end{lem}
	\begin{proof}
		%\textbf{Proof of Claim (i)}. \quad 
		Since the derivatives of $b$ and $\sigma$ are bounded, by   Assumption  \ref{assumption 2.2}, Lemmas \ref{basic lemma}, (iii),   \ref{bound of X n} and   \ref{bound of M n}, we have
		\begin{align}\label{est of A-1}
			&\E[|n^{2H}\mathcal{A}^{n,k}_{1,s}|^2]\nonumber\\&\leq n^{4H}\E\Big[\Big|\int_{0}^{\frac{[ns]}{n}}\Big(K(s-u)-K(\frac{[ns]}{n}-u)\Big)b^k(X^n_{\frac{[nu]}{n}})\d u+b^k(X^n_\frac{[ns]}{n})\int_{\frac{[ns]}{n}}^{s}K(s-u)\d u\nonumber\\&\quad+\sum_{l=1}^{d}\int_{0}^{\frac{[ns]}{n}}\Big(K(s-u)-K(\frac{[ns]}{n}-u)\Big)\partial_l b^k(X^n_{\frac{[nu]}{n}})M_u^{n,l}\d u\nonumber\\&\quad+\sum_{l=1}^{d}\partial_l b^k(X^n_\frac{[ns]}{n})\int_{\frac{[ns]}{n}}^{s}K(s-u)M_u^{n,l}\d u\Big|^2\Big]\nonumber\\&\leq Cn^{4H}n^{-2(H+1/2)}\sup_{s\in [0,T]}\E\Big[\big|b^k(X_{\frac{[ns]}{n}}^n)+\partial_l b^k(X^n_{\frac{[ns]}{n}})M_s^{n,l}\big|^2\Big]\nonumber\\&\leq Cn^{2H-1}\to 0\quad \text{as $n\to\infty$}.
		\end{align}
		Hence $\lim_{n\to\infty}\sup_{s\in[0,T]}\|n^{2H}\mathcal{A}^{n,k}_{1,s}\|_{L^2}=0.$
		
		%\textbf{Proof of Claim (ii)}.\quad  Since the derivatives of $b$ and $\sigma$ are bounded, by the Assumption \ref{assumption 2.1},   Assumption \ref{assumption 2.2}, Lemma \ref{basic lemma}-(iv) and Lemma \ref{bound of M n}, 
		Similarly,   we have
		\begin{align}\label{est of A-2}
			&\E[|n^{2H}\mathcal{A}^{n,k}_{2,s}|^2]\nonumber\\&\leq n^{4H}\E\Big[\Big|\sum_{j=1}^{m}\sum_{l=1}^{d}\int_{0}^{\frac{[ns]}{n}}\Big(K(s-u)-K(\frac{[ns]}{n}-u)\Big)\partial_l \sigma^k_j(X^n_{\frac{[nu]}{n}})M_u^{n,l}\d W_u^j\nonumber\\&\quad+\sum_{j=1}^{m}\sum_{l=1}^{d}\partial_l \sigma^k_j(X^n_\frac{[ns]}{n})\int_{\frac{[ns]}{n}}^{s}K(s-u)M_u^{n,l}\d W_u^j\Big|^2\Big]\nonumber\\&\leq Cn^{4H}n^{-2H}\sup_{s\in [0,T]}\E\Big[|\partial_l \sigma^k_j(X^n_{\frac{[ns]}{n}})M_s^{n,l}|^2\Big]\nonumber\\&\leq C\quad \text{as $n\to\infty$}.
		\end{align}
		Hence $\lim_{n\to\infty}\sup_{s\in[0,T]}\|n^{2H}\mathcal{A}^{n,k}_{2,s}\|_{L^2}<\infty.$
		
		The proof of Claim (iii) is much more complicated since we want to identity the  limit. 
		We decompose $\mathcal{A}^{n,k}_{2,s}$ as 
		\begin{align}
			\mathcal{A}^{n,k}_{2,s}=\sum_{j=1}^{m}\sum_{l=1}^{d}\int_{0}^{\frac{[ns]}{n}}f_{kjl}(s,u)M_u^{n,l}\d W_u^j+\sum_{j=1}^{m}\sum_{l=1}^{d}\partial_l \sigma^k_j(X^n_\frac{[ns]}{n})\int_{\frac{[ns]}{n}}^{s}K(s-u)M_u^{n,l}\d W_u^j
			,\nonumber\end{align}
		where
		\begin{align}\label{f kjl}
			f_{kjl}(s,u):=\Big(K(s-u)-K(\frac{[ns]}{n}-u)\Big)\partial_l \sigma^k_j(X^n_{\frac{[nu]}{n}}).
		\end{align}
		This yields 
		\begin{align}
			n^{4H}&\int_{0}^{t}\mathcal{A}^{n,k_1}_{2,s}\mathcal{A}^{n,k_2}_{2,s}ds\nonumber\\&=\sum_{j,j'=1}^{m}\sum_{l,l'=1}^{d}n^{4H}\int_{0}^{t}\Big(\int_{0}^{\frac{[ns]}{n}}f_{k_1jl}(s,u)M_u^{n,l}\d W_u^j\Big)\Big(\int_{0}^{\frac{[ns]}{n}}f_{k_2j'l'}(s,u)M_u^{n,l'}\d W_u^{j'}\Big)\d s\nonumber\\
			&\quad+\sum_{j,j'=1}^{m}\sum_{l,l'=1}^{d}n^{4H}\int_{0}^{t}\partial_{l'}\sigma^{k_2}_{j'}(X^n_\frac{[ns]}{n})\Big(\int_{0}^{\frac{[ns]}{n}}f_{k_1jl}(s,u)M_u^{n,l}\d W_u^j\Big)\Big(\int_{\frac{[ns]}{n}}^{s}K(s-u)M_u^{n,l'}\d W_u^{j'}\Big)\d s\nonumber\\
			&\quad+\sum_{j,j'=1}^{m}\sum_{l,l'=1}^{d}n^{4H}\int_{0}^{t}\partial_{l}\sigma^{k_1}_{j}(X^n_\frac{[ns]}{n})\Big(\int_{0}^{\frac{[ns]}{n}}f_{k_2j'l'}(s,u)M_u^{n,l'}\d W_u^{j'}\Big)\Big(\int_{\frac{[ns]}{n}}^{s}K(s-u)M_u^{n,l}\d W_u^{j}\Big)\d s\nonumber\\
			&\quad+\sum_{j,j'=1}^{m}\sum_{l,l'=1}^{d}n^{4H}\int_{0}^{t}\partial_{l}\sigma^{k_1}_{j}(X^n_\frac{[ns]}{n})\partial_{l'}\sigma^{k_2}_{j'}(X^n_\frac{[ns]}{n})\Big(\int_{\frac{[ns]}{n}}^{s}K(s-u)M_u^{n,l}\d W_u^{j}\Big)\Big(\int_{\frac{[ns]}{n}}^{s}K(s-u)M_u^{n,l'}\d W_u^{j'}\Big)\d s\nonumber\\
			&:=\sum_{j,j'=1}^{m}\sum_{l,l'=1}^{d}\big[I^{n,k_1,k_2}_{1,j,j',l,l'}+I^{n,k_1,k_2}_{2,j,j',l,l'}+I^{n,k_1,k_2}_{3,j,j',l,l'} +I^{n,k_1,k_2}_{4,j,j',l,l'}\big].\nonumber
		\end{align}
		In subsections  \ref{3.2.1}-\ref{3.2.3} below, we will show the following limits in $L^2$ as $n\to \infty$:  
		\[
		\begin{split} 
			I^{n,k_1,k_2}_{1,j,j',l,l'} \rightarrow & \begin{cases}\sum_{i=1}^{m}\big(C^{(1)}_G+C^{(2)}_G\big)\int_{0}^{t}\sigma^l_i(X_{s})\sigma^{l'}_{i}(X_{s})\partial_l\sigma^{k_1}_j(X_{s})\partial_{l'}\sigma^{k_2}_{j}(X_{s})\d s & \text { if } j=j', \\ 0 & \text { if } j \neq j',\end{cases} \\
			I^{n,k_1,k_2}_{2, j,j',l,l'} \rightarrow & 0 \\
			I^{n,k_1,k_2}_{3,j,j',l,l'}  \rightarrow&0\\ 
			I^{n,k_1,k_2}_{4,j,j',l,l'} \rightarrow & \begin{cases}\sum_{i=1}^{m}\big(C^{(3)}_G+C^{(4)}_G\big)\int_{0}^{t}\sigma^l_i(X_{s})\sigma^{l'}_{i}(X_{s})\partial_l\sigma^{k_1}_j(X_{s})\partial_{l'}\sigma^{k_2}_{j}(X_{s})\d s & \text { if } j=j', \\ 0& \text { if } j \neq j'.\end{cases} 
		\end{split}
		\]
		These limits  yield 
		\begin{align}
			n^{4H}&\int_{0}^{t}\mathcal{A}^{n,k_1}_{2,s}\mathcal{A}^{n,k_2}_{2,s}ds\xrightarrow[\text { in } L^2]{n \rightarrow \infty} C_M\sum_{i,j=1}^{m}\sum_{l,l'=1}^{d}\int_{0}^{t}\sigma^l_i(X_{s})\sigma^{l'}_{i}(X_{s})\partial_l\sigma^{k_1}_j(X_{s})\partial_{l'}\sigma^{k_2}_{j}(X_{s})\d s, 
		\end{align}
		where
		$C_M^2:=C^{(1)}_G+C^{(2)}_G+C^{(3)}_G+C^{(4)}_G $,  and the lemma is proved.  
	\end{proof}
	\subsubsection{The $L^2$-limit of $I^{n,k_1,k_2}_{1,j,j',l,l'}$}\label{3.2.1}
	%In the following proof, when we write $A_n\to B$, we mean that the convergence is in $L^2$ as 
	%$n\to \infty$.
	
	We use the following integration by parts formula   for any progressively measurable square integrable  processes  $h_1,h_2 $: 
	\begin{align}\label{dcp of stochastic integral}
		\Big(\int_{s}^{t}h_1(u)&\d W_u^j\Big)\Big(\int_{s}^{t}h_2(u)\d W_u^{j'}\Big)=\int_{s}^{t}\Big(\int_{s}^{u}h_1(r)\d W_r^j\Big)h_2(u)\d W_u^{j'}\nonumber\\&\quad+\int_{s}^{t}\Big(\int_{s}^{u}h_2(r)\d W_r^{j'}\Big)h_1(u)\d W_u^{j}+\int_{s}^{t}h_1(u)h_2(u)\d \langle W^{j}, W^{j'} \rangle_u.
	\end{align}
Thus,  the term $I^{n,k_1,k_2}_{1,j,j',l,l'}$ can be rewritten as
	\begin{align}
		I^{n,k_1,k_2}_{1,j,j',l,l'}&=n^{4H}\int_{0}^{t}\int_{0}^{\frac{[ns]}{n}}\Big(\int_{0}^{u}f_{k_1jl}(s,r)M_r^{n,l}\d W_r^j\Big)f_{k_2j'l'}(s,u)M_u^{n,l'}\d W_u^{j'}\d s\nonumber\\&\quad+n^{4H}\int_{0}^{t}\int_{0}^{\frac{[ns]}{n}}\Big(\int_{0}^{u}f_{k_2j'l'}(s,r)M_r^{n,l'}\d W_r^{j'}\Big)f_{k_1jl}(s,u)M_u^{n,l}\d W_u^{j}\d s\nonumber\\&\quad+n^{4H}\int_{0}^{t}\int_{0}^{\frac{[ns]}{n}}f_{k_1jl}(s,u)f_{k_2j'l'}(s,u)M_u^{n,l}M_u^{n,l'}\d \langle W^{j}, W^{j'} \rangle_u\d s\nonumber\\&:=\mathcal{O}^n_1+\mathcal{O}^n_2+\mathcal{O}^n_3.\nonumber
	\end{align} 
	Note that the last term vanishes if $j\neq j'$, and if $j=j'$, we have
	\begin{align}
		\mathcal{O}^n_3&=n^{4H}\int_{0}^{t}\int_{0}^{\frac{[ns]}{n}}f_{k_1jl}(s,u)f_{k_2jl'}(s,u)M_u^{n,l}M_u^{n,l'}\d u\d s.\nonumber
	\end{align}
	Recalling the  expression of $M_u^{n,l}$, i.e. \eqref{def of M u}, and applying \eqref{dcp of stochastic integral} again, we have
	\begin{align}\label{dcp of M-2}
		&M_u^{n,l}M_u^{n,l'}\nonumber\\&=\sum_{i,i'=1}^{m}\Big(\int_{0}^{\frac{[nu]}{n}}g_{li}(u,v)\d W^i_v\Big)\Big(\int_{0}^{\frac{[nu]}{n}}g_{l'i'}(u,v)\d W^{i'}_v\Big)+M_u^{1,n,l}M_u^{2,n,l'}+M_u^{2,n,l}M_u^{1,n,l'}\nonumber\\&\quad+\sum_{i,i'=1}^{m}F^{i,i',l,l'}_u\Big(\int_{\frac{[nu]}{n}}^{u}K(u-v)\d W^i_v\Big)\Big(\int_{\frac{[nu]}{n}}^{u}K(u-v)\d W^{i'}_v\Big)\nonumber\\&=\sum_{i,i'=1}^{m}\int_{0}^{\frac{[nu]}{n}}\Big(\int_{0}^{v}g_{li}(u,w)\d W^i_{w}\Big)g_{l'i'}(u,v)\d W_v^{i'}+\sum_{i,i'=1}^{m}\int_{0}^{\frac{[nu]}{n}}\Big(\int_{0}^{v}g_{l'i'}(u,w)\d W^{i'}_{w}\Big)g_{li}(u,v)\d W_v^{i}\nonumber\\&\quad+\sum_{i,i'=1}^{m}\int_{0}^{\frac{[nu]}{n}}g_{li}(u,v)g_{l'i'}(u,v)\d \langle W^{i}, W^{i'} \rangle_v\nonumber\\&\quad+\sum_{i,i'=1}^{m}F^{i,i',l,l'}_u\int_{\frac{[nu]}{n}}^{u}\Big(\int_{\frac{[nu]}{n}}^{v}K(u-w)\d W^i_w\Big)K(u-v)\d W^{i'}_v\nonumber\\&\quad+\sum_{i,i'=1}^{m}F^{i,i',l,l'}_u\int_{\frac{[nu]}{n}}^{u}\Big(\int_{\frac{[nu]}{n}}^{v}K(u-w)\d W^{i'}_w\Big)K(u-v)\d W^{i}_v\nonumber\\&\quad+\sum_{i,i'=1}^{m}F^{i,i',l,l'}_u\int_{\frac{[nu]}{n}}^{u}K(u-v)^2\d \langle W^{i}, W^{i'} \rangle_v+M_u^{1,n,l}M_u^{2,n,l'}+M_u^{2,n,l}M_u^{1,n,l'},
	\end{align}
	where
	\begin{align}\label{g li}
		g_{li}(u,v):=\Big(K(u-v)-K(\frac{[nu]}{n}-v)\Big) \sigma^l_i(X^n_{\frac{[nv]}{n}})
	\end{align}
	and
	\begin{align}\label{def of F}
		F^{i,i',l,l'}_u:=\sigma^l_i(X^n_{\frac{[nu]}{n}})\sigma^{l'}_{i'}(X^n_{\frac{[nu]}{n}}).
	\end{align}
	Note that $ \langle W^{i}, W^{i'} \rangle_v=0$ if $i\neq i'$. Therefore, if $i=i'$, $\mathcal{O}^n_3$ can be decomposed further as 
	\begin{align}
		\mathcal{O}^n_3
		&=n^{4H} \sum_{i,i'=1}^{m}\int_{0}^{t}\int_{0}^{\frac{[ns]}{n}}f_{k_1jl}(s,u)f_{k_2jl'}(s,u)\int_{0}^{\frac{[nu]}{n}}\Big(\int_{0}^{v}g_{li}(u,w)\d W^i_{w}\Big)g_{l'i'}(u,v)\d W_v^{i'}\d u\d s\nonumber\\&\quad+n^{4H}\sum_{i,i'=1}^{m}\int_{0}^{t}\int_{0}^{\frac{[ns]}{n}}f_{k_1jl}(s,u)f_{k_2jl'}(s,u)\int_{0}^{\frac{[nu]}{n}}\Big(\int_{0}^{v}g_{l'i'}(u,w)\d W^{i'}_{w}\Big)g_{li}(u,v)\d W_v^{i}\d u\d s\nonumber\\&\quad+n^{4H}\sum_{i=1}^{m}\int_{0}^{t}\int_{0}^{\frac{[ns]}{n}}f_{k_1jl}(s,u)f_{k_2jl'}(s,u)\int_{0}^{\frac{[nu]}{n}}g_{li}(u,v)g_{l'i}(u,v)\d v\d u\d s\nonumber\\&\quad+n^{4H}\sum_{i,i'=1}^{m}\int_{0}^{t}\int_{0}^{\frac{[ns]}{n}}f_{k_1jl}(s,u)f_{k_2jl'}(s,u)F^{i,i',l,l'}_u\int_{\frac{[nu]}{n}}^{u}\Big(\int_{\frac{[nu]}{n}}^{v}K(u-w)\d W^i_w\Big)K(u-v)\d W^{i'}_v\d u\d s\nonumber\\&\quad+n^{4H}\sum_{i,i'=1}^{m}\int_{0}^{t}\int_{0}^{\frac{[ns]}{n}}f_{k_1jl}(s,u)f_{k_2jl'}(s,u)F^{i,i',l,l'}_u\int_{\frac{[nu]}{n}}^{u}\Big(\int_{\frac{[nu]}{n}}^{v}K(u-w)\d W^{i'}_w\Big)K(u-v)\d W^{i}_v\d u\d s\nonumber\\&\quad+n^{4H}\sum_{i=1}^{m}\int_{0}^{t}\int_{0}^{\frac{[ns]}{n}}f_{k_1jl}(s,u)f_{k_2jl'}(s,u)F^{i,i,l,l'}_u\int_{\frac{[nu]}{n}}^{u}K(u-v)^2\d v\d u\d s\nonumber\\&\quad+n^{4H}\int_{0}^{t}\int_{0}^{\frac{[ns]}{n}}f_{k_1jl}(s,u)f_{k_2jl'}(s,u)M_u^{1,n,l}M_u^{2,n,l'}\d u\d s\nonumber\\&\quad+n^{4H}\int_{0}^{t}\int_{0}^{\frac{[ns]}{n}}f_{k_1jl}(s,u)f_{k_2jl'}(s,u)M_u^{2,n,l}M_u^{1,n,l'}\d u\d s\nonumber\\&:=\mathcal{O}^n_{31}+\mathcal{O}^n_{32}+\sum_{i=1}^{m}\mathcal{O}^n_{33}+\mathcal{O}^n_{34}+\mathcal{O}^n_{35}+\sum_{i=1}^{m}\mathcal{O}^n_{36}+\mathcal{O}^n_{37}+\mathcal{O}^n_{38}.\nonumber
	\end{align}
	\begin{lem}\label{31 32}
		For any $i,i',j=1,\cdots,m;l,l'=1,\cdots,d,$
		\[
		\begin{split} 
			(i)\qquad &     \lim_{n\to\infty}\sup_{t\in[0,T]}\|\mathcal{O}^{n}_{31}\|_{L^2}=0,\\
			(ii)\qquad &  \lim_{n\to\infty}\sup_{t\in[0,T]}\|\mathcal{O}^{n}_{32}\|_{L^2}=0. 
		\end{split}
		\]
	\end{lem}
	\begin{proof}
		Denote 
		\begin{align}\label{def of D 1}
			D^{i ,i',l,l'}_{1,u}:=n^{2H}\int_{0}^{\frac{[nu]}{n}}\Big(\int_{0}^{v}g_{li}(u,w)\d W^i_{w}\Big)g_{l'i'}(u,v)\d W_v^{i'}.
		\end{align}
		By Fubini's theorem, we have
		\begin{align}
			\mathcal{O}^n_{31}&=\sum_{i,i'=1}^{m}\int_{0}^{t}n^{2H}\int_{0}^{s}f_{k_1jl}(s,u)f_{k_2jl'}(s,u)D^{i ,i',l,l'}_{1,u}\d u\d s\nonumber\\&\quad-\sum_{i,i'=1}^{m}\int_{0}^{t}n^{2H}\int_{\frac{[ns]}{n}}^{s}f_{k_1jl}(s,u)f_{k_2jl'}(s,u)D^{i ,i',l,l'}_{1,u}\d u\d s\nonumber\\&=\sum_{i,i'=1}^{m}\int_{0}^{t}n^{2H}\int_{0}^{u}f_{k_1jl}(s,u)f_{k_2jl'}(s,u)\d sD^{i ,i',l,l'}_{1,u}\d u\nonumber\\&\quad-\sum_{i,i'=1}^{m}\int_{0}^{t}n^{2H}\int_{\frac{[ns]}{n}}^{s}f_{k_1jl}(s,u)f_{k_2jl'}(s,u)\d uD^{i ,i',l,l'}_{1,s}\d s.\nonumber
		\end{align}
		Since the derivatives of $\sigma$ are bounded we see 
		\begin{align}
			\|\mathcal{O}^n_{31}\|_{L^2}
			&\leq \sum_{i,i'=1}^{m}\Big\|\int_{0}^{t}D^{i ,i',l,l'}_{1,u}\d u\Big\|_{L^2}\sup_{u\in[0,T]}n^{2H}\int_{0}^{u}\Big(K(s-u)-K(\frac{[ns]}{n}-u)\Big)^2\d s\nonumber\\
			&\quad+\sum_{i,i'=1}^{m}\Big\|\int_{0}^{t}D^{i ,i',l,l'}_{1,s}\d s\Big\|_{L^2}\sup_{s\in[0,T]}n^{2H}\int_{\frac{[ns]}{n}}^{s}\Big(K(s-u)-K(\frac{[ns]}{n}-u)\Big)^2\d u\nonumber\\
			&\leq C\sum_{i,i'=1}^{m}\Big\|\int_{0}^{t}D^{i ,i',l,l'}_{1,u}\d u\Big\|_{L^2} 
%			\nonumber\\
%			&{\red \leq 0 \quad \text{as $n\to\infty$}},\nonumber
		\end{align}
		by \eqref{f kjl}, Minkowski's inequality, Fubini's theorem, Lemmas \ref{2-small time estimate},   \ref{3-small time estimate} and   \ref{esti of D 1}. Applying the dominated convergence theorem  with respect to $\d u\otimes \d s$, we have $\mathcal{O}^n_{31} \rightarrow  0$ \text { in } $L^2$ as ${n \rightarrow \infty}$.
		
		In a similar way we can prove  $\mathcal{O}^n_{32}\to 0$ in $L^2$.
	\end{proof}
	
	\begin{lem}\label{33}
		For any $i,j=1,\cdots,m;l,l'=1,\cdots,d,$                                                                                                                                                                  
		\begin{align}
			\mathcal{O}^n_{33}& \rightarrow %\xrightarrow[\text { in } L^2]{n \rightarrow \infty}\
			C^{(1)}_G\int_{0}^{t}\sigma^l_i(X_{s})\sigma^{l'}_{i}(X_{s})\partial_l\sigma^{k_1}_j(X_{s})\partial_{l'}\sigma^{k_2}_{j}(X_{s})\d s,\nonumber
		\end{align}
		where
		\begin{align}
			C_G^{(1)}&:=\frac{1}{G^2}\int_{0}^{\infty}\mu(r,1)^2\d r\Big(\int_{0}^{1}r^{2H}\int_{-\infty}^{r}\nu(z,r)^2\d z\d r-\frac{1}{(4H+1)}\Big)\,,\nonumber
		\end{align}
where  $\mu(r,1)$  is given in \eqref{basic types}  and $\nu(z,r)$ is given by  \eqref{def of nu}.
	\end{lem}
	\begin{proof}
		By   change of variables $\tilde{v}=\frac{[nu]}{n}-v$, $r=\frac{\tilde{v}}{\delta_{(n,u)}}$, \eqref{basic types} and \eqref{g li}, we have
		\begin{align}\label{transfer of g}
			&\int_{0}^{\frac{[nu]}{n}}g_{li}(u,v)g_{l'i}(u,v)\d v\nonumber\\&=\int_{0}^{\frac{[nu]}{n}}\Big(K(u-v)-K(\frac{[nu]}{n}-v)\Big)^2 \sigma^l_i(X^n_{\frac{[nv]}{n}})\sigma^{l'}_{i}(X^n_{\frac{[nv]}{n}})\d v\nonumber\\&=\int_{0}^{\frac{[nu]}{n}}\Big(K(\tilde{v}+\delta_{(n,u)})-K(\tilde{v})\Big)^2F^{i,i,l,l'}_{\frac{[nu]+[-n\tilde{v}]}{n}}\d \tilde{v}=|\delta_{(n,u)}|^{2H}\tilde{g}^n_u,
		\end{align}
		where
		\begin{align}\label{ti g}
			\tilde{g}^n_u:=\frac{1}{G}\int_{0}^{\frac{[nu]}{n\delta_{(n,u)}}}|\mu(r,1)|^2 F^{i,i,l,l'}_{\frac{[nu]+[-n\delta_{(n,u)}r]}{n}}\d r.
		\end{align}
		So by \eqref{f kjl},  \eqref{transfer of g} and Fubini's theorem we have
		\begin{align}
			\mathcal{O}^n_{33}&=n^{4H}\int_{0}^{t}\int_{0}^{\frac{[ns]}{n}}f_{k_1jl}(s,u)f_{k_2jl'}(s,u)\int_{0}^{\frac{[nu]}{n}}g_{li}(u,v)g_{l'i}(u,v)\d v\d u\d s\nonumber\\&=n^{4H}\int_{0}^{t}\int_{0}^{s}f_{k_1jl}(s,u)f_{k_2jl'}(s,u)\int_{0}^{\frac{[nu]}{n}}g_{li}(u,v)g_{l'i}(u,v)\d v\d u\d s\nonumber\\&\quad-n^{4H}\int_{0}^{t}\int_{\frac{[ns]}{n}}^{s}f_{k_1jl}(s,u)f_{k_2jl'}(s,u)\int_{0}^{\frac{[nu]}{n}}g_{li}(u,v)g_{l'i}(u,v)\d v\d u\d s\nonumber\\&=\frac{1}{2}n^{4H}\int_{0}^{t}\int_{0}^{t}f_{k_1jl}(s,u)f_{k_2jl'}(s,u)\int_{0}^{\frac{[nu]}{n}}g_{li}(u,v)g_{l'i}(u,v)\d v\d s\d u\nonumber\\&\quad-n^{4H}\int_{0}^{t}\int_{\frac{[ns]}{n}}^{s}f_{k_1jl}(s,u)f_{k_2jl'}(s,u)\d u\int_{0}^{\frac{[ns]}{n}}g_{li}(u,v)g_{l'i'}(u,v)\d v \d s\nonumber\\&=n^{4H}\int_{0}^{t}|\delta_{(n,u)}|^{2H}\tilde{g}^n_uE^{j,j,l,l'}_u\int_{0}^{u}\Big(K(s-u)-K(\frac{[ns]}{n}-u)\Big)^2\d s\d u\nonumber\\&\quad- n^{4H}\int_{0}^{t}|\delta_{(n,s)}|^{2H}\tilde{g}^n_sE^{j,j,l,l'}_s\int_{\frac{[ns]}{n}}^{s}\Big(K(s-u)-K(\frac{[ns]}{n}-u)\Big)^2\d u\d s\nonumber\\&:=\mathcal{O}^n_{331}+\mathcal{O}^n_{332},\nonumber
		\end{align}
		where
		\begin{align}\label{def of E}
			E^{j,j',l,l'}_u:=\partial_l \sigma^{k_1}_j(X^n_{\frac{[nu]}{n}})\partial_{l'}\sigma^{k_2}_{j'}(X^n_{\frac{[nu]}{n}}).
		\end{align}
		Therefore we can rewrite $\mathcal{O}^n_{331}$ as
		\begin{align}
			\mathcal{O}^n_{331}&=\frac{1}{G}\int_{0}^{t}|n\delta_{(n,u)}|^{2H}\tilde{g}^n_uE^{j,j,l,l'}_u\int_{-[nu]}^{nu-[nu]}\nu(z,nu-[nu])^2\d z\d u\nonumber\\&=\frac{1}{G}\int_{0}^{t}|nu-[nu]|^{2H}\int_{-\infty}^{nu-[nu]}\nu(z,nu-[nu])^2\d z\tilde{g}^n_uE^{j,j,l,l'}_u\d u\nonumber\\&\quad-\frac{1}{G}\int_{0}^{t}|nu-[nu]|^{2H}\int_{-\infty}^{-[nu]}\nu(z,nu-[nu])^2\d z\tilde{g}^n_uE^{j,j,l,l'}_u\d u\nonumber\\&:=\mathcal{O}^n_{3311}+\mathcal{O}^n_{3312}\nonumber.
		\end{align}
		For the term $\mathcal{O}^n_{3311}$, we shall  apply Lemma \ref{limit distribution of stochastic integral}.    To this end,  we first show  (as $n\to \infty$)
		\begin{align}\label{back need}
			\tilde{g}^n_uE^{j,j,l,l'}_u&\xrightarrow {  L^2(\d u\otimes \d P) } \frac{1}{G}\sigma^l_i(X_{u})\sigma^{l'}_{i}(X_{u})\int_{0}^{\infty}\mu(r,1)^2\d r\partial_l\sigma^{k_1}_j(X_{u})\partial_{l'}\sigma^{k_2}_{j}(X_{u}).
		\end{align}
		Note that the right hand side is a continuous function of $s$. It follows from \eqref{ti g}, Fubini's theorem and Minkowski's   inequality that
		\begin{align}
			\E&\Big[\int_{0}^{t}\Big|\tilde{g}^n_uE^{j,j,l,l'}_u-\frac{1}{G}\sigma^l_i(X_{u})\sigma^{l'}_{i}(X_{u})\int_{0}^{\infty}\mu(r,1)^2\d r\partial_l\sigma^{k_1}_j(X_{u})\partial_{l'}\sigma^{k_2}_j(X_{u})\Big|^2\d u\Big]\nonumber\\&\leq \E\Big[\int_{0}^{t}\Big|\tilde{g}^n_u-\frac{1}{G}\sigma^l_i(X_{u})\sigma^{l'}_{i}(X_{u})\int_{0}^{\infty}\mu(r,1)^2\d r\Big|^2\Big|E^{j,j,l,l'}_u\Big|^2\d u\Big]\nonumber\\&\quad+\E\Big[\int_{0}^{t}\Big|\frac{1}{G}\sigma^l_i(X_{u})\sigma^{l'}_{i}(X_{u})\int_{0}^{\infty}\mu(r,1)^2\d r\Big|^2\Big|E^{j,j,l,l'}_u-\partial_l\sigma^{k_1}_j(X_{u})\partial_{l'}\sigma^{k_2}_j(X_{u})\Big|^2\d u\Big]\nonumber
			\\&\leq C\E
			\Big[\int_{0}^{t}\Big|\int_{0}^{\infty}|\mu(r,1)|^2\Big(\mathbb{I}_{(0,\frac{[nu]}{n\delta_{(n,u)}})}(r)F^{i,i,l,l'}_{\frac{[nu]+[-n\delta_{(n,u)}r]}{n}}-\sigma^l_i(X_{u})\sigma^{l'}_{i}(X_{u})\Big) \Big|^2\d u\Big]\nonumber\\&\quad+C\E\Big[\int_{0}^{t}\Big|\sigma^l_i(X_{u})\sigma^{l'}_{i}(X_{u})\Big|^2\Big|E^{j,j,l,l'}_u-\partial_l\sigma^{k_1}_j(X_{u})\partial_{l'}\sigma^{k_2}_j(X_{u})\Big|^2\d u\Big]\nonumber\\&\leq C\int_{0}^{t}\Big(\int_{0}^{\infty}|\mu(r,1)|^2\Big\|\mathbb{I}_{(0,\frac{[nu]}{n\delta_{(n,u)}})}(r)F^{i,i,l,l'}_{\frac{[nu]+[-n\delta_{(n,u)}r]}{n}}-\sigma^l_i(X_{u})\sigma^{l'}_{i}(X_{u}) \Big\|_{L^2}\d r\Big)^2\d u\nonumber\\&\quad+C\int_{0}^{t}\E\Big[\Big|\sigma^l_i(X_{u})\sigma^{l'}_{i}(X_{u})\Big|^2\Big|E^{j,j,l,l'}_u-\partial_l\sigma^{k_1}_j(X_{u})\partial_{l'}\sigma^{k_2}_{j}(X_{u})\Big|^2\Big]\d u:=\Delta_1+\Delta_2.\nonumber
		\end{align}
		Since
		\begin{align}\label{sigma-1}
			&\Big\|\mathbb{I}_{(0,\frac{[nu]}{n\delta_{(n,u)}})}(r)F^{i,i,l,l'}_{\frac{[nu]+[-n\delta_{(n,u)}r]}{n}}-\sigma^l_i(X_{u})\sigma^{l'}_{i}(X_{u})\Big\|_{L^2}\nonumber\\&\leq \Big\|\sigma^l_i(X^n_{\frac{[nu]+[-n\delta_{(n,u)}r]}{n}})\sigma^{l'}_{i}(X^n_{\frac{[nu]+[-n\delta_{(n,u)}r]}{n}})-\sigma^l_i(X_{u})\sigma^{l'}_{i}(X_{u})\Big\|_{L^2}\mathbb{I}_{(0,\frac{[nu]}{n\delta_{(n,u)}})}(r)\nonumber\\&\quad+\|\sigma^l_i(X_{u})\sigma^{l'}_{i}(X_{u})\|_{L^2}\mathbb{I}_{(\frac{[nu]}{n\delta_{(n,u)}},\infty)}(r),
		\end{align}
		with the last term vanishing as $n\to\infty$. For the first term, by   Assumption \ref{assumption 2.2}, the Cauchy-Schwarz inequality, Lemmas \ref{bound of X},    \ref{bound of X n},   \ref{continuous of X n} and  \ref{rate X-X n}, we have
		\begin{align}\label{sigma-2}
			&\big\|\sigma^l_i(X^n_{\frac{[nu]+[-n\delta_{(n,u)}r]}{n}})\sigma^{l'}_{i}(X^n_{\frac{[nu]+[-n\delta_{(n,u)}r]}{n}})-\sigma^l_i(X_{u})\sigma^{l'}_{i}(X_{u})\big\|_{L^2}\nonumber\\&\leq \Big\|\big(\sigma^l_i(X^n_{\frac{[nu]+[-n\delta_{(n,u)}r]}{n}})-\sigma^l_i(X_{u})\big)\sigma^{l'}_{i}(X^n_{\frac{[nu]+[-n\delta_{(n,u)}r]}{n}})\Big\|_{L^2}\nonumber\\&\quad+\Big\|\big(\sigma^{l'}_{i}(X^n_{\frac{[nu]+[-n\delta_{(n,u)}r]}{n}})-\sigma^{l'}_{i}(X_{u})\big)\sigma^{l}_{i}(X_{u})\Big\|_{L^2}\nonumber\\&\leq \Big\|\big(\sigma^l_i(X^n_{\frac{[nu]+[-n\delta_{(n,u)}r]}{n}})-\sigma^l_i(X_{u})\big)\Big\|_{L^4}\|\sigma^{l'}_{i}(X^n_{\frac{[nu]+[-n\delta_{(n,u)}r]}{n}})\|_{L^4}\nonumber\\&\quad+\Big\|\big(\sigma^{l'}_{i}(X^n_{\frac{[nu]+[-n\delta_{(n,u)}r]}{n}})-\sigma^{l'}_{i}(X_{u})\big)\Big\|_{L^4}\|\sigma^{l}_{i}(X_{u})\|_{L^4}\nonumber\\&\leq C\Big\|X^n_{\frac{[nu]+[-n\delta_{(n,u)}r]}{n}}-X^n_u+X^n_u-X_u\Big\|_{L^4}\nonumber\\&\leq C\|X^n_{\frac{[nu]+[-n\delta_{(n,u)}r]}{n}}-X^n_u\|_{L^4}+C\|X^n_u-X_u\|_{L^4}\nonumber\\&\leq C\Big(\Big|\frac{[nu]+[-n\delta_{(n,u)}r]}{n}-u\Big|^H+ n^{-2H}\Big)\nonumber\\&\leq Cr^Hn^{-H}+Cn^{-2H}\to 0 \quad \text{as $n\to\infty$}.
		\end{align}
		Consequently $\Delta_1\to 0$ by applying the dominated convergence theorem  for the integral of $\d u$ and $\d r$ respectively.
		
		On the other hand, since the derivatives of $\sigma$ are bounded, by   Assumptions \ref{assumption 2.2} and  \ref{assumption 2.3}, the Cauchy-Schwarz inequality, Lemmas \ref{bound of X},   \ref{continuous of X},   \ref{bound of X n},  and  \ref{rate X-X n}, we have
		\begin{align}\label{sigma-3}
			\E\Big[&|\sigma^l_i(X_{u})\sigma^{l'}_{i}(X_{u})|^2\Big|\partial_l\sigma^{k_1}_j(X^n_{\frac{[nu]}{n}})\partial_{l'}\sigma^{k_2}_{j}(X^n_{\frac{[nu]}{n}})-\partial_l\sigma^{k_1}_j(X_{u})\partial_{l'}\sigma^{k_2}_{j}(X_{u})\Big|^2\Big]\nonumber\\&\leq C \|\sigma^l_i(X_{u})\sigma^{l'}_{i}(X_{u})\|^2_{L^4}\Big\|\partial_l\sigma^{k_1}_j(X^n_{\frac{[nu]}{n}})\partial_{l'}\sigma^{k_2}_{j}(X^n_{\frac{[nu]}{n}})-\partial_l\sigma^{k_1}_j(X_{u})\partial_{l'}\sigma^{k_2}_{j}(X_{u})\Big\|^2_{L^4}\nonumber\\&\leq C\Big\|\partial_l\sigma^{k_1}_j(X^n_{\frac{[nu]}{n}})\partial_{l'}\sigma^{k_2}_{j}(X^n_{\frac{[nu]}{n}})-\partial_l\sigma^{k_1}_j(X_{u})\partial_{l'}\sigma^{k_2}_{j}(X_{u})\Big\|^2_{L^4}\nonumber\\&\leq C\Big\|\partial_l\sigma^{k_1}_j(X^n_{\frac{[nu]}{n}})\partial_{l'}\sigma^{k_2}_{j}(X^n_{\frac{[nu]}{n}})-\partial_l\sigma^{k_1}_j(X^n_{\frac{[nu]}{n}})\partial_{l'}\sigma^{k_2}_{j}(X_{u})\Big\|^2_{L^4}\nonumber\\&\quad+C\Big\|\partial_l\sigma^{k_1}_j(X^n_{\frac{[nu]}{n}})\partial_{l'}\sigma^{k_2}_{j}(X_{u})-\partial_l\sigma^{k_1}_j(X_{u})\partial_{l'}\sigma^{k_2}_{j}(X_{u})\Big\|^2_{L^4}\nonumber\\&\leq C\Big\|\partial_{l'}\sigma^{k_2}_{j}X^n_{\frac{[nu]}{n}})-\partial_{l'}\sigma^{k_2}_{j}(X_{u})\Big\|^2_{L^4}+C\Big\|\partial_l\sigma^{k_1}_j(X^n_{\frac{[nu]}{n}})-\partial_l\sigma^{k_1}_j(X_{u})\Big\|^2_{L^4}\nonumber
			\\&\leq C\Big\|X^n_{\frac{[nu]}{n}}-X_{u}\Big\|^2_{L^4}\leq C\Big\|X^n_{\frac{[nu]}{n}}-X_{\frac{[nu]}{n}}\Big\|^2_{L^4}+C\Big\|X_{\frac{[nu]}{n}}-X_{u}\Big\|^2_{L^4}\nonumber\\&\leq C(n^{-4H}+n^{-2H})\,. 
		\end{align}
		Hence,  $\Delta_2\to 0$ by applying the dominated convergence theorem  for the integral of $\d u$. Then Lemma \ref{limit distribution of stochastic integral} gives that
		\begin{align}
			&\mathcal{O}^n_{3311}\xrightarrow {  L^2 }\nonumber\\& \frac{1}{G^2}\int_{0}^{\infty}\mu(r,1)^2\d r\int_{0}^{1}r^{2H}\int_{-\infty}^{r}v(z,r)^2\d z\d r\int_{0}^{t}\sigma^l_i(X_{u})\sigma^{l'}_{i}(X_{u})\partial_{l}\sigma^{k_1}_j(X_{u})\partial_{l'}\sigma^{k_2}_{j}(X_{u})\d u.\nonumber
		\end{align}
		For the term $\mathcal{O}^n_{3312}$, since the derivatives of $\sigma$ are bounded, applying Minkowski's   inequality, Lemma \ref{small time estimate 2} and \eqref{ti g}, we have
		\begin{align}
			\|\mathcal{O}^n_{3312}\|_{L^2}&\leq \Big\|\int_{0}^{t}\int_{-\infty}^{-[nu]}\nu(z,nu-[nu])^2\d z\cdot\tilde{g}^n_u\d u\Big\|_{L^2}\nonumber\\&\leq \int_{0}^{t}\Big(\int_{-\infty}^{-[nu]}\nu(z,nu-[nu])^2\d z\cdot\|\tilde{g}^n_u\|_{L^2}\Big)\d u\nonumber\\&\leq C\int_{0}^{t}\int_{-\infty}^{-[nu]}\nu(z,nu-[nu])^2\d z\d u\to 0\quad\text{as $n\to\infty$}.\nonumber
		\end{align}
		Since
		\begin{align}\label{est of ti g}
			\|\tilde{g}^n_u\|_{L^2}&\leq \frac{1}{G}\int_{0}^{\frac{[nu]}{n\delta_{(n,u)}}}|\mu(r,1)|^2 \|\sigma^l_i(X^n_{\frac{[nu]+[-n\delta_{(n,u)}r]}{n}})\|_{L^4}\|\sigma^{l'}_{i}(X^n_{\frac{[nu]+[-n\delta_{(n,u)}r]}{n}})\|_{L^4}\d r\nonumber\\&\leq C\int_{0}^{\infty}|\mu(r,1)|^2\d r<\infty,
		\end{align}
		where Minkowski's  inequality, Assumption \ref{assumption 2.2} and Lemma \ref{bound of X n} are used.
		For the term $\mathcal{O}^n_{332}$, by \eqref{transfer of K-2} we can rewrite $\mathcal{O}^n_{332}$ as
		\begin{align}
			\mathcal{O}^n_{332}&=-\frac{1}{G} \int_{0}^{t}|n\delta_{(n,s)}|^{4H}\tilde{g}^n_sE^{j,j,l,l'}_s\int_{0}^{1}|\mu(r,1)|^2\d r\d s\,. \nonumber
		\end{align}
		Thus  \eqref{back need}  and  Lemma \ref{limit distribution of stochastic integral} yield 
		\begin{align}
			\mathcal{O}^n_{332}\xrightarrow { L^2 } -\frac{1}{G^2(4H+1)}\int_{0}^{\infty}\mu(r,1)^2\d r\int_{0}^{t}\sigma^l_i(X_{u})\sigma^{l'}_{i}(X_{u})\partial_l\sigma^{k_1}_j(X_{u})\partial_{l'}\sigma^{k_2}_{j}(X_{u})\d u.\nonumber
		\end{align}
		The proof is complete.
	\end{proof}
	\begin{lem}\label{34 35}
		For any $i,i',j=1,\cdots,m;l,l'=1,\cdots,d,$
		\[
		\begin{split}
			(i)\qquad & \lim_{n\to\infty}\sup_{t\in[0,T]}\|\mathcal{O}^{n}_{34}\|_{L^2}=0,\\
			(ii)\qquad  &\lim_{n\to\infty}\sup_{t\in[0,T]}\|\mathcal{O}^{n}_{35}\|_{L^2}=0. 
		\end{split}
		\]
	\end{lem}
	\begin{proof}
		Set
		\begin{align}\label{def of D 2}
			D^{i ,i',l,l'}_{2,u}:=n^{2H}F^{i,i',l,l'}_u\int_{\frac{[nu]}{n}}^{u}\Big(\int_{\frac{[nu]}{n}}^{v}K(u-w)\d W^{i'}_w\Big)K(u-v)\d W^{i}_v.
		\end{align}
		Similar to the term $\mathcal{O}^n_{31}$, since the derivatives of $\sigma$ are bounded, we have
		\begin{align}
			\|\mathcal{O}^n_{34}\|_{L^2}&\leq \sum_{i,i'=1}^{m}\Big\|\int_{0}^{t}D^{i ,i',l,l'}_{2,u}\d u\Big\|_{L^2}\sup_{u\in[0,T]}n^{2H}\int_{0}^{u}\Big(K(s-u)-K(\frac{[ns]}{n}-u)\Big)^2\d s\nonumber\\&\quad+\sum_{i,i'=1}^{m}\Big\|\int_{0}^{t}D^{i ,i',l,l'}_{2,s}\d s\Big\|_{L^2}\sup_{s\in[0,T]}n^{2H}\int_{\frac{[ns]}{n}}^{s}\Big(K(s-u)-K(\frac{[ns]}{n}-u)\Big)^2\d u\nonumber\\&\leq C\sum_{i,i'=1}^{m}\Big\|\int_{0}^{t}D^{i ,i',l,l'}_{2,u}\d u\Big\|_{L^2} 
		\end{align}
		by \eqref{f kjl}, Minkowski's inequality, Fubini's theorem, Lemmas \ref{2-small time estimate},   \ref{3-small time estimate} and  \ref{esti of D 2}. Applying the dominated convergence theorem  with respect to $\d u\otimes \d s$, we have $\mathcal{O}^n_{34} \xrightarrow[\text { in } L^2]{n \rightarrow \infty} 0$.
		
		Similarly  to $\mathcal{O}^n_{34}$, we have  $\mathcal{O}^n_{35}\to 0$ in $L^2$. 
	\end{proof}
	\begin{lem}\label{36}
		For any $i,j=1,\cdots,m;l,l'=1,\cdots,d,$
		\begin{align}
			\mathcal{O}^n_{36}&\rightarrow {L^2 } C^{(2)}_G\int_{0}^{t}\sigma^l_i(X_{s})\sigma^{l'}_{i}(X_{s})\partial_l\sigma^{k_1}_j(X_{s})\partial_{l'}\sigma^{k_2}_{j}(X_{s})\d s,\nonumber
		\end{align}
		where
		\begin{align}
			C_G^{(2)}&:=\frac{1}{2HG^2}\int_{0}^{\infty}\mu(r,1)^2\d r\Big(\int_{0}^{1}r^{2H}\int_{-\infty}^{r}\nu(z,r)^2\d z\d r\nonumber\\&-\int_{0}^{1}\int_{-r}^{0}\mu(z,r)^2(z+[-z])^{2H}\d z\d r\Big).\nonumber
		\end{align}
	\end{lem}
	\begin{proof}
		By \eqref{f kjl}, \eqref{def of E}, Fubini's theorem and Lemma \ref{small time estimate}, we have that
		\begin{align}
			\mathcal{O}^n_{36}&=n^{4H}\int_{0}^{t}\int_{0}^{\frac{[ns]}{n}}f_{k_1jl}(s,u)f_{k_2jl'}(s,u)F^{i,i,l,l'}_u\int_{\frac{[nu]}{n}}^{u}K(u-v)^2\d v\d u\d s\nonumber\\&=\frac{n^{4H}}{2HG}\int_{0}^{t}\int_{0}^{\frac{[ns]}{n}}f_{k_1jl}(s,u)f_{k_2jl'}(s,u)F^{i,i,l,l'}_u|\delta_{(n,u)}|^2\d u\d s\nonumber\\&=\frac{n^{4H}}{2HG}\int_{0}^{t}\int_{0}^{s}f_{k_1jl}(s,u)f_{k_2jl'}(s,u)F^{i,i,l,l'}_u|\delta_{(n,u)}|^2\d u\d s\nonumber\\&\quad-\frac{n^{4H}}{2HG}\int_{0}^{t}\int_{\frac{[ns]}{n}}^{s}f_{k_1jl}(s,u)f_{k_2jl'}(s,u)F^{i,i,l,l'}_u|\delta_{(n,u)}|^2\d u\d s\nonumber\\&=\frac{n^{4H}}{2HG}\int_{0}^{t}|\delta_{(n,u)}|^{2H}E^{j,j,l,l'}_uF^{i,i,l,l'}_u\int_{0}^{u}\Big(K(s-u)-K(\frac{[ns]}{n}-u)\Big)^2\d s\d u\nonumber\\&\quad-\frac{n^{4H}}{2HG}\int_{0}^{t}E^{j,j,l,l'}_sF^{i,i,l,l'}_s\int_{\frac{[ns]}{n}}^{s}\Big(K(s-u)-K(\frac{[ns]}{n}-u)\Big)^2|\delta_{(n,u)}|^2\d u\d s\nonumber\\&:=\mathcal{O}^n_{361}+\mathcal{O}^n_{362}.\nonumber
		\end{align}
		By \eqref{tranfer of K} we have 
		\begin{align}
			\mathcal{O}^n_{361}&=\frac{1}{2HG^2}\int_{0}^{t}|n\delta_{(n,u)}|^{2H}E^{j,j,l,l'}_uF^{i,i,l,l'}_u \int_{-[nu]}^{nu-[nu]}\nu(z,nu-[nu])^2\d z\d u\nonumber\\&=\frac{1}{2HG^2}\int_{0}^{t}|n\delta_{(n,u)}|^{2H}E^{j,j,l,l'}_uF^{i,i,l,l'}_u \int_{-\infty}^{nu-[nu]}\nu(z,nu-[nu])^2\d z\d u\nonumber\\&\quad-\frac{1}{2HG^2}\int_{0}^{t}|n\delta_{(n,u)}|^{2H}E^{j,j,l,l'}_uF^{i,i,l,l'}_u \int_{-\infty}^{-[nu]}\nu(z,nu-[nu])^2\d z\d u\nonumber\\&:=\mathcal{O}^n_{3611}+\mathcal{O}^n_{3612}.\nonumber
		\end{align}
		For the term $\mathcal{O}^n_{3611}$, we shall use Lemma \ref{limit distribution of stochastic integral}  and we    show that
		\begin{align}\label{limit distribution-1}
			E^{j,j',l,l'}_uF^{i,i',l,l'}_u\xrightarrow { L^2(\d u\otimes \d P)} \sigma^l_i(X_{u})\sigma^{l'}_{i}(X_{u})\partial_l\sigma^{k_1}_j(X_{u})\partial_{l'}\sigma^{k_2}_{j}(X_{u}).
		\end{align}
		By Fubini's theorem, Minkowski's inequality, the Cauchy-Schwarz inequality, Assumption \ref{assumption 2.2}, Lemmas \ref{bound of X} and  \ref{continuous of X} and \eqref{sigma-3}, we have
		\begin{align}
			\E&\Big[\int_{0}^{t}\Big|\partial_l \sigma^{k_1}_j(X^n_{\frac{[ns]}{n}})\partial_{l'}\sigma^{k_2}_{j}(X^n_{\frac{[ns]}{n}})\sigma_i^l(X_{\frac{[ns]}{n}})\sigma_{i}^{l'}(X_{\frac{[ns]}{n}})-\sigma^l_i(X_{s})\sigma^{l'}_{i}(X_{s})\partial_l\sigma^{k_1}_j(X_{s})\partial_{l'}\sigma^{k_2}_{j}(X_{s})\Big|^2\d s\Big]\nonumber\\&\leq \E\Big[\int_{0}^{t}\Big|\sigma^l_i(X_{\frac{[ns]}{n}})\sigma^{l'}_{i}(X_{\frac{[ns]}{n}})-\sigma^l_i(X_{s})\sigma^{l'}_{i}(X_{s})\Big|^2\Big|\partial_l\sigma^{k_1}_j(X^n_{\frac{[ns]}{n}})\partial_{l'}\sigma^{k_2}_{j}(X^n_{\frac{[ns]}{n}})\Big|^2\d s\Big]\nonumber\\&\quad+\E\Big[\int_{0}^{t}|\sigma^l_i(X_{u})\sigma^{l'}_{i}(X_{u})|^2\Big|\partial_l\sigma^{k_1}_j(X^n_{\frac{[ns]}{n}})\partial_{l'}\sigma^{k_2}_{j}(X^n_{\frac{[ns]}{n}})-\partial_l\sigma^{k_1}_j(X_{s})\partial_{l'}\sigma^{k_2}_{j}(X_{s})\Big|^2\d s\Big]\nonumber
			\\&\leq C\int_{0}^{t}\E\Big[\Big|\sigma^l_i(X_{\frac{[ns]}{n}})\sigma^{l'}_{i}(X_{\frac{[ns]}{n}})-\sigma^l_i(X_{s})\sigma^{l'}_{i}(X_{s})\Big|^2\Big]\d s\nonumber\\&\quad+ \int_{0}^{t}\E\Big[|\sigma^l_i(X_{u})\sigma^{l'}_{i}(X_{u})|^2\Big|\partial_l\sigma^{k_1}_j(X^n_{\frac{[ns]}{n}})\partial_{l'}\sigma^{k_2}_{j}(X^n_{\frac{[ns]}{n}})-\partial_l\sigma^{k_1}_j(X_{s})\partial_{l'}\sigma^{k_2}_{j}(X_{s})\Big|^2\Big]\d s\nonumber\\&\leq C\big(\|\sigma^l_i(X_{\frac{[ns]}{n}})\|^2_{L^4}+\|\sigma^{l'}_{i}(X_{s})\|^2_{L^4}\big)\big\|X_{\frac{[ns]}{n}}-X_{s}\big\|^2_{L^4}\nonumber\\&\quad+C(n^{-4H}+n^{-2H})\leq C(n^{-4H}+n^{-2H})\to 0 \quad \text{as $n\to\infty$},\nonumber
		\end{align}
		where the boundness of the derivatives of $\sigma$ are used.  Then Lemma \ref{limit distribution of stochastic integral} yields
		\begin{align}
			&\mathcal{O}^n_{3611}\xrightarrow {  L^2}\nonumber\\& \frac{1}{2HG^2}\int_{0}^{1}\int_{-\infty}^{r}\nu(z,r)^2\d z\d r\int_{0}^{t}\sigma^l_i(X_{u})\sigma^{l'}_{i}(X_{u})\partial_l\sigma^{k_1}_j(X_{u})\partial_{l'}\sigma^{k_2}_{j}(X_{u})\d u.\nonumber
		\end{align}
		For the term $\mathcal{O}^n_{3612}$, Lemma \ref{small time estimate 2} shows that $\mathcal{O}^n_{361}\xrightarrow {  L^2}0$. 
		
		Using the change of variable again, $z=[ns]-nu$, we obtain
		\begin{align}\label{K-2}
			&n^{2H}\int_{\frac{[ns]}{n}}^{s}\Big(K(s-u)-K(\frac{[ns]}{n}-u)\Big)^2|n\delta_{(n,u)}|^{2H}\d u\nonumber\\&=\frac{1}{G}\int_{[ns]-ns}^{0}\mu(z,ns-[ns])^2(z+[-z])^{2H}\d z.
		\end{align}
		By Lemma \ref{limit distribution of stochastic integral}, \eqref{limit distribution-1} and \eqref{K-2} we have
		\begin{align}
			&\mathcal{O}^n_{362}\xrightarrow {  L^2 } -\frac{1}{2HG^2}\int_{0}^{1}\int_{-r}^{0}\mu(z,r)^2(z+[-z])^{2H}\d z\d r\int_{0}^{t}\sigma^l_i(X_{u})\sigma^{l'}_{i}(X_{u})\partial_l\sigma^{k_1}_j(X_{u})\partial_{l'}\sigma^{k_2}_{j}(X_{u})\d u.\nonumber
		\end{align}
		This concludes the proof of the lemma. 
	\end{proof}
	\begin{lem}\label{37}
		For any $i,i',j=1,\cdots,m;l,l'=1,\cdots,d,$
		$$\quad  \lim_{n\to\infty}\sup_{t\in[0,T]}\|\mathcal{O}^{n}_{37}\|_{L^2}=0,$$
		$$\quad  \lim_{n\to\infty}\sup_{t\in[0,T]}\|\mathcal{O}^{n}_{38}\|_{L^2}=0.$$
	\end{lem}
	\begin{proof}
		Let
		\begin{align}\label{D-3}
			D^{i ,i',l,l'}_{3,u}:=n^{2H}\Big(\int_{0}^{\frac{[nu]}{n}}g_{li}(u,v)\d W^i_v\Big)\sigma^{l'}_{i'}(X^n_{\frac{[nu]}{n}})\int_{\frac{[nu]}{n}}^{u}K(u-v)\d W^{i'}_v.
		\end{align}
		By   \eqref{M u 2}, \eqref{g li} and Fubini's theorem we have
		\begin{align}
			&\mathcal{O}^n_{37}= n^{4H}\int_{0}^{t}\int_{0}^{\frac{[ns]}{n}}f_{k_1jl}(s,u)f_{k_2jl'}(s,u)M_u^{1,n,l}M_u^{2,n,l'}\d u\d s\nonumber\\&\quad= \sum_{i,i'=1}^{m}n^{4H}\int_{0}^{t}\int_{0}^{\frac{[ns]}{n}}f_{k_1jl'}(s,u)f_{k_2jl'}(s,u)\Big(\int_{0}^{\frac{[nu]}{n}}g_{li}(u,v)\d W^i_v\Big)\sigma^{l'}_{i'}(X^n_{\frac{[nu]}{n}})\int_{\frac{[nu]}{n}}^{u}K(u-v)\d W^{i'}_v\d u\d s\nonumber\\&\quad=\sum_{i,i'=1}^{m}n^{2H}\int_{0}^{t}\int_{0}^{\frac{[ns]}{n}}f_{k_1jl}(s,u)f_{k_2jl'}(s,u)D^{i ,i',l,l'}_{3,u}\d u\d s\nonumber\\&\quad=\sum_{i,i'=1}^{m}n^{2H}\int_{0}^{t}\int_{0}^{s}f_{k_1jl}(s,u)f_{k_2jl'}(s,u)D^{i, i',l,l'}_{3,u}\d u\d s-\sum_{i,i'=1}^{m}n^{2H}\int_{0}^{t}\int_{\frac{[ns]}{n}}^{s}f_{k_1jl}(s,u)f_{k_2jl'}(s,u)D^{i, i',l,l'}_{3,u}\d u\d s\nonumber\\&\quad=\sum_{i,i'=1}^{m}n^{2H}\int_{0}^{t}D^{i, i',l,l'}_{3,u}\int_{0}^{u}f_{k_1jl}(s,u)f_{k_2jl'}(s,u)\d s\d u-\sum_{i,i'=1}^{m}n^{2H}\int_{0}^{t}\int_{\frac{[ns]}{n}}^{s}f_{k_1jl}(s,u)f_{k_2jl'}(s,u)D^{i ,i',l,l'}_{3,u}\d u\d s\nonumber\\&\quad:=\mathcal{O}^n_{371}+\mathcal{O}^n_{372}.\nonumber
		\end{align}
		By \eqref{dcp of stochastic integral}, \eqref{g li}, \eqref{def of D 1}, \eqref{def of D 2}, the linear growth properties of $\sigma$, the Cauchy-Schwarz inequality, Minkowski's   inequality, Lemmas \ref{bound of X n},  \ref{small time estimate},   \ref{esti of D 1},  and  \ref{esti of D 2}, we have
		\begin{align}\label{bound of D-3}
			&\quad \E[|D^{i ,i',l,l'}_{3,u}|^2]\nonumber\\&=n^{4H}\E\Big[\Big(\int_{0}^{\frac{[nu]}{n}}g_{li}(u,v)\d W^i_v\Big)^2\Big( \int_{\frac{[nu]}{n}}^{u}K(u-v)\sigma^{l'}_{i'}(X^n_{\frac{[nv]}{n}})\d W^{i'}_v\Big)^2\Big]\nonumber\\&=n^{4H}\E\Big[\Big(\int_{0}^{\frac{[nu]}{n}}\int_{0}^{v}g_{li}(u,r)\d W_r^{i}g_{li}(u,v)\d W^i_v+\int_{0}^{\frac{[nu]}{n}}\int_{0}^{r}g_{li}(u,v)\d W_v^{i}g_{li}(u,r)\d W^i_r\nonumber\\&\quad+\int_{0}^{\frac{[nu]}{n}}|g_{li}(u,v)|^2\d v\Big)\Big(\int_{\frac{[nu]}{n}}^{u}\int_{\frac{[nu]}{n}}^{v}K(u-r)\sigma^{l'}_{i'}(X^n_{\frac{[nr]}{n}})\d W^{i'}_rK(u-v)\sigma^{l'}_{i'}(X^n_{\frac{[nv]}{n}})\d W^{i'}_v\nonumber\\&\quad+\int_{\frac{[nu]}{n}}^{u}\int_{\frac{[nu]}{n}}^{r}K(u-v)\sigma^{l'}_{i'}(X^n_{\frac{[nv]}{n}})\d W^{i'}_vK(u-r)\sigma^{l'}_{i'}(X^n_{\frac{[nr]}{n}})\d W^{i'}_r+\int_{\frac{[nu]}{n}}^{u}K(u-v)^2|\sigma^{l'}_{i'}(X^n_{\frac{[nv]}{n}})|^2\d v\Big)\Big]\nonumber\\&\leq C\E\Big[\Big(D^{i, i,l,l'}_{1,u}+n^{2H}\int_{0}^{\frac{[nu]}{n}}|g_{li}(u,v)|^2\d v\Big)\Big(D^{i, i,l,l'}_{2,u}+n^{2H}\int_{\frac{[nu]}{n}}^{u}K(u-v)^2|\sigma^{l'}_{i'}(X^n_{\frac{[nv]}{n}})|^2\d v\Big)\Big]\nonumber\\&\leq C\Big(\|D^{i ,i,l,l'}_{1,u}\|_{L^2}\|D^{i ,i,l,l'}_{2,u}\|_{L^2}+n^{2H}\int_{\frac{[nu]}{n}}^{u}K(u-v)^2\|\sigma^{l'}_{i'}(X^n_{\frac{[nv]}{n}})\|^2_{L^4}\d v\|D^{i, i,l,l'}_{1,u}\|_{L^2}\nonumber\\&\quad+\int_{0}^{\frac{[nu]}{n}}\Big(K(u-v)-K(\frac{[nu]}{n}-v)\Big)^2 \|\sigma^{l'}_{i'}(X^n_{\frac{[nv]}{n}})\|^2_{L^4}\d v\|D^{i, i,l,l'}_{1,u}\|_{L^2}\nonumber\\&\quad+\int_{0}^{\frac{[nu]}{n}}\Big(K(u-v)-K(\frac{[nu]}{n}-v)\Big)^2 \|\sigma^{l'}_{i'}(X^n_{\frac{[nv]}{n}})\|^2_{L^4}\d v\int_{\frac{[nv]}{n}}^{u}K(u-v)^2\|\sigma^{l'}_{i'}(X^n_{\frac{[nv]}{n}})\|^2_{L^4}\d v\Big)\nonumber\\&\leq C,
		\end{align}
		with $C$ independent of $n$.
		
		For the term $\mathcal{O}^n_{371}$, by \eqref{f kjl}, the boundness of the derivatives of $\sigma$, Fubini's theorem, Lemmas \ref{small time estimate} and   \ref{2-small time estimate}, we have
		\begin{align}
			\E[|\mathcal{O}^n_{371}|^2]&\leq C \int_{0}^{t}\int_{0}^{s}\E[ D^{i ,i',l,l'}_{3,s} D^{i ,i',l,l'}_{3,z}]\d z \d s\nonumber\\&\leq C\int_{0}^{t}\int_{0}^{\frac{[ns]}{n}}\E[ D^{i ,i',l,l'}_{3,s} D^{i ,i',l,l'}_{3,z}]\d z \d s+C\int_{0}^{t}\int_{\frac{[ns]}{n}}^{s}\E[ D^{i ,i',l,l'}_{3,s} D^{i, i',l,l'}_{3,z}]\d z \d s\nonumber\\&:=\mathcal{O}^n_{3711}+\mathcal{O}^n_{3712}.\nonumber
		\end{align}
		For the term $\mathcal{O}^n_{3711}$, by \eqref{D-3}, \eqref{bound of D-3} 
		%\footnote{\eqref{bound of D-3} appears later}
		and the tower property (see, e.g. \cite{BD}) we have
		\begin{align}
			\mathcal{O}^n_{3711}&= C\int_{0}^{t}\int_{0}^{\frac{[ns]}{n}}\E[ D^{i ,i',l,l'}_{3,s} D^{i, i',l,l'}_{3,z}]\d z \d s\nonumber\\&
		 =C\int_{0}^{t}\int_{0}^{\frac{[ns]}{n}}\E\Big[ n^{2H}\Big(\int_{0}^{\frac{[ns]}{n}}g_{li}(s,v)\d W^i_v\Big)\sigma^{l'}_{i'}(X^n_{\frac{[ns]}{n}}) n^{2H}\Big(\int_{0}^{\frac{[nz]}{n}}g_{li}(z,v')\d W^i_{v'}\Big) \nonumber\\
			&\quad  \cdot\sigma^{l'}_{i'}(X^n_{\frac{[nz]}{n}})\int_{\frac{[nz]}{n}}^{z}K(z-v')\d W^{i'}_{v'}\E\Big[\int_{\frac{[ns]}{n}}^{s}K(s-v)\d W^{i'}_v|\mathcal{F}_{\frac{[ns]}{n}}\Big]\Big]\d z\d s  \nonumber\\&
		  =0 .\nonumber
		\end{align} 
		For the term $\mathcal{O}^n_{3712}$, by \eqref{bound of D-3}, the Cauchy-Schwarz inequality and the dominated convergence theorem  we have
		\begin{align}
			\lim_{n\to\infty}\mathcal{O}^n_{3712}&\leq C\lim_{n\to\infty} \int_{0}^{t}\int_{0}^{t}\mathbb{I}_{( \frac{[ns]}{n},s)}(z)\d z \d s=0.\nonumber
		\end{align}
		Since the derivatives of $\sigma$ are bounded, by Fubini's theorem and \eqref{f kjl} we have
		\begin{align}
			\E[&|\mathcal{O}^n_{372}|^2] \nonumber\\ &\leq C n^{4H}\int_{0}^{t}\int_{0}^{s}\E\Big[\int_{\frac{[ns]}{n}}^{s}\Big(K(s-u)-K(\frac{[ns]}{n}-u)\Big)^2 D^{i ,i',l,l'}_{3,u}\d u\nonumber\\&\quad\cdot \int_{\frac{[nz]}{n}}^{z}\Big(K(z-w)-K(\frac{[nz]}{n}-w)\Big)^2 D^{i, i',l,l'}_{3,w}\d w\Big]\d z \d s\nonumber\\&\leq C \int_{0}^{t}\int_{0}^{\frac{[ns]}{n}}n^{4H}\int_{\frac{[ns]}{n}}^{s}\int_{\frac{[nz]}{n}}^{z}\Big(K(s-u)-K(\frac{[ns]}{n}-u)\Big)^2\nonumber\\&\quad\cdot \Big(K(z-w)-K(\frac{[nz]}{n}-w)\Big)^2 \E\Big[ D^{i ,i',l,l'}_{3,u}  D^{i ,i',l,l'}_{3,w}\Big]\d w\d u\d z \d s\nonumber\\&\quad+C \int_{0}^{t}\int_{\frac{[ns]}{n}}^{s}n^{4H}\int_{\frac{[ns]}{n}}^{s}\int_{\frac{[nz]}{n}}^{z}\Big(K(s-u)-K(\frac{[ns]}{n}-u)\Big)^2\nonumber\\&\quad\cdot \Big(K(z-w)-K(\frac{[nz]}{n}-w)\Big)^2 \E\Big[ D^{i, i',l,l'}_{3,u}  D^{i, i',l,l'}_{3,w}\Big]\d w\d u\d z \d s\nonumber\\&:=\mathcal{O}^n_{3721}+\mathcal{O}^n_{3722}.\nonumber
		\end{align}
		For the term $\mathcal{O}^n_{3721}$, note that $w\leq z\leq \frac{[ns]}{n}\leq u$, the tower property, \eqref{D-3} and \eqref{bound of D-3} provide
		\begin{align}\label{est of 3721}
			&\E\Big[D^{i ,i',l,l'}_{3,u}  D^{i, i',l,l'}_{3,w}\Big]=\E\Big[ n^{2H}\Big(\int_{0}^{\frac{[nu]}{n}}g_{li}(u,v)\d W^i_v\Big)\sigma^{l'}_{i'}(X^n_{\frac{[nu]}{n}})\int_{\frac{[nu]}{n}}^{u}K(u-v)\d W^{i'}_v\nonumber\\&\quad\quad\cdot n^{2H}\Big(\int_{0}^{\frac{[nw]}{n}}g_{li}(w,v')\d W^i_{v'}\Big)\sigma^{l'}_{i'}(X^n_{\frac{[nw]}{n}})\int_{\frac{[nw]}{n}}^{w}K(w-v')\d W^{i'}_{v'}\Big]\nonumber\\&\quad=\E\Big[ n^{2H}\Big(\int_{0}^{\frac{[nu]}{n}}g_{li}(u,v)\d W^i_v\Big)\sigma^{l'}_{i'}(X^n_{\frac{[nu]}{n}}) n^{2H}\Big(\int_{0}^{\frac{[nw]}{n}}g_{li}(w,v')\d W^i_{v'}\Big)\nonumber\\&\quad\quad\cdot\sigma^{l'}_{i'}(X^n_{\frac{[nw]}{n}})\int_{\frac{[nw]}{n}}^{w}K(w-v')\d W^{i'}_{v'}\E\Big[\int_{\frac{[nu]}{n}}^{u}K(u-v)\d W^{i'}_v|\mathcal{F}_{\frac{[nu]}{n}}\Big]\Big]\nonumber\\&\quad=0,
		\end{align}
		then Lemma \ref{2-small time estimate} gives that
		$$\mathcal{O}^n_{3721}\to0\quad \text{as $n\to\infty$}.$$
		For the term $\mathcal{O}^n_{3722}$, by \eqref{bound of D-3}, the Cauchy-Schwarz inequality, Lemma \ref{2-small time estimate} and the dominated convergence theorem, we obtain
		\begin{align}
			\lim_{n\to\infty}\mathcal{O}^n_{3722}&\leq C\lim_{n\to\infty} \int_{0}^{t}\int_{\frac{[ns]}{n}}^{s}n^{4H}\int_{\frac{[ns]}{n}}^{s}\int_{\frac{[nz]}{n}}^{z}\Big(K(s-u)-K(\frac{[ns]}{n}-u)\Big)^2\nonumber\\&\quad\cdot \Big(K(z-w)-K(\frac{[nz]}{n}-w)\Big)^2 \d w\d u\d z \d s\nonumber\\&\leq C\int_{0}^{t} \int_{0}^{s}\lim_{n\to\infty}\mathbb{I}_{( \frac{[ns]}{n},s)}(z)\d z \d s=0.\nonumber
		\end{align}
		From these computations  it follows  that $\lim_{n\to\infty}\|\mathcal{O}^n_{37}\|_{L^2}=0$.
		Similar to $\mathcal{O}^n_{37}$, it holds that $\mathcal{O}^n_{38}\to 0$ in $L^2$.
	\end{proof}
	
	\begin{lem}\label{1,2}
		For any $j,j'=1,\cdots,m;l,l'=1,\cdots,d,$
		\[
		\begin{split}
			\lim_{n\to\infty}\sup_{t\in[0,T]}\|\mathcal{O}^{n}_{1}\|_{L^2}=&0,\\
			\lim_{n\to\infty}\sup_{t\in[0,T]}\|\mathcal{O}^{n}_{2}\|_{L^2}=&0. 
		\end{split}
		\]
	\end{lem}
	\begin{proof}
		Denote 
		\begin{align}\label{D-4}
			D^{j ,j',l,l'}_{4,s}:=n^{4H}\int_{0}^{\frac{[ns]}{n}}\Big(\int_{0}^{u}f_{k_1jl}(s,r)M_r^{n,l}\d W_r^j\Big)f_{k_2j'l'}(s,u)M_u^{n,l'}\d W_u^{j'},
		\end{align}
		and by Fubini's theorem we see  that
		$$\E[|\mathcal{O}^n_1|^2]=\E\Big[\int_{0}^{t}\int_{0}^{t}D^{j, j',l,l'}_{4,s}D^{j, j',l,l'}_{4,v}\d v \d s\Big]=2\int_{0}^{t}\int_{0}^{s}\E[D^{j, j',l,l'}_{4,s}D^{j, j',l,l'}_{4,v}]\d v \d s.$$
		By \eqref{f kjl}, \eqref{dcp of stochastic integral}, Fubini's theorem and  the tower property, we have
		\begin{align}\label{def of double D-4}
			&\E[D^{j, j',l,l'}_{4,s}D^{j, j',l,l'}_{4,v}]\nonumber\\&=n^{8H}\E\Big[\int_{\frac{[nv]}{n}}^{\frac{[ns]}{n}}\Big(\int_{0}^{u}f_{k_1jl}(s,r)M_r^{n,l}\d W_r^j\Big) f_{k_2j'l'}(s,u)M_u^{n,l'}\d W_u^{j'}\nonumber\\&\quad\cdot \int_{0}^{\frac{[nv]}{n}}\Big(\int_{0}^{u}f_{k_1jl}(v,r)M_r^{n,l}\d W_r^j\Big) f_{k_2j'l'}(v,u)M_u^{n,l'}\d W_u^{j'}\Big]\nonumber\\&\quad+n^{8H}\E\Big[\int_{0}^{\frac{[nv]}{n}}\Big(\int_{0}^{u}f_{k_1jl}(s,r)M_r^{n,l}\d W_r^j\Big)\Big(\int_{0}^{u}f_{k_1jl}(v,r)M_r^{n,l}\d W_r^j\Big)\nonumber\\&\quad\cdot  f_{k_2j'l'}(s,u)f_{k_2j'l'}(v,u)|M_u^{n,l'}|^2\d u\Big]\nonumber\\&\leq n^{8H}\E\Big[\int_{0}^{\frac{[nv]}{n}}\Big(\int_{0}^{u}f_{k_1jl}(v,r)M_r^{n,l}\d W_r^j\Big) f_{k_2j'l'}(v,u)M_u^{n,l'}\d W_u^{j'}\nonumber\\&\quad\cdot \E\Big[\int_{\frac{[nv]}{n}}^{\frac{[ns]}{n}}\Big(\int_{0}^{u}f_{k_1jl}(s,r)M_r^{n,l}\d W_r^j\Big) f_{k_2j'l'}(s,u)M_u^{n,l'}\d W_u^{j'}|\mathcal{F}_{\frac{[nv]}{n}}\Big]\Big]\nonumber\\&\quad+n^{2H}\int_{0}^{\frac{[nv]}{n}}\Big(K(s-u)-K(\frac{[ns]}{n}-u)\Big)\Big(K(v-u)-K(\frac{[nv]}{n}-u)\Big)\E[D^{j,l,l'}_{5,u}]\d u\nonumber\\&=n^{2H}\int_{0}^{\frac{[nv]}{n}}\Big(K(s-u)-K(\frac{[ns]}{n}-u)\Big)\Big(K(v-u)-K(\frac{[nv]}{n}-u)\Big)\E[D^{j,l,l'}_{5,u}]\d u,
		\end{align}
		where
		\begin{align}\label{D-5}
			D^{j,l,l'}_{5,u}:=n^{6H}\Big(\int_{0}^{u}f_{k_1jl}(s,r)M_r^{n,l}\d W_r^j\Big)\Big(\int_{0}^{u}f_{k_1jl}(v,r)M_r^{n,l}\d W_r^j\Big)|M_u^{n,l'}|^2.
		\end{align}
		Since for any $m\geq 2, u\leq \frac{[ns]}{n}$ and $s\leq T$, it holds
		\begin{align}\label{esti of stochastic integral}
			\Big\|\int_{0}^{u}f_{k_1jl}(s,r)M_r^{n,l}\d W_r^j\Big\|_{L^m}&\leq C \Big\|\int_{0}^{u}|f_{k_1jl}(s,r)M_r^{n,l}|^2\d r\Big\|^{\frac{1}{2}}_{L^{m/2}}\nonumber\\&\leq C\Big\|\int_{0}^{\frac{[ns]}{n}}\Big(K(s-u)-K(\frac{[ns]}{n}-u)\Big)^2|M_r^{n,l}|^2\d r\Big\|^{\frac{1}{2}}_{L^{m/2}}\nonumber\\&\leq C\Big(\int_{0}^{\frac{[ns]}{n}}\Big(K(s-u)-K(\frac{[ns]}{n}-u)\Big)^2\||M_r^{n,l}|^2\|_{L^{m/2}}\d r\Big)^{1/2}\nonumber\\&\leq Cn^{-2H},
		\end{align}
		where \eqref{f kjl}, the boundness of the derivatives of $\sigma$, BDG's inequality, Minkowski's   inequality and Lemma \ref{bound of M n} are used.
		
		By the Cauchy-Schwarz inequality, Lemma \ref{bound of M n} and \eqref{esti of stochastic integral} we have
		\begin{align}\label{est of D-5}
			\E[D^{j,l,l'}_{5,u}]&\leq n^{6H}\Big\|\int_{0}^{u}f_{k_1jl}(s,r)M_r^{n,l}\d W_r^j\Big\|_{L^4}\Big\| \int_{0}^{u}f_{k_1jl}(v,r)M_r^{n,l}\d W_r^j\Big\|_{L^4}\|M_u^{n,l'}\|^2_{L^4}\nonumber\\&\leq C.
		\end{align}
		By Lemma \ref{esti of A} and \eqref{def of double D-4}  this implies 
		$$|\E[D^{j, j',l,l'}_{4,s}D^{j, j',l,l'}_{4,v}]|\leq C n^{2H}\int_{0}^{\frac{[nv]}{n}}\Big(K(s-u)-K(\frac{[ns]}{n}-u)\Big)\Big(K(v-u)-K(\frac{[nv]}{n}-u)\Big)\d u=0.$$
		Applying the dominated convergence theorem  with respect to $\d v\otimes \d s$, we have $$\mathcal{O}^n_{1}\xrightarrow {  L^2} 0.$$
		Similar to $\mathcal{O}^n_{1}$, it holds that $\mathcal{O}^n_{2}\to 0$ in $L^2$.
	\end{proof}
	\subsubsection{The $L^2$-limits of $I^{n,k_1,k_2}_{2,j,j',l,l'}$ and $I^{n,k_1,k_2}_{3,j,j',l,l'} $}\label{3.2.2}
	\begin{lem}\label{II, III}
		For any $j,j'=1,\cdots,m;l,l'=1,\cdots,d,$
		$$\quad  \lim_{n\to\infty}\sup_{t\in[0,T]}\|I^{n,k_1,k_2}_{2,j,j',l,l'}\|_{L^2}=0,$$
		$$\quad  \lim_{n\to\infty}\sup_{t\in[0,T]}\|I^{n,k_1,k_2}_{3,j,j',l,l'} \|_{L^2}=0.$$
	\end{lem}
	\begin{proof}
		Recall that
		\begin{align}
			I^{n,k_1,k_2}_{2,j,j',l,l'}=\int_{0}^{t}\partial_{l'}\sigma^{k_2}_{j'}(X^n_\frac{[ns]}{n})D^{j,j',l,l'}_{6,s}\d s,\nonumber
		\end{align}
		where
		$$D^{j,j',l,l'}_{6,s}=n^{4H}\Big(\int_{0}^{\frac{[ns]}{n}}f_{k_1jl}(s,u)M_u^{n,l}\d W_u^j\Big)\Big(\int_{\frac{[ns]}{n}}^{s}K(s-u)M_u^{n,l'}\d W_u^{j'}\Big).$$
		We will show $I^{n,k_1,k_2}_{2,j,j',l,l'}\xrightarrow[\text { in } L^2]{n \rightarrow \infty} 0$ for all $j,j',l,l'$. By Fubini's theorem we have
		\begin{align}
			&\E\Big[|I^{n,k_1,k_2}_{2,j,j',l,l'}|^2\Big]\nonumber\\&=\E\Big[\int_{0}^{t}\int_{0}^{t}\partial_{l'}\sigma^{k_2}_{j'}(X^n_\frac{[ns]}{n})\partial_{l'}\sigma^{k_2}_{j'}(X^n_\frac{[nv]}{n})D^{j,j',l,l'}_{6,s}D^{j,j',l,l'}_{6,v}\d v\d s\Big]\nonumber\\&=2\E\Big[\int_{0}^{t}\int_{0}^{\frac{[ns]}{n}}\partial_{l'}\sigma^{k_2}_{j'}(X^n_\frac{[ns]}{n})\partial_{l'}\sigma^{k_2}_{j'}(X^n_\frac{[nv]}{n})D^{j,j',l,l'}_{6,s}D^{j,j',l,l'}_{6,v}\d v\d s\Big]\nonumber\\&\quad+2\E\Big[\int_{0}^{t}\int_{\frac{[ns]}{n}}^{s}\partial_{l'}\sigma^{k_2}_{j'}(X^n_\frac{[ns]}{n})\partial_{l'}\sigma^{k_2}_{j'}(X^n_\frac{[nv]}{n})D^{j,j',l,l'}_{6,s}D^{j,j',l,l'}_{6,v}\d v\d s\Big]\nonumber\\&:=\mathcal{P}^n_1+\mathcal{P}^n_2.\nonumber
		\end{align}
		Let 
		\begin{align}\label{D-7}
			D^{j,j',l}_{7,s}:=n^{4H}\int_{\frac{[ns]}{n}}^{s}\int_{\frac{[ns]}{n}}^{u}f_{k_1jl}(s,r)M_r^{n,l}\d W^{j'}_rf_{k_1jl}(s,u)M_u^{n,l}\d W^{j'}_u.
		\end{align}
		Since the derivatives of $\sigma$ are bounded, by \eqref{f kjl}, \eqref{dcp of stochastic integral}, Fubini's theorem,\eqref{D-4}, \eqref{D-5}, \eqref{est of D-5} and Lemma \ref{2-small time estimate},  we have
		\begin{align}\label{est of D-7}
			&\E[|D^{j, j',l}_{7,s}|^2]\nonumber\\&=n^{8H}\E\Big[\int_{\frac{[ns]}{n}}^{s}\Big(\int_{\frac{[ns]}{n}}^{u}f_{k_1jl}(s,r)M_r^{n,l}\d W_r^{j'}\Big)^2  |f_{k_1jl}(s,u)|^2|M_u^{n,l}|^2\d u\Big]\nonumber\\&\leq n^{8H}\int_{\frac{[ns]}{n}}^{s}\Big(K(s-u)-K(\frac{[ns]}{n}-u)\Big)^2\E\Big[\Big(\int_{\frac{[ns]}{n}}^{u}f_{k_1jl}(s,r)M_r^{n,l}\d W_r^{j'}\Big)^2|M_u^{n,l}|^2\Big]\d u\nonumber\\&=n^{2H}\int_{\frac{[ns]}{n}}^{s}\Big(K(s-u)-K(\frac{[ns]}{n}-u)\Big)^2\E[D^{j,l,l'}_{5,u}]\d u\leq C.
		\end{align}
		By \eqref{f kjl}, \eqref{dcp of stochastic integral}, \eqref{D-4}, the linear growth properties of $\sigma$, the Cauchy-Schwarz inequality, Minkowski's inequality, Lemmas \ref{bound of X n},  \ref{bound of M n}, \ref{small time estimate} and   \ref{2-small time estimate}, we have
		\begin{align}\label{bound of D-6}
			&\E[|D^{j,j',l,l'}_{6,s}|^2]\nonumber\\&=n^{8H}\E\Big[\Big(\int_{0}^{\frac{[ns]}{n}}f_{k_1jl}(s,u)M_u^{n,l}\d W_u^j\Big)^2\Big( \int_{\frac{[ns]}{n}}^{s}K(s-u)M_u^{n,l'}\d W_u^{j'}\Big)^2\Big]\nonumber\\&=n^{8H}\E\Big[\Big(\int_{0}^{\frac{[ns]}{n}}\int_{0}^{u}f_{k_1jl}(s,r)M_r^{n,l}\d W_r^{j}f_{k_1jl}(s,u)M_u^{n,l}\d W^j_u\nonumber\\&\quad+\int_{0}^{\frac{[ns]}{n}}\int_{0}^{r}f_{k_1jl}(s,u)M_u^{n,l}\d W_u^{j}f_{k_1jl}(s,r)M_r^{n,l}\d W^j_r\nonumber\\&\quad+\int_{0}^{\frac{[ns]}{n}}|f_{k_1 jl}(s,u)M_u^{n,l}|^2\d u\Big)\Big(\int_{\frac{[ns]}{n}}^{s}\int_{\frac{[ns]}{n}}^{u}f_{k_1jl}(s,r)M_r^{n,l}\d W^{j'}_rf_{k_1jl}(s,u)M_u^{n,l}\d W^{j'}_u\nonumber\\&\quad+\int_{\frac{[ns]}{n}}^{s}\int_{\frac{[ns]}{n}}^{r}f_{k_1jl}(s,u)M_u^{n,l}\d W^{j'}_uf_{k_1jl}(s,r)M_r^{n,l}\d W^{j'}_r+\int_{\frac{[ns]}{n}}^{s}|f_{k_1jl}(s,u)M_u^{n,l}|^2\d u\Big)\Big]\nonumber\\&\leq C\E\Big[\Big(D^{j,j',l,l'}_{4,s}\Big\vert^{k_1=k_2}_{j=j',l=l'}+n^{4H}\int_{0}^{\frac{[ns]}{n}}|f_{k_1 jl}(s,u)M_u^{n,l}|^2\d u\Big)\Big(D^{j,j',l}_{7,s}+n^{4H}\int_{\frac{[ns]}{n}}^{s}|f_{k_1jl}(s,u)M_u^{n,l}|^2\d u\Big)\Big]\nonumber\\&\leq C+Cn^{4	H}\int_{0}^{\frac{[ns]}{n}}\Big(K(s-u)-K(\frac{[ns]}{n}-u)\Big)^2\E[|M_u^{n,l}|^2]\d u\nonumber\\&\quad+Cn^{4	H}\int_{\frac{[ns]}{n}}^{s}\Big(K(s-u)-K(\frac{[ns]}{n}-u)\Big)^2\E[|M_u^{n,l}|^2]\d u\leq C,
		\end{align}
		where $C$ is a positive constant which does not depend on $n$.
		The inequalities  \eqref{D-4}, \eqref{def of double D-4}, \eqref{est of D-5} and Lemma \ref{esti of A} yield
		\begin{align}
			\E[|D_{4,s}^{j,j',l,l'}|^2]\leq C.\nonumber
		\end{align}
	The inequality  \eqref{bound of D-6} gives
		$$\E[D^{j,j',l,l'}_{6,s}|\mathcal{F}_{\frac{[ns]}{n}}]=n^{4H}\Big(\int_{0}^{\frac{[ns]}{n}}f_{k_1jl}(s,u)M_u^{n,l}\d W_u^j\Big)\E\Big[\int_{\frac{[ns]}{n}}^{s}K(s-u)M_u^{n,l'}\d W_u^{j'}|\mathcal{F}_{\frac{[ns]}{n}}\Big]=0.$$
		Now the tower property and Fubini's theorem provide
		$$\mathcal{P}^n_1=\int_{0}^{t}\int_{0}^{\frac{[ns]}{n}}\E\Big[\partial_{l'}\sigma^{k_2}_{j'}(X^n_\frac{[ns]}{n})\partial_{l'}\sigma^{k_2}_{j'}(X^n_\frac{[nv]}{n})D^{j,j',l,l'}_{6,v}\E[D^{j,j',l,l'}_{6,s}|\mathcal{F}_{\frac{[ns]}{n}}]\Big]\d v\d s=0.$$
		For the term $\mathcal{P}^n_2$, the boundedness of the derivatives of $\sigma$, Fubini's theorem, \eqref{bound of D-6} and the dominated convergence theorem  yield that
		\begin{align}
			\lim_{n\to\infty}\mathcal{P}^n_2=\int_{0}^{t}\int_{0}^{t}\mathbb{I}_{(\frac{[ns]}{n},s)}(v)E\Big[\partial_{l'}\sigma^{k_2}_{j'}(X^n_\frac{[ns]}{n})\partial_{l'}\sigma^{k_2}_{j'}(X^n_\frac{[nv]}{n})D^{j,j',l,l'}_{6,v}D^{j,j',l,l'}_{6,s}\Big]\d v\d s=0.\nonumber
		\end{align}
		Therefore $I^{n,k_1,k_2}_{2,j,j',l,l'}\to 0$ in $L^2$. Similar to $I^{n,k_1,k_2}_{2,j,j',l,l'}$, it holds that $I^{n,k_1,k_2}_{3,j,j',l,l'} \to 0$ in $L^2$.
	\end{proof}
	\subsubsection{The $L^2$-limit of $I^{n,k_1,k_2}_{4,j,j',l,l'}$}\label{3.2.3}
	According to \eqref{dcp of stochastic integral}, by \eqref{def of E} we have
	\begin{align}
		I^{n,k_1,k_2}_{4,j,j',l,l'}&=n^{4H}\int_{0}^{t}\Big(\int_{\frac{[ns]}{n}}^{s}K(s-u)M_u^{n,l}\d W_u^{j}\Big)\Big(\int_{\frac{[ns]}{n}}^{s}K(s-u)M_u^{n,l'}\d W_u^{j'}\Big)\d s\nonumber\\&=n^{4H}\int_{0}^{t}E^{j,j',l,l'}_s\Big(\int_{\frac{[ns]}{n}}^{s}\Big(\int_{\frac{[ns]}{n}}^{u}K(s-r)M_r^{n,l'}\d W_r^{j'}\Big)K(s-u)M_u^{n,l}\d W_u^{j}\Big)\d s\nonumber\\&\quad+n^{4H}\int_{0}^{t}E^{j,j',l,l'}_s\Big(\int_{\frac{[ns]}{n}}^{s}\Big(\int_{\frac{[ns]}{n}}^{u}K(s-r)M_r^{n,l}\d W_r^{j}\Big)K(s-u)M_u^{n,l'}\d W_u^{j'}\Big)\d s\nonumber\\&\quad+n^{4H}\int_{0}^{t}E^{j,j',l,l'}_s\int_{\frac{[ns]}{n}}^{s}|K(s-u)|^2M_u^{n,l}M_u^{n,l'}\d \langle W^{j}, W^{j'} \rangle_u\d s\nonumber\\&:=\mathcal{Q}_1+\mathcal{Q}_2+\mathcal{Q}_3.\nonumber
	\end{align}
	Note that the last term vanishes if $j\neq j'$, and if $j=j'$, by \eqref{dcp of M-2}, \eqref{def of F} and \eqref{def of D 1} we have
	\begin{align}
		\mathcal{Q}^n_3&=n^{4H}\int_{0}^{t}E^{j,j,l,l'}_s\int_{\frac{[ns]}{n}}^{s}|K(s-u)|^2M_u^{n,l}M_u^{n,l'}\d u\d s\nonumber\\&=n^{2H}\sum_{i,i'=1}^{m}\int_{0}^{t}E^{j,j,l,l'}_s\int_{\frac{[ns]}{n}}^{s}|K(s-u)|^2D^{i ,i',l,l'}_{1,u}\d u\d s\nonumber\\&\quad+n^{2H}\sum_{i,i'=1}^{m}\int_{0}^{t}E^{j,j,l,l'}_s\int_{\frac{[ns]}{n}}^{s}|K(s-u)|^2D^{i ',i,l',l}_{1,u}\d u\d s\nonumber\\&\quad+n^{4H}\sum_{i=1}^{m}\int_{0}^{t}E^{j,j,l,l'}_s\int_{\frac{[ns]}{n}}^{s}|K(s-u)|^2\int_{0}^{\frac{[nu]}{n}}g_{li}(u,v)g_{l'i}(u,v)\d v\d u\d s\nonumber\\&\quad+n^{2H}\sum_{i,i'=1}^{m}\int_{0}^{t}E^{j,j,l,l'}_s\int_{\frac{[ns]}{n}}^{s}|K(s-u)|^2D^{i ,i',l,l'}_{2,u}\d u\d s\nonumber\\&\quad+n^{2H}\sum_{i,i'=1}^{m}\int_{0}^{t}E^{j,j,l,l'}_s\int_{\frac{[ns]}{n}}^{s}|K(s-u)|^2D^{i ',i,l,l'}_{2,u}\d u\d s\nonumber\\&\quad+n^{4H}\sum_{i=1}^{m}\int_{0}^{t}E^{j,j,l,l'}_s\int_{\frac{[ns]}{n}}^{s}|K(s-u)|^2F^{i,i',l,l'}_u\int_{\frac{[nu]}{n}}^{u}K(u-v)^2\d v\d u\d s\nonumber\\&\quad+n^{4H}\int_{0}^{t}E^{j,j,l,l'}_s\int_{\frac{[ns]}{n}}^{s}|K(s-u)|^2M_u^{1,n,l}M_u^{2,n,l'}\d u\d s\nonumber\\&\quad+n^{4H}\int_{0}^{t}E^{j,j,l,l'}_s\int_{\frac{[ns]}{n}}^{s}|K(s-u)|^2M_u^{2,n,l}M_u^{1,n,l'}\d u\d s\nonumber\\&:=\mathcal{Q}^n_{31}+\mathcal{Q}^n_{32}+\sum_{i=1}^{m}\mathcal{Q}^n_{33}+\mathcal{Q}^n_{34}+\mathcal{Q}^n_{35}+\sum_{i=1}^{m}\mathcal{Q}^n_{36}+\mathcal{Q}^n_{37}+\mathcal{Q}^n_{38}.\nonumber
	\end{align}
	\begin{lem}\label{Q-31, 32}
		For any $i,i',j=1,\cdots,m;l,l'=1,\cdots,d,$
		$$(i)\quad  \lim_{n\to\infty}\sup_{t\in[0,T]}\|\mathcal{Q}^{n}_{31}\|_{L^2}=0,$$
		$$(ii)\quad  \lim_{n\to\infty}\sup_{t\in[0,T]}\|\mathcal{Q}^{n}_{32}\|_{L^2}=0.$$
	\end{lem}
	\begin{proof}
		Since the derivatives of $\sigma$ are bounded, we have that
		\begin{align}
			\|\mathcal{Q}^n_{31}\|_{L^2}&\leq \sum_{i,i'=1}^{m}\Big\|\int_{0}^{t}D^{i ,i',l,l'}_{1,s}\d s\Big\|_{L^2}\sup_{s\in[0,T]}n^{2H}\int_{\frac{[ns]}{n}}^{s}K(s-u)^2\d u\nonumber\\&\leq C\sum_{i,i'=1}^{m}\Big\|\int_{0}^{t}D^{i ,i',l,l'}_{1,s}\d s\Big\|_{L^2}\nonumber\\&\leq 0 \quad \text{as $n\to\infty$},\nonumber
		\end{align}
		by \eqref{def of E}, Minkowski's inequality, Fubini's theorem, Lemmas \ref{small time estimate} and   \ref{esti of D 1}. Applying the dominated convergence theorem  with respect to $\d u\otimes \d s$, we have $\mathcal{Q}^n_{31} \xrightarrow[\text { in } L^2]{n \rightarrow \infty} 0$.
		
		Similar to $\mathcal{Q}^n_{31}$, it holds that $\mathcal{Q}^n_{32}\to 0$ in $L^2$.	
	\end{proof}
	\begin{lem}\label{Q-33}
		For any $i,j=1,\cdots,m;l,l'=1,\cdots,d,$  we have 
		\begin{align}
			\mathcal{Q}^n_{33}&\xrightarrow[\text { in } L^2]{n \rightarrow \infty} C_G^{(3)}\int_{0}^{t}\sigma^l_i(X_{s})\sigma^{l'}_{i}(X_{s})\partial_l\sigma^{k_1}_j(X_{s})\partial_{l'}\sigma^{k_2}_{j}(X_{s})\d s,\nonumber
		\end{align}
		where
		\begin{align}
			C_G^{(3)}&:=\frac{1}{2HG^2(4H+1)}\int_{0}^{\infty}\mu(r,1)^2\d r.\nonumber
		\end{align}
	\end{lem}
	\begin{proof}
		By \eqref{transfer of g}, Lemma \ref{small time estimate} and Fubini's theorem, we have
		\begin{align}
			\mathcal{Q}^n_{33}&=n^{4H}\int_{0}^{t}E^{j,j,l,l'}_s\int_{\frac{[ns]}{n}}^{s}|K(s-u)|^2\int_{0}^{\frac{[nu]}{n}}g_{li}(u,v)g_{l'i}(u,v)\d v\d u\d s\nonumber\\&=n^{4H}\int_{0}^{t}E^{j,j,l,l'}_s\int_{0}^{\frac{[ns]}{n}}g_{li}(u,v)g_{l'i}(u,v)\d v\int_{\frac{[ns]}{n}}^{s}|K(s-u)|^2\d u\d s\nonumber\\&=\frac{1}{2HG}n^{4H}\int_{0}^{t}|\delta_{(n,s)}|^{4H}E^{j,j,l,l'}_s\tilde{g}^n_s\d u\d s.\nonumber
		\end{align}
		Lemma \ref{limit distribution of stochastic integral} is applied here. 	By \eqref{back need}, it is clear that 
		\begin{align}
			\tilde{g}^n_sE^{j,j,l,l'}_s&\xrightarrow[\text { in } L^2(\d u\otimes \d P)]{n \rightarrow \infty} \frac{1}{G}\sigma^l_i(X_{s})\sigma^{l'}_{i}(X_{s})\int_{0}^{\infty}\mu(r,1)^2\d r\partial_l\sigma^{k_1}_j(X_{s})\partial_{l'}\sigma^{k_2}_{j}(X_{s})\,. 
		\end{align}
	  Then,   the  assumptions  in  Lemma \ref{limit distribution of stochastic integral}   are satisfied and an application of this lemma yields
		\begin{align}
			\mathcal{Q}^n_{33}\xrightarrow[\text { in } L^2]{n \rightarrow \infty} \frac{1}{2HG^2(4H+1)}\int_{0}^{\infty}\mu(r,1)^2\d r\int_{0}^{t}\sigma^l_i(X_{s})\sigma^{l'}_{i}(X_{s})\partial_l\sigma^{k_1}_j(X_{s})\partial_{l'}\sigma^{k_2}_{j}(X_{s})\d s.\nonumber
		\end{align}
		The proof is complete.
	\end{proof}
	\begin{lem}\label{Q-34 35}
		For any $i,i',j=1,\cdots,m;l,l'=1,\cdots,d,$
		$$(i)\quad  \lim_{n\to\infty}\sup_{t\in[0,T]}\|\mathcal{Q}^{n}_{34}\|_{L^2}=0,$$
		$$(ii)\quad  \lim_{n\to\infty}\sup_{t\in[0,T]}\|\mathcal{Q}^{n}_{35}\|_{L^2}=0.$$
	\end{lem}
	\begin{proof}
		Since the derivatives of $\sigma$ are bounded, we see 
		\begin{align}
			\|\mathcal{Q}^n_{34}\|_{L^2}&\leq \sum_{i,i'=1}^{m}\Big\|\int_{0}^{t}D^{i ,i',l,l'}_{2,s}\d s\Big\|_{L^2}\sup_{s\in[0,T]}n^{2H}\int_{\frac{[ns]}{n}}^{s}K(s-u)^2\d u\nonumber\\&\leq C\sum_{i,i'=1}^{m}\Big\|\int_{0}^{t}D^{i ,i',l,l'}_{2,s}\d s\Big\|_{L^2}\nonumber\\&\leq 0 \quad \text{as $n\to\infty$},\nonumber
		\end{align}
		by \eqref{def of E}, Minkowski's inequality, Fubini's theorem, Lemmas  \ref{small time estimate} and   \ref{esti of D 2}. 
		%Applying the dominated convergence theorem  with respect to $\d u\otimes \d s$, we have $\mathcal{Q}^n_{34} \xrightarrow[\text { in } L^2]{n \rightarrow \infty} 0$.
		
		Similar to $\mathcal{Q}^n_{34}$, it holds that $\mathcal{Q}^n_{35}\to 0$ in $L^2$. 
	\end{proof}
	\begin{lem}\label{Q-36}
		For any $i,j=1,\cdots,m;l,l'=1,\cdots,d,$  we have 
		\begin{align}
			\mathcal{Q}^n_{36}&\xrightarrow[\text { in } L^2]{n \rightarrow \infty} C^{(4)}_G\int_{0}^{t}\sigma^l_i(X_{s})\sigma^{l'}_{i}(X_{s})\partial_l\sigma^{k_1}_j(X_{s})\partial_{l'}\sigma^{k_2}_{j}(X_{s})\d s,\nonumber
		\end{align}
		where
		\begin{align}
			C_G^{(4)}&:=\frac{1}{2HG^2}\int_{0}^{1}\int_{-r}^{0}(z+r)^{2H-1}(z+[-z])^{2H}\d z\d r.\nonumber
		\end{align}
	\end{lem}
	\begin{proof}
		By Lemma \ref{small time estimate} we have
		\begin{align}
			\mathcal{Q}^n_{33}&=n^{4H}\int_{0}^{t}E^{j,j,l,l'}_s\int_{\frac{[ns]}{n}}^{s}|K(s-u)|^2F^{i,i,l,l'}_u\int_{\frac{[nu]}{n}}^{u}K(u-v)^2\d v\d u\d s\nonumber\\&=n^{4H}\int_{0}^{t}E^{j,j,l,l'}_sF^{i,i,l,l'}_s\int_{\frac{[ns]}{n}}^{s}|K(s-u)|^2\int_{\frac{[nu]}{n}}^{u}K(u-v)^2\d v\d u\d s\nonumber\\&=\frac{1}{2HG}n^{4H}\int_{0}^{t}E^{j,j,l,l'}_sF^{i,i,l,l'}_s\int_{\frac{[ns]}{n}}^{s}|K(s-u)|^2|\delta_{(n,u)}|^{2H}\d u\d s.\nonumber
		\end{align}
		Using the change of variable $z=[ns]-nu$, we obtain
		\begin{align}\label{K-3}
			n^{2H}\int_{\frac{[ns]}{n}}^{s}K(s-u)^2|n\delta_{(n,u)}|^{2H}\d u=\frac{1}{G}\int_{[ns]-ns}^{0}(z+ns-[ns])^{2H-1}(z+[-z])^{2H}\d z.
		\end{align}
		By  \eqref{limit distribution-1},  \eqref{K-3},  and Lemma \ref{limit distribution of stochastic integral}  we have
		\begin{align}
			&\mathcal{O}^n_{362}\xrightarrow[\text { in } L^2]{n \rightarrow \infty} \frac{1}{2HG^2}\int_{0}^{1}\int_{-r}^{0}(z+r)^{2H-1}(z+[-z])^{2H}\d z\d r\int_{0}^{t}\sigma^l_i(X_{s})\sigma^{l'}_{i}(X_{s})\partial_l\sigma^{k_1}_j(X_{s})\partial_{l'}\sigma^{k_2}_{j}(X_{s})\d s.\nonumber
		\end{align}
		The proof is complete.
	\end{proof}
	\begin{lem}\label{Q-37}
		For any $j=1,\cdots,m;l,l'=1,\cdots,d,$
		$$\quad  \lim_{n\to\infty}\sup_{t\in[0,T]}\|\mathcal{Q}^{n}_{37}\|_{L^2}=0,$$
		$$\quad  \lim_{n\to\infty}\sup_{t\in[0,T]}\|\mathcal{Q}^{n}_{38}\|_{L^2}=0.$$
	\end{lem}
	\begin{proof}
		By \eqref{M u 1}, \eqref{M u 2}, \eqref{g li}, \eqref{D-3} and Fubini's theorem we have
		\begin{align}
			\mathcal{Q}^n_{37}&= n^{4H}\int_{0}^{t}E^{j,j,l,l'}_s\int_{\frac{[ns]}{n}}^{s}K(s-u)^2M_u^{1,n,l}M_u^{2,n,l'}\d u\d s\nonumber\\&=\sum_{i,i'=1}^{m}n^{4H}\int_{0}^{t}E^{j,j,l,l'}_s\int_{\frac{[ns]}{n}}^{s}K(s-u)^2\Big(\int_{0}^{\frac{[nu]}{n}}g_{li}(u,v)\d W^i_v\Big)\sigma^{l'}_{i'}(X^n_{\frac{[nu]}{n}})\int_{\frac{[nu]}{n}}^{u}K(u-v)\d W^{i'}_v\d u\d s\nonumber\\&=\sum_{i,i'=1}^{m}n^{2H}\int_{0}^{t}E^{j,j,l,l'}_s\int_{\frac{[ns]}{n}}^{s}K(s-u)^2D^{i,i',l,l'}_{3,u}\d u\d s.\nonumber
		\end{align}
		By \eqref{dcp of stochastic integral}, \eqref{g li}, \eqref{def of D 1}, \eqref{def of D 2}, the linear growth properties of $\sigma$, the Cauchy-Schwarz inequality, Minkowski's inequality, Lemmas \ref{bound of X n},  \ref{small time estimate},   \ref{esti of D 1}, and  \ref{esti of D 2}, we have
		\begin{align}\label{bound of D-3-2}
			\E[&|D^{i ,i',l,l'}_{3,u}|^2]\nonumber\\&=n^{4H}\E\Big[\Big(\int_{0}^{\frac{[nu]}{n}}g_{li}(u,v)\d W^i_v\Big)^2\Big( \int_{\frac{[nu]}{n}}^{u}K(u-v)\sigma^{l'}_{i'}(X^n_{\frac{[nv]}{n}})\d W^{i'}_v\Big)^2\Big]\nonumber\\&=n^{4H}\E\Big[\Big(\int_{0}^{\frac{[nu]}{n}}\int_{0}^{v}g_{li}(u,r)\d W_r^{i}g_{li}(u,v)\d W^i_v+\int_{0}^{\frac{[nu]}{n}}\int_{0}^{r}g_{li}(u,v)\d W_v^{i}g_{li}(u,r)\d W^i_r\nonumber\\&\quad+\int_{0}^{\frac{[nu]}{n}}|g_{li}(u,v)|^2\d v\Big)\Big(\int_{\frac{[nu]}{n}}^{u}\int_{\frac{[nu]}{n}}^{v}K(u-r)\sigma^{l'}_{i'}(X^n_{\frac{[nr]}{n}})\d W^{i'}_rK(u-v)\sigma^{l'}_{i'}(X^n_{\frac{[nv]}{n}})\d W^{i'}_v\nonumber\\&\quad+\int_{\frac{[nu]}{n}}^{u}\int_{\frac{[nu]}{n}}^{r}K(u-v)\sigma^{l'}_{i'}(X^n_{\frac{[nv]}{n}})\d W^{i'}_vK(u-r)\sigma^{l'}_{i'}(X^n_{\frac{[nr]}{n}})\d W^{i'}_r+\int_{\frac{[nu]}{n}}^{u}K(u-v)^2|\sigma^{l'}_{i'}(X^n_{\frac{[nv]}{n}})|^2\d v\Big)\Big]\nonumber\\&\leq C\E\Big[\Big(D^{i, i,l,l'}_{1,u}+n^{2H}\int_{0}^{\frac{[nu]}{n}}|g_{li}(u,v)|^2\d v\Big)\Big(D^{i, i,l,l'}_{2,u}+n^{2H}\int_{\frac{[nu]}{n}}^{u}K(u-v)^2|\sigma^{l'}_{i'}(X^n_{\frac{[nv]}{n}})|^2\d v\Big)\Big]\nonumber\\&\leq C\Big(\|D^{i ,i,l,l'}_{1,u}\|_{L^2}\|D^{i ,i,l,l'}_{2,u}\|_{L^2}+n^{2H}\int_{\frac{[nu]}{n}}^{u}K(u-v)^2\|\sigma^{l'}_{i'}(X^n_{\frac{[nv]}{n}})\|^2_{L^4}\d v\|D^{i, i,l,l'}_{1,u}\|_{L^2}\nonumber\\&\quad+\int_{0}^{\frac{[nu]}{n}}\Big(K(u-v)-K(\frac{[nu]}{n}-v)\Big)^2 \|\sigma^{l'}_{i'}(X^n_{\frac{[nv]}{n}})\|^2_{L^4}\d v\|D^{i, i,l,l'}_{1,u}\|_{L^2}\nonumber\\&\quad+\int_{0}^{\frac{[nu]}{n}}\Big(K(u-v)-K(\frac{[nu]}{n}-v)\Big)^2 \|\sigma^{l'}_{i'}(X^n_{\frac{[nv]}{n}})\|^2_{L^4}\d v\int_{\frac{[nv]}{n}}^{u}K(u-v)^2\|\sigma^{l'}_{i'}(X^n_{\frac{[nv]}{n}})\|^2_{L^4}\d v\Big)\nonumber\\&\leq C,
		\end{align}
		with $C$ being independent of $n$.
		
		Since the derivatives of $\sigma$ are bounded, by Fubini's theorem we have
		\begin{align}
			\E[&|\mathcal{Q}^n_{37}|^2] \nonumber\\ &\leq C n^{4H}\int_{0}^{t}\int_{0}^{s}E\Big[\int_{\frac{[ns]}{n}}^{s}K(s-u)^2 D^{i ,i',l,l'}_{3,u}\d u \int_{\frac{[nz]}{n}}^{z}K(z-w)^2 D^{i, i',l,l'}_{3,w}\d w\Big]\d z \d s\nonumber\\&\leq C \int_{0}^{t}\int_{0}^{\frac{[ns]}{n}}n^{4H}\int_{\frac{[ns]}{n}}^{s}\int_{\frac{[nz]}{n}}^{z}K(s-u)^2 K(z-w)^2 \E\Big[ D^{i ,i',l,l'}_{3,u}  D^{i ,i',l,l'}_{3,w}\Big]\d w\d u\d z \d s\nonumber\\&\quad+C \int_{0}^{t}\int_{\frac{[ns]}{n}}^{s}n^{4H}\int_{\frac{[ns]}{n}}^{s}\int_{\frac{[nz]}{n}}^{z}K(s-u)^2 K(z-w)^2 \E\Big[ D^{i ,i',l,l'}_{3,u}  D^{i ,i',l,l'}_{3,w}\Big]\d w\d u\d z \d s\nonumber\\&:=\mathcal{Q}^n_{371}+\mathcal{O}^n_{372}.\nonumber
		\end{align}
		For the term $\mathcal{Q}^n_{371}$, Lemma \ref{small time estimate} and \eqref{est of 3721} yield that
		$$\mathcal{Q}^n_{371}\to0\quad \text{as $n\to\infty$}.$$
		For the term $\mathcal{Q}^n_{372}$, by \eqref{bound of D-3-2}, the Cauchy-Schwarz inequality, Lemma \ref{small time estimate} and the dominated convergence theorem, we obtain
		\begin{align}
			\lim_{n\to\infty}\mathcal{O}^n_{372}&\leq C\lim_{n\to\infty} \int_{0}^{t}\int_{\frac{[ns]}{n}}^{s}n^{4H}\int_{\frac{[ns]}{n}}^{s}\int_{\frac{[nz]}{n}}^{z}K(s-u)^2K(z-w)^2 \d w\d u\d z \d s\nonumber\\&\leq C\int_{0}^{t} \int_{0}^{s}\lim_{n\to\infty}\mathbb{I}_{( \frac{[ns]}{n},s)}(z)\d z \d s=0.\nonumber
		\end{align}
		Therefore we conclude that $\lim_{n\to\infty}\|\mathcal{Q}^n_{37}\|_{L^2}=0$.
		Similar to $\mathcal{Q}^n_{37}$, it holds that $\mathcal{Q}^n_{38}\to 0$ in $L^2$.
	\end{proof}
	\begin{lem}\label{Q-1,2}
		For any $j,j'=1,\cdots,m;l,l'=1,\cdots,d,$
		\[
		\begin{split} 
			&\lim_{n\to\infty}\sup_{t\in[0,T]}\|\mathcal{Q}^{n}_{1}\|_{L^2}=0,\\
			&\lim_{n\to\infty}\sup_{t\in[0,T]}\|\mathcal{Q}^{n}_{2}\|_{L^2}=0.\\
		\end{split}
		\]
	\end{lem}
	\begin{proof}
		Denote
		\begin{align}\label{D-8}
			D^{j ,j',l,l'}_{8,s}:=n^{4H}\int_{\frac{[ns]}{n}}^{s}\Big(\int_{\frac{[ns]}{n}}^{u}K(s-r)M_r^{n,l'}\d W_r^{j'}\Big)K(s-u)M_u^{n,l}\d W_u^{j}.
		\end{align}
		By Fubini's theorem we  see  that
		$$\E[|\mathcal{Q}^n_1|^2]=\E\Big[\int_{0}^{t}\int_{0}^{t}D^{j, j',l,l'}_{8,s}D^{j, j',l,l'}_{8,v}\d v \d s\Big]=2\int_{0}^{t}\int_{0}^{s}\E[D^{j, j',l,l'}_{8,s}D^{j, j',l,l'}_{8,v}]\d v \d s.$$
		Note that $v<s$   implies that $\frac{[nv]}{n}<\frac{[ns]}{n}$, and if $v\leq \frac{[ns]}{n}$ by the tower property we have
		\begin{align}
			\E[D^{j, j',l,l'}_{8,s}D^{j, j',l,l'}_{8,v}]=\E[\E[D^{j, j',l,l'}_{8,s}|\mathcal{F}_{\frac{[ns]}{n}}]D^{j, j',l,l'}_{8,v}]=0.\nonumber
		\end{align}
		For $v>\frac{[ns]}{n}$, by \eqref{dcp of stochastic integral}, Fubini's theorem and the tower property we have
		\begin{align}\label{def of double D-8}
			&\E[D^{j, j',l,l'}_{8,s}D^{j, j',l,l'}_{8,v}]\nonumber\\&=n^{8H}\E\Big[\int_{\frac{[ns]}{n}}^{s}\Big(\int_{\frac{[ns]}{n}}^{u}K(s-r)M_r^{n,l}\d W_r^j\Big) K(s-u)M_u^{n,l'}\d W_u^{j'}\nonumber\\&\quad\cdot \int_{\frac{[nv]}{n}}^{\frac{[ns]}{n}}\Big(\int_{\frac{[nv]}{n}}^{u}K(v-r)M_r^{n,l}\d W_r^j\Big) K(v-u)M_u^{n,l'}\d W_u^{j'}\Big]\nonumber\\&\quad+n^{8H}\E\Big[\int_{\frac{[ns]}{n}}^{v}\Big(\int_{\frac{[ns]}{n}}^{u}K(s-r)M_r^{n,l}\d W_r^j\Big) K(s-u)M_u^{n,l'}\d W_u^{j'}\nonumber\\&\quad\cdot \int_{\frac{[ns]}{n}}^{v}\Big(\int_{\frac{[nv]}{n}}^{u}K(v-r)M_r^{n,l}\d W_r^j\Big) K(v-u)M_u^{n,l'}\d W_u^{j'}\Big]\nonumber\\&\quad+n^{8H}\E\Big[\int_{\frac{[ns]}{n}}^{v}\Big(\int_{\frac{[ns]}{n}}^{u}K(s-r)M_r^{n,l}\d W_r^j\Big) K(s-u)M_u^{n,l'}\d W_u^{j'}\nonumber\\&\quad\cdot \int_{v}^{s}\Big(\int_{\frac{[nv]}{n}}^{u}K(v-r)M_r^{n,l}\d W_r^j\Big) K(v-u)M_u^{n,l'}\d W_u^{j'}\Big]\nonumber\\&=n^{8H}\E\Big[\E\Big[\int_{\frac{[ns]}{n}}^{s}\Big(\int_{\frac{[ns]}{n}}^{u}K(s-r)M_r^{n,l}\d W_r^j\Big) K(s-u)M_u^{n,l'}\d W_u^{j'}|\mathcal{F}_{\frac{[ns]}{n}}\Big]\nonumber\\&\quad\cdot \int_{\frac{[nv]}{n}}^{\frac{[ns]}{n}}\Big(\int_{\frac{[nv]}{n}}^{u}K(v-r)M_r^{n,l}\d W_r^j\Big) K(v-u)M_u^{n,l'}\d W_u^{j'}\Big]\nonumber\\&\quad+n^{8H}\E\Big[\int_{\frac{[ns]}{n}}^{v}\Big(\int_{\frac{[ns]}{n}}^{u}K(s-r)M_r^{n,l}\d W_r^j\Big) \Big(\int_{\frac{[nv]}{n}}^{u}K(v-r)M_r^{n,l}\d W_r^j\Big)\nonumber\\&\quad\cdot K(s-u)K(v-u)|M_u^{n,l'}|^2\d u\Big]\nonumber\\&\quad+n^{8H}\E\Big[\int_{\frac{[ns]}{n}}^{v}\Big(\int_{\frac{[ns]}{n}}^{u}K(s-r)M_r^{n,l}\d W_r^j\Big) K(s-u)M_u^{n,l'}\d W_u^{j'}\nonumber\\&\quad\cdot \E\Big[\int_{v}^{s}\Big(\int_{\frac{[nv]}{n}}^{u}K(v-r)M_r^{n,l}\d W_r^j\Big) K(v-u)M_u^{n,l'}\d W_u^{j'}|\mathcal{F}_v\Big]\Big]\nonumber\\&=n^{2H}\int_{\frac{[ns]}{n}}^{v}K(s-u)K(v-u)\E[D^{l,l'}_{9,u}]\d u,
		\end{align}
		where
		\begin{align}\label{D-9}
			D^{l,l'}_{9,u}:=n^{6H}\Big(\int_{\frac{[ns]}{n}}^{u}K(s-r)M_r^{n,l}\d W_r^j\Big) \Big(\int_{\frac{[nv]}{n}}^{u}K(v-r)M_r^{n,l}\d W_r^j\Big)|M_u^{n,l'}|^2.
		\end{align}
		For any $m\geq 2, u\leq s$ and $s\leq T$, it holds
		\begin{align}\label{esti of stochastic integral-2}
			\Big\|\int_{\frac{[ns]}{n}}^{u}K(s-r)M_r^{n,l}\d W_r^j\Big\|_{L^m}&\leq C \Big\|\int_{\frac{[ns]}{n}}^{u}|K(s-r)M_r^{n,l}|^2\d r\Big\|^{\frac{1}{2}}_{L^{m/2}}\nonumber\\&\leq C\Big\|\int_{\frac{[ns]}{n}}^{s}K(s-r)^2|M_r^{n,l}|^2\d r\Big\|^{\frac{1}{2}}_{L^{m/2}}\nonumber\\&\leq C\Big(\int_{\frac{[ns]}{n}}^{s}K(s-r)^2\||M_r^{n,l}|^2\|_{L^{m/2}}\d r\Big)^{1/2}\nonumber\\&\leq Cn^{-2H},
		\end{align}
		where BDG's inequality, Minkowski's    ininequality, Lemmas \ref{bound of M n} and   \ref{small time estimate} are used.
		
		By the Cauchy-Schwarz inequality, Lemma \ref{bound of M n} and \eqref{esti of stochastic integral-2} we have
		\begin{align}\label{est of D-9}
			E[D^{l,l'}_{9,u}]&\leq n^{6H}\Big\|\int_{\frac{[ns]}{n}}^{u}K(s-r)M_r^{n,l}\d W_r^j\Big\|_{L^4}\Big\| \int_{\frac{[nv]}{n}}^{u}K(v-r)M_r^{n,l}\d W_r^j\Big\|\|M_u^{n,l'}\|^2_{L^4}\nonumber\\&\leq C.
		\end{align}
		This implies by Lemma \ref{esti of B}, \eqref{def of double D-8} and \eqref{est of D-9} that
		$$|\E[D^{j, j',l,l'}_{8,s}D^{j, j',l,l'}_{8,v}]|\leq Cn^{2H}\int_{\frac{[ns]}{n}}^{v}K(s-u)K(v-u)\d u\to 0.$$
		Applying the dominated convergence theorem  with respect to $\d v\otimes \d s$, we have $$\mathcal{Q}^n_{1}\xrightarrow {  L^2} 0.$$
		Similar to $\mathcal{Q}^n_{1}$, it holds that $\mathcal{Q}^n_{2}\to 0$ in $L^2$.
	\end{proof}
	From \eqref{cov of V} and Lemma \ref{est of A-1,2}, we quickly get that
	\begin{lem}\label{est of V-1}
		For all $t\in [0,T], k_1,k_2\in\{1,\cdots,d\}, 1\leq j\leq m,$
		$$\langle V^{n,k_1,j},  V^{n,k_2,j}\rangle_t\xrightarrow {  L^2} C_M\sum_{i,i'=1}^{m}\sum_{l,l'=1}^{d}\int_{0}^{t}\sigma^l_i(X_{s})\sigma^{l'}_{i}(X_{s})\partial_l\sigma^{k_1}_{i'}(X_{s})\partial_{l'}\sigma^{k_2}_{i'}(X_{s})\d s.$$
	\end{lem}
	\subsection{The limit of covariation of $V^n$ and $W$}
	\begin{lem}\label{est of V-2}
		For all $t\in [0,T], k\in\{1,\cdots,d\}, 1\leq j\leq m,$   we have 
		$$\langle V^{n,k,j}, W^{j} \rangle_t\xrightarrow { L^1} 0.$$
	\end{lem}
	\begin{proof}
		Recall the expression of \eqref{def of V}.  We shall  compute the $L^1$-limit of 
		$$\langle V^{n,k,j}, W^{j} \rangle_t=n^{2H}\int_{0}^{t}\big[X^n_s-X^n_{\frac{[ns]}{n}}-M_s^{n}\big]^{k}\d s.$$
		By   Fubini's theorem, the Cauchy-Schwarz inequality, Minkowski's   inequality and Lemma \ref{est of A-1,2}, we have
		\begin{align}
			&\E\Big[n^{4H}\Big|\int_{0}^{t}\big[X^n_s-X^n_{\frac{[ns]}{n}}-M_s^{n}\big]^{k}\d s\Big|^2\Big]\nonumber\\&=2\E\Big[n^{4H}\int_{0}^{t}\int_{0}^{s}\big[X^n_s-X^n_{\frac{[ns]}{n}}-M_s^{n}\big]^{k}\big[X^n_v-X^n_{\frac{[nv]}{n}}-M_v^{n}\big]^{k}\d v \d s\Big]\nonumber\\&=2\int_{0}^{t}\int_{0}^{s}n^{4H}\E\Big[\big[X^n_s-X^n_{\frac{[ns]}{n}}-M_s^{n}\big]^{k}\big[X^n_v-X^n_{\frac{[nv]}{n}}-M_v^{n}\big]^{k}\Big]\d v \d s\nonumber\\&\leq C\int_{0}^{t}\int_{0}^{s}\Big[ \|n^{2H}\mathcal{A}^{n,k}_{1,s}\|_{L^2}+\|n^{2H}\mathcal{A}^{n,k}_{2,s}\|_{L^2}\Big]\Big[ \|n^{2H}\mathcal{A}^{n,k}_{1,v}\|_{L^2}+\|n^{2H}\mathcal{A}^{n,k}_{2,v}\|_{L^2}\Big]\d v\d s\nonumber\\&\leq C.\nonumber
		\end{align}
		Therefore, in light of the dominated convergence theorem, it suffices to show $$ n^{4H}\E\Big[\big[X^n_s-X^n_{\frac{[ns]}{n}}-M_s^{n}\big]^{k}\big[X^n_v-X^n_{\frac{[nv]}{n}}-M_v^{n}\big]^{k}\Big]\to 0$$ 
		for each $s,v$ with $v<s$. We only need to consider the case $v<\frac{[ns]}{n}$.
		In this case by \eqref{different parts of Z} and the tower property we have
		\begin{align}
			&\E\Big[\big[X^n_s-X^n_{\frac{[ns]}{n}}-M_s^{n}\big]^{k}\big[X^n_v-X^n_{\frac{[nv]}{n}}-M_v^{n}\big]^{k}\Big]\nonumber\\&=\E\Big[\E\Big[\big[X^n_s-X^n_{\frac{[ns]}{n}}-M_s^{n}\big]^{k}|\mathcal{F}_{\frac{[ns]}{n}}\Big]\big[X^n_v-X^n_{\frac{[nv]}{n}}-M_v^{n}\big]^{k}\Big]\nonumber\\&=\E\Big[\Big(\mathcal{A}^{n,k}_{1,s}+\sum_{j=1}^{m}\sum_{l=1}^{d}\int_{0}^{\frac{[ns]}{n}}\Big(K(s-u)-K(\frac{[ns]}{n}-u)\Big)\partial_l \sigma^k_j(X^n_{\frac{[nu]}{n}})M_u^{n,l}\d W_u^j\Big)\big[X^n_v-X^n_{\frac{[nv]}{n}}-M_v^{n}\big]^{k}\Big]\nonumber\\&=\E\Big[\mathcal{A}^{n,k}_{1,s}\big[X^n_v-X^n_{\frac{[nv]}{n}}-M_v^{n}\big]^{k}\Big]\nonumber\\&\quad+\sum_{j=1}^{m}\sum_{l=1}^{d}\E\Big[\int_{0}^{\frac{[ns]}{n}}\Big(K(s-u)-K(\frac{[ns]}{n}-u)\Big)\partial_l \sigma^k_j(X^n_{\frac{[nu]}{n}})M_u^{n,l}\d W_u^j\cdot\mathcal{A}^{n,k}_{1,v}\Big]\nonumber\\&\quad+\sum_{j,j'=1}^{m}\sum_{l,l'=1}^{d}\E\Big[\int_{0}^{\frac{[ns]}{n}}\Big(K(s-u)-K(\frac{[ns]}{n}-u)\Big)\partial_l \sigma^k_j(X^n_{\frac{[nu]}{n}})M_u^{n,l}\d W_u^j\nonumber\\&\quad\cdot\int_{0}^{\frac{[nv]}{n}}\Big(K(v-u)-K(\frac{[nv]}{n}-u)\Big)\partial_{l'} \sigma^k_{j'}(X^n_{\frac{[nu]}{n}})M_u^{n,l'}\d W_u^{j'}\Big]\nonumber\\&\quad+\sum_{j,j'=1}^{m}\sum_{l,l'=1}^{d}\E\Big[\int_{0}^{\frac{[ns]}{n}}\Big(K(s-u)-K(\frac{[ns]}{n}-u)\Big)\partial_l \sigma^k_j(X^n_{\frac{[nu]}{n}})M_u^{n,l}\d W_u^j\nonumber\\&\quad\cdot\partial_{l'} \sigma^k_{j'}(X^n_\frac{[nv]}{n})\int_{\frac{[nv]}{n}}^{v}K(v-u)M_u^{n,l'}\d W_u^{j'}\Big]\nonumber\\&:=\mathcal{M}_1+\sum_{j=1}^{m}\sum_{l=1}^{d}\mathcal{M}_2+\sum_{j,j'=1}^{m}\sum_{l,l'=1}^{d}\mathcal{M}_3+\sum_{j,j'=1}^{m}\sum_{l,l'=1}^{d}\mathcal{M}_4.\nonumber
		\end{align}
		For the term $\mathcal{M}_1$, by Fubini's theorem, Minkowski's inequality, Lemma \ref{basic lemma}-(iii),  the fact that  the derivatives of  $b$ and $\sigma$ are bounded, \eqref{est of A-1}, \eqref{est of A-2} and Lemma \ref{bound of M n}, we have
		\begin{align}
			\mathcal{M}_1&=\E\Big[\mathcal{A}^{n,k}_{1,s}\big[X^n_v-X^n_{\frac{[nv]}{n}}-M_v^{n}\big]^{k}\Big]\nonumber\\&\leq \int_{0}^{\frac{[ns]}{n}}\Big(K(s-u)-K(\frac{[ns]}{n}-u)\Big)\E\Big[b^k(X^n_{\frac{[nu]}{n}})\big[X^n_v-X^n_{\frac{[nv]}{n}}-M_v^{n}\big]^{k}\Big]\d u\nonumber\\&\quad+\int_{\frac{[ns]}{n}}^{s}K(s-u)\E\Big[b^k(X^n_\frac{[ns]}{n})\big[X^n_v-X^n_{\frac{[nv]}{n}}-M_v^{n}\big]^{k}\Big]\d u\nonumber\\&\quad+C\sum_{l=1}^{d}\int_{0}^{\frac{[ns]}{n}}\Big(K(s-u)-K(\frac{[ns]}{n}-u)\Big)\E\Big[M_u^{n,l}\big[X^n_v-X^n_{\frac{[nv]}{n}}-M_v^{n}\big]^{k}\Big]\d u\nonumber\\&\quad+C\sum_{l=1}^{d}\int_{\frac{[ns]}{n}}^{s}K(s-u)\E\Big[M_u^{n,l}\big[X^n_v-X^n_{\frac{[nv]}{n}}-M_v^{n}\big]^{k}\Big]\d u\nonumber\\&\leq \int_{0}^{\frac{[ns]}{n}}\Big(K(s-u)-K(\frac{[ns]}{n}-u)\Big)\|b^k(X^n_{\frac{[nu]}{n}})\|_{L^2}\Big\|\big[X^n_v-X^n_{\frac{[nv]}{n}}-M_v^{n}\big]^{k}\Big\|_{L^2}\d u\nonumber\\&\quad+\int_{\frac{[ns]}{n}}^{s}K(s-u)\|b^k(X^n_{\frac{[ns]}{n}})\|_{L^2}\Big\|\big[X^n_v-X^n_{\frac{[nv]}{n}}-M_v^{n}\big]^{k}\Big\|_{L^2}\d u\nonumber\\&\quad+C\sum_{l=1}^{d}\int_{0}^{\frac{[ns]}{n}}\Big(K(s-u)-K(\frac{[ns]}{n}-u)\Big)\|M^{n,l}_u\|_{L^2}\Big\|\big[X^n_v-X^n_{\frac{[nv]}{n}}-M_v^{n}\big]^{k}\Big\|_{L^2}\d u\nonumber\\&\quad+C\sum_{l=1}^{d}\int_{\frac{[ns]}{n}}^{s}K(s-u)\|M^{n,l}_u\|_{L^2}\Big\|\big[X^n_v-X^n_{\frac{[nv]}{n}}-M_v^{n}\big]^{k}\Big\|_{L^2}\d u\nonumber\\&\leq Cn^{-2H}\Big[\int_{0}^{\frac{[ns]}{n}}\Big(K(s-u)-K(\frac{[ns]}{n}-u)\Big)\d u+\int_{\frac{[ns]}{n}}^{s}K(s-u)\d u\Big]\leq Cn^{-(3H+1/2)}.\nonumber
		\end{align}
		Hence, $n^{4H}\mathcal{M}_1\leq Cn^{H-1/2}\to 0.$ 
		
		For the term $\mathcal{M}_2$, by the Cauchy-Schwarz inequality, \eqref{est of A-1} and \eqref{esti of stochastic integral} we have
		\begin{align}
			\mathcal{M}_2&=\E\Big[\int_{0}^{\frac{[ns]}{n}}\Big(K(s-u)-K(\frac{[ns]}{n}-u)\Big)\partial_l \sigma^k_j(X^n_{\frac{[nu]}{n}})M_u^{n,l}\d W_u^j\cdot\mathcal{A}^{n,k}_{1,v}\Big]\nonumber\\&\leq\Big\|\int_{0}^{\frac{[ns]}{n}}\Big(K(s-u)-K(\frac{[ns]}{n}-u)\Big)\partial_l \sigma^k_j(X^n_{\frac{[nu]}{n}})M_u^{n,l}\d W_u^j\Big\|_{L^2}\|\mathcal{A}^{n,k}_{1,v}\|_{L^2}\nonumber\\&\leq Cn^{-(3H+1/2)},\nonumber
		\end{align}
		i.e.  $n^{4H}\mathcal{M}_2\leq Cn^{H-1/2}\to 0.$
		
		For the term $\mathcal{M}_3$, notice  that $v<\frac{[ns]}{n}$,  by \eqref{dcp of stochastic integral}, Fubini's theorem, Lemma \ref{esti of A} and the Cauchy-Schwarz inequality, we have
		\begin{align}
			n^{4H}\mathcal{M}_3&=n^{4H}\E\Big[\int_{0}^{\frac{[ns]}{n}}\Big(K(s-u)-K(\frac{[ns]}{n}-u)\Big)\partial_l \sigma^k_j(X^n_{\frac{[nu]}{n}})M_u^{n,l}\d W_u^j\nonumber\\&\quad\cdot\partial_{l'} \sigma^k_{j'}(X^n_\frac{[nv]}{n})\int_{\frac{[nv]}{n}}^{v}\Big(K(v-u)-K(\frac{[nv]}{n}-u)\Big)M_u^{n,l'}\d W_u^{j'}\Big]\nonumber\\&=n^{4H}\E\Big[\int_{0}^{\frac{[ns]}{n}}\Big(K(s-u)-K(\frac{[ns]}{n}-u)\Big)\partial_l \sigma^k_j(X^n_{\frac{[nu]}{n}})M_u^{n,l}\d W_u^j\nonumber\\&\quad\cdot\int_{0}^{\frac{[ns]}{n}}\mathbb{I}_{(\frac{[nv]}{n},v)}(u)\Big(K(v-u)-K(\frac{[nv]}{n}-u)\Big)\partial_{l'} \sigma^k_{j'}(X^n_\frac{[nu]}{n})M_u^{n,l'}\d W_u^{j'}\Big]\nonumber\\&\leq C n^{4H} \E\Big[\int_{\frac{[nv]}{n}}^{v}\Big(K(s-u)-K(\frac{[ns]}{n}-u)\Big)\Big(K(v-u)-K(\frac{[nv]}{n}-u)\Big)M_u^{n,l}M_u^{n,l'}\d u\Big]\nonumber\\&\leq Cn^{2H}\int_{\frac{[nv]}{n}}^{v}\Big(K(s-u)-K(\frac{[ns]}{n}-u)\Big)\Big(K(v-u)-K(\frac{[nv]}{n}-u)\Big)\d u\to 0\quad \text{as $n\to\infty$}.\nonumber
		\end{align}
		
		For the term $\mathcal{M}_4$, by \eqref{dcp of stochastic integral}, Fubini's theorem and the Cauchy-Schwarz inequality we have
		\begin{align}
			\mathcal{M}_4&=\E\Big[\int_{0}^{\frac{[ns]}{n}}\Big(K(s-u)-K(\frac{[ns]}{n}-u)\Big)\partial_l \sigma^k_j(X^n_{\frac{[nu]}{n}})M_u^{n,l}\d W_u^j\nonumber\\&\quad\cdot\partial_{l'} \sigma^k_{j'}(X^n_\frac{[nv]}{n})\int_{\frac{[nv]}{n}}^{v}K(v-u)M_u^{n,l'}\d W_u^{j'}\Big]\nonumber\\&=\E\Big[\int_{0}^{\frac{[ns]}{n}}\Big(K(s-u)-K(\frac{[ns]}{n}-u)\Big)\partial_l \sigma^k_j(X^n_{\frac{[nu]}{n}})M_u^{n,l}\d W_u^j\nonumber\\&\quad\cdot\int_{0}^{\frac{[ns]}{n}}\mathbb{I}_{(\frac{[nv]}{n},v)}(u)K(v-u)\partial_{l'} \sigma^k_{j'}(X^n_\frac{[nu]}{n})M_u^{n,l'}\d W_u^{j'}\Big]\nonumber\\&\leq C \E\Big[\int_{\frac{[nv]}{n}}^{v}\Big(K(s-u)-K(\frac{[ns]}{n}-u)\Big)K(v-u)M_u^{n,l}M_u^{n,l'}\d u\Big]\nonumber\\&\leq C \int_{\frac{[nv]}{n}}^{v}\Big(K(s-u)-K(\frac{[ns]}{n}-u)\Big)K(v-u)\|M_u^{n,l}\|_{L^2}\|M_u^{n,l'}\|_{L^2}\d u\nonumber\\&\leq Cn^{-4H}\int_{0}^{nv-[nv]}\Big|(z+ns-nv)^{H-1/2}-(z+[ns]-nv)^{H-1/2}\Big|z^{H-1/2}\d z\nonumber\\&\leq Cn^{-4H}\int_{0}^{1}\Big|(z+ns-nv)^{H-1/2}-(z+[ns]-nv)^{H-1/2}\Big|z^{H-1/2}\d z,\nonumber
		\end{align}
		where the change of variable $z=n(v-u)$ is used. And for $v<s$, $\Big|(z+ns-nv)^{H-1/2}-(z+[ns]-nv)^{H-1/2}\Big|\to 0.$   Thus,  $n^{4H}\mathcal{M}_4\leq Cn^{H-1/2}\to 0.$
	\end{proof}
	\begin{lem}\label{charact of V}
		The process $V^n$ converges stably in law in $\mathcal{C}_0$
		%\footnote{what is $\mathcal{C}_0$?}  
		to the following  continuous process% given by  %$V=\{V^{k, j}\}$ of the following form:
		$$
		V^{k, j}=\sum_{i=1}^{m}\sum_{i'=1}^{m}\sum_{l=1}^dC_M \int_0^{\cdot} \sigma_i^l\left(X_s\right)\partial_l\sigma^{k}_{i'}(X_{s}) \mathrm{d} B_s^{i,i',j},
		$$
		where $\mathcal{C}_0$ denotes the set of all  $\mathbb{R}^d$-valued continuous functions on $[0,T]$ vanishing at $t=0$, $B$ is $m^3$-dimensional standard Brownian motion, independent of the original Brownian  motion $W$. 
		%and defined on some extension of $(\Omega,\mathcal{F},(\mathcal{F}_t)_{t\ge0},\mathbb{P})$.
	\end{lem}
	\begin{proof}
		By Lemmas \ref{est of V-1},    \ref{est of V-2},  and    \cite[Theorem 4-1]{Ja1}  we  see that $V^n$  converges stably   in law in $\mathcal{C}_0$ to a conditional  Gaussian martingale $V=\{V^{k, j}\}$ with
		$$
		\begin{aligned}
			\langle V^{k_1, j_1}, V^{k_2, j_2}\rangle_t & = \begin{cases}C_M^2\sum_{i,i'=1}^{m}\sum_{l,l'=1}^{d}\int_{0}^{t}\sigma^l_i(X_{s})\sigma^{l'}_{i}(X_{s})\partial_l\sigma^{k_1}_{i'}(X_{s})\partial_{l'}\sigma^{k_2}_{i'}(X_{s})\d s & \text { if } j_1=j_2, \\
				0 & \text { if } j_1 \neq j_2,\end{cases} \\
			\langle V^{k, i}, W^j\rangle_t & =0, \quad{\forall} k \in\{1, \ldots, d\},{ \forall}(i, j) \in\{1, \ldots, m\}^2 .
		\end{aligned}
		$$
		
		Furthermore, since $\mathrm{d}\langle V, V\rangle_t \ll \mathrm{d} t$,  an application of  \cite[Proposition 1-4]{Ja1} yields that $V$ can be  represented by 
		$$
		V^{k, j}=\sum_{i=1}^{m}\sum_{i'=1}^{m}\sum_{l=1}^dC_M \int_0^{\cdot} \sigma_i^l\left(X_s\right)\partial_l\sigma^{k}_{i'}(X_{s}) \mathrm{d} B_s^{i,i',j},
		$$
		where $B$ is $m^3
		$-dimensional standard Brownian motion  independent of $W$. This concludes the proof.
	\end{proof}
	\subsection{The tightness of normalized error process $U^n$}
	\begin{lem}\label{nabla b}
		For all $i\in\{1,\cdots,d\}$ and any $\epsilon\in (0,H)$   we have 
		$$ \int_{0}^{t}K(t-s) n^{2H}\nabla b^i(X^n_{\frac{[ns]}{n}})\Big(X^n_s-X^n_{\frac{[ns]}{n}}-M_s^{n}\Big)\d s
		\xrightarrow{ \text { in } \mathcal{C}_0^{H-\epsilon}} 0\quad \hbox{in probability} .$$
	\end{lem}
	\begin{proof}
		First,  we have 
		\begin{equation}
			\int_0^t K(t-s) n^{2H} \nabla b(X^n_{\frac{[ns]}{n}})\big(X^n_s-X^n_{\frac{[ns]}{n}}-M_s^{n}\big) \mathrm{d} s=\sum_{i=1}^m\sum_{k=1}^d \int_0^t K(t-s) \partial_k b^i(X^{n}_{\frac{[ns]}{n}}) \mathrm{d}\langle V^{n, k, j}, W^j\rangle_s \nonumber.
		\end{equation}
		Since $\left\langle V^{n, k, j}, W^j\right\rangle_t$ tends to zero in $L_1$ for all $t \in[0, T]$ by Lemma \ref{est of V-2}, we will prove desired result  by using    Kurtz and Protter \cite[Theorem 7.10]{KP} and   \cite[Lemma A.3]{FU}. We start with showing the tightness in $\mathcal{C}_0$ of $\left\{\left\langle V^{n, k, j}, W^j\right\rangle\right\}_{n \in \mathbb{N}}$. By Minkowski's inequality, \eqref{est of A-1} and \eqref{est of A-2}, we have for $0 \leq s<t \leq T$,
		$$
		\begin{aligned}
			\E\left[\left|\langle V^{n, k, j}, W^j\rangle_t-\langle V^{n, k, j}, W^j\rangle_s\right|^p\right] & \leq n^{2H p}\Big(\int_s^t \Big\|\big[X^n_s-X^n_{\frac{[ns]}{n}}-M_s^{n}\big]^{k}\Big\|_{L^p} \d u\Big)^p\\
			& \leq C n^{2H p}\Big(\int_{s}^{t}\frac{1}{n^{2Hp}}\d u\Big)^p \leq C|t-s|^p .
		\end{aligned}
		$$
		
		Therefore, Kolmogorov's continuity criterion implies $\mathbb{E}\left[\left\|\left\langle V^{n, k, j}, W^j\right\rangle\right\|_{\mathcal{C}_0^\rho}\right] \leq C_\rho$ for any $\rho \in(0,1)$ { with $C_\rho$ being independent of $n$,}  and then, $\left\{\left\langle V^{n, k, j}, W^j\right\rangle\right\}_{n \in \mathbb{N}}$ is tight in $\mathcal{C}_0^\rho$ for any $\rho \in(0,1)$ by     \cite[Theorem B.2]{FU}. Since any subsequence $\left\{\left\langle V^{n_q, k, j}, W^j\right\rangle\right\}_{q \in \mathbb{N}}$ of $\left\{\left\langle V^{n, k, j}, W\right\rangle\right\}_{n \in \mathbb{N}}$ is tight in $\mathcal{C}_0^\rho$ too, Prokhorov's theorem yields that there exists a further subsequence $\left\{\left\langle V^{n_q^{\prime}, k, j}, W^j\right\rangle\right\}_{q \in \mathbb{N}}$ of $\left\{\left\langle V^{n_q, k, j}, W^j\right\rangle\right\}_{q \in \mathbb{N}}$ which is convergent in $\mathcal{C}_0^\rho$. The limit process has to be 0 because  by Lemma \ref{est of V-2}  $\left\langle V^{n_q^{\prime}, k, j}, W^j\right\rangle_t$ tends to zero in $L_1$   for all $t$. Thus, $\left\langle V^{n, k, j}, W^j\right\rangle \xrightarrow{n \rightarrow \infty} 0$ in law both  in $\mathcal{C}_0^\rho$ and in $\mathcal{C}_0$ as well. This implies $\left\langle V^{n, k, j}, W^j\right\rangle$ tends to zero also in probability in $\mathcal{C}_0$. Let $\left\{Y^n\right\}_{n \in \mathbb{N}}$ be the set of simple predictable processes almost surely bounded by one. By \eqref{est of A-1} and \eqref{est of A-2}, we have
		$$
		\E\left[\left|\int_0^t Y_{s-}^n \mathrm{~d}\left\langle V^{n, k, j}, W^j\right\rangle_s\right|\right] \leq \int_0^t n^{2H} \E\Big[\Big|\big[X^n_s-X^n_{\frac{[ns]}{n}}-M_s^{n}\big]^{k}\Big|\Big] \mathrm{d} s \leq C,
		$$
		which implies $\left\langle V^{n, k, j}, W^j\right\rangle$ is uniformly tight in the sense of      Kurtz and Protter \cite[Definition 7.4]{KP}. Noting also that  by Lemma \ref{strong X-X n},  $X^n-X \rightarrow 0$ in probability in $\mathcal{C}_0$,   \cite[Theorem 7.10]{KP} yields
		$$
		\left(\partial_k b^i(X^n),\left\langle V^{n, k, j}, W^j\right\rangle, \int_0^{\cdot} \partial_k b^i\left(X^n_{\frac{[ns]}{n}}\right) \mathrm{d}\left(V^{n, k, j}, W^j\right\rangle_s\right) \xrightarrow {\mathcal{D}_{3 m d^2}}\left(\partial_k b^i(X), 0,0\right).
		$$
		
		Next,  we argue that $\int_0^j \partial_k b^i\left(X^n_{\frac{[ns]}{n}}\right) \mathrm{d}\left\langle V^{n, k, j}, W^j\right\rangle_s$ converges to zero in probability not only in $\mathcal{C}_0$ but also in $\mathcal{C}_0^\rho$. We have for $0 \leq s<t\leq T$,
		$$
		\begin{aligned}
			& \E\left[\left|\sum_{k=1}^d\left(\int_0^t \partial_k b^i\left(X^n_{\frac{[nu]}{n}}\right) \mathrm{d}\left\langle V^{n, k, j}, W^j\right\rangle_u-\int_0^s \partial_k b^i\left(X^n_{\frac{[nu]}{n}}\right) \mathrm{d}\left(V^{n, k, j}, W^j\right\rangle_u\right)\right|^p\right] \\
			& \quad \leq\left(\int_s^t n^{2H} \sum_{k=1}^d \E\left[\left|\partial_k b^i\left(X^n_{\frac{[nu]}{n}}\right)\big[X^n_s-X^n_{\frac{[ns]}{n}}-M_s^{n}\big]^{k}\right|^p\right]^{\frac{1}{p}} \mathrm{~d} u\right)^p \\
			& \quad \leq C|t-s|^p
		\end{aligned}
		$$
		by Minkowski's inequality, \eqref{est of A-1}, \eqref{est of A-2}, and the boundedness of the derivatives of $b$. Thus, Kolmogorov's continuity criterion yields
		$$
		\E\left[\left\|\sum_{k=1}^d \int_0^{\cdot} \partial_k b^i\left(X^n_{\frac{[ns]}{n}}\right) \mathrm{d}\left(V^{n, k, j}, W^j\right\rangle_s\right\|_{\mathcal{C}_0^\rho}\right] \leq C
		$$
		with $C$ being independent of $n$ for any $\rho \in(0,1)$. Consequently, by \cite[Corollary B.3]{FU}, we have the process $\sum_k \int_0^{\cdot} \partial_k b^i\left(X^n_{\frac{[ns]}{n}}\right) \mathrm{d}\left\langle V^{n, k, j}, W^j\right\rangle_s$ converges in $\mathcal{C}_0^{1 / 2-\epsilon}$ to zero, and the desired result follows  by    \cite[Lemma A.3 ]{FU}.
	\end{proof}
	\begin{lem}\label{remainder term}
		$\|\mathcal{R}^n\|_{C_0^{\gamma}}$ tends to zero in $L^p$ for any $\gamma\in (0,H)$ and any $p\geq 1.$
	\end{lem}
	\begin{proof}
		Recall the decomposition of $U^n$, i.e. \eqref{dcp of U n}. By Taylor's theorem we have
		$$\mathcal{R}_t^{n}=\mathcal{R}^{n,1}_{b,t}+\mathcal{R}^{n,2}_{b,t}+\mathcal{R}^{n,1}_{\sigma,t}+\mathcal{R}^{n,2}_{\sigma,t},$$
		where
		\begin{align}
			\mathcal{R}^{n,1}_{b,t}&=n^{2H}\int_{0}^{t}K(t-s)\Big[\big(b(X_s)-b(X^n_s)\big)-\nabla b(X^n_s)\big(X_s-X^n_s\big)\Big]\d s\nonumber\\&=n^{2H}\int_{0}^{t}K(t-s)\int_{0}^{1}\big(\nabla b(X^n_s+r\big(X_s-X^n_s\big))-\nabla b(X^n_s)\big)\d r\big(X_s-X^n_s\big)\d s\label{e.r.1} \\
			\mathcal{R}^{n,1}_{\sigma,t}&=n^{2H}\int_{0}^{t}K(t-s)\Big[\big(\sigma(X_s)-\sigma(X^n_s)\big)-\nabla \sigma(X^n_s)\big(X_s-X^n_s\big)\Big]\d W_s\nonumber\\&=n^{2H}\int_{0}^{t}K(t-s)\int_{0}^{1}\big(\nabla \sigma(X^n_s+r\big(X_s-X^n_s\big))-\nabla \sigma(X^n_s)\big)\d r\big(X_s-X^n_s\big)\d W_s\label{e.r.2}\\
			\mathcal{R}^{n,2}_{b,t}&=n^{2H}\int_{0}^{t}K(t-s)\Big[\big(b(X^n_s)-b(X^n_{\frac{[ns]}{n}})\big)-\nabla b(X^n_{\frac{[ns]}{n}})(X^n_s-X^n_{\frac{[ns]}{n}})\Big]\d s\nonumber\\&=n^{2H}\int_{0}^{t}K(t-s)\int_{0}^{1}\big(\nabla b(X^n_s+r(X^n_s-X^n_{\frac{[ns]}{n}}))-\nabla b(X^n_s)\big)\d r(X^n_s-X^n_{\frac{[ns]}{n}})\d s\label{e.r.3}\\
			\mathcal{R}^{n,2}_{\sigma,t}&=n^{2H}\int_{0}^{t}K(t-s)\Big[\big(\sigma(X^n_s)-\sigma(X^n_{\frac{[ns]}{n}})\big)-\nabla \sigma(X^n_{\frac{[ns]}{n}})(X^n_s-X^n_{\frac{[ns]}{n}})\Big]\d s\nonumber\\&=n^{2H}\int_{0}^{t}K(t-s)\int_{0}^{1}\big(\nabla \sigma(X^n_s+r(X^n_s-X^n_{\frac{[ns]}{n}}))-\nabla \sigma(X^n_s)\big)\d r(X^n_s-X^n_{\frac{[ns]}{n}})\d s\,. \label{e.r.4}
		\end{align}
		For the term $\mathcal{R}^{n,1}_{b,t}$, let
		$$\psi_a(s,x,y)=(x_s-y_s)\int_{0}^{1}\big(\nabla a(y_s+r\big(x_s-y_s\big))-\nabla a(y_s)\big)\d r,$$
		for $a=b$ or $a=\sigma$ and for any $\mathbb{R}^d$-valued adapted process. By Lemma \ref{basic lemma}, (iii) we have
		\begin{align}
			\E[&|\mathcal{R}^{n,1}_{b,t+h}-\mathcal{R}^{n,1}_{b,t}|^p]
			\nonumber\\&\leq 2^{p-1}n^{2Hp}\E\Big[\Big|\int_{0}^{t}\Big(K(t+h-s)-K(t-s)\Big)\psi_b(s,X_s,X^n_s)\d s\Big|^p\Big]\nonumber\\&\quad+2^{p-1}n^{2Hp}\E\Big[\Big|\int_{t}^{t+h}K(t+h-s)\psi_b(s,X_s,X^n_s)\big(X_s-X^n_s\big)\d s\Big|^p\Big]\nonumber\\&\leq Ch^{(H+1/2)p}n^{2Hp}\sup_{s\in[0,T]}\E[|\psi_b(s,X_s,X^n_s)|^p].
		\end{align}
		Similarly, we have
		\begin{align}
			\E[|\mathcal{R}^{n,1}_{\sigma,t+h}-\mathcal{R}^{n,1}_{\sigma,t}|^p]&\leq Ch^{(H+1/2)p}n^{2Hp}\sup_{s\in[0,T]}\E[|\psi_\sigma(s,X_s,X^n_s)|^p],\nonumber\\\E[|\mathcal{R}^{n,2}_{b,t+h}-\mathcal{R}^{n,2}_{b,t}|^p]&\leq Ch^{Hp}n^{2Hp}\sup_{s\in[0,T]}\E[|\psi_b(s,X^n_s,X^n_{\frac{[ns]}{n}})|^p],\nonumber\\\E[|\mathcal{R}^{n,2}_{\sigma,t+h}-\mathcal{R}^{n,2}_{\sigma,t}|^p]&\leq Ch^{Hp}n^{2Hp}\sup_{s\in[0,T]}\E[|\psi_\sigma(s,X_s,X^n_s)|^p].\nonumber
		\end{align}
		By the Cauchy-Schwarz inequality, Assumption \ref{assumption 2.3}, Lemmas  \ref{continuous of X n} and   \ref{rate X-X n} we have 
		\begin{align}
			\E[|\psi_a(s,X_s,X^n_s)|^p]&\leq \E[|X_s-X^n_s|^{2p}]^{1/2}\E\Big[\Big|\int_{0}^{1}\big(\nabla a(X^n_s+r\big(X_s-X^n_s\big))-\nabla a(X^n_s)\big)\d r\Big|^{2p}\Big]^{1/2}\nonumber\\&\leq Cn^{-4Hp},\nonumber
		\end{align}
		and
		\begin{align}
			\E[|\psi_a(s,X^n_s,X^n_{\frac{[ns]}{n}})|^p]&\leq \E[|X^n_s-X^n_{\frac{[ns]}{n}}|^{2p}]^{1/2}\E\Big[\Big|\int_{0}^{1}\big(\nabla a(X^n_{\frac{[ns]}{n}}+r\big(X^n_s-X^n_{\frac{[ns]}{n}}\big))-\nabla a(X^n_{\frac{[ns]}{n}})\big)\d r\Big|^{2p}\Big]^{1/2}\nonumber\\&\leq Cn^{-2Hp},\nonumber
		\end{align}
		for $a=b,\sigma.$
		
		Therefore
		\begin{align}
			\E[&|\mathcal{R}^{n}_{t+h}-\mathcal{R}^{n}_{t}|^p]\nonumber\\&\leq Ch^{(H+1/2)p}n^{2Hp}\Big[\sup_{s\in[0,T]}\E[|\psi_{b}(s,X_s,X^n_s)|^p]+\sup_{s\in[0,T]}\E[|\psi_{\sigma}(s,X_s,X^n_s)|^p]\Big]\nonumber\\&\quad+Ch^{Hp}n^{2Hp}\Big[\sup_{s\in[0,T]}\E[|\psi_{b}(s,X_s,X^n_s)|^p]+\sup_{s\in[0,T]}\E[|\psi_{\sigma}(s,X_s,X^n_s)|^p]\Big]\nonumber\\&\leq Ch^{Hp}\,, \nonumber
		\end{align}
		where $C$ is independent of $n,t$ and $h$. Letting  $t=0$ and noticing  that $\mathcal{R}^{n}_{0}=0$, by Kolmogorov's continuity criterion we obtain that
		$$\|\mathcal{R}^n\|_{C_0^{\gamma}}\to 0,$$
		for all $\gamma\in(0,H)$.
	\end{proof}
	\begin{lem}\label{limit of U}
		If the sequence
		$$
		\left(U^n, V^n,\left\{\nabla b^i(X^n)\right\}_i,\left\{\partial_k \sigma_j^i(X^n)\right\}_{i j k}\right)
		$$
		converges in law in $\mathcal{C}_0^{2H-\epsilon} \times \mathcal{C}_0 \times \mathcal{D}_{d^2} \times \mathcal{D}_{d^2 m}$ to
		$$
		\left(U, V,\left\{\nabla b^i(X)\right\}_i,\left\{\partial_k \sigma_j^i(X)\right\}_{i j k}\right),
		$$
		then $U$ is the  solution of \eqref{limit Volterra eq}.
	\end{lem}
	\begin{proof}
		First we show that $V^n$ is uniformly tight in the sense of     \cite[Definition 7.4]{KP}. Let $\left\{Y^n\right\}_{n \in \mathbb{N}}$ be a sequence of simple predictable processes almost surely bounded by one. Then, for all $t \in[0, T]$, it follows from the It\^o  isometry,
		$$
		\E\Big[\Big|\int_0^t Y_{s-}^n \mathrm{~d} V_s^{n, k, j}\Big|^2\Big]=\E\left[\int_0^t\left|Y_{s-}^n\right|^2 \mathrm{d}\langle V^{n, k, j}\rangle_s\right] \leq \sup _n \E\Big[\langle V^{n, k, j}\rangle_t\Big]<\infty,
		$$
		where the last bound follows from the fact that $\left\{\langle V^{n, k, j}\rangle_t\right\}_{n \in \mathbb{N}}$ is convergent in $L_1$ by Lemma \ref{est of V-1}.
		Now, define $\Phi^n=\left(\Phi^{n, 1}, \ldots, \Phi^{n, d}\right)$ by
		$$
		\begin{aligned}
			\Phi_t^{n, i}&=\sum_{k=1}^d \int_0^t \partial_k b^i\left(X^n_s\right) U_s^{n, k} \mathrm{d} s+\sum_{j=1}^m \sum_{k=1}^d \int_0^t \partial_k \sigma_j^i\left(X^n_s\right) U_s^{n, k} \mathrm{d} W_s^j \\
			&\quad+\sum_{j=1}^m \sum_{k=1}^d \int_0^l \partial_k \sigma_j^i\left(X^n_{\frac{[ns]}{n}}\right) \mathrm{d} V_s^{n, k, j} .
		\end{aligned}
		$$
		
		By the uniform tightness of $V^n$,      \cite[Theorem 7.10]{KP}  implies $\left(U^n, \Phi^n\right) \rightarrow$ $(U, \Phi)$ in law, where $\Phi=\left(\Phi^1, \ldots, \Phi^d\right)$ are defined by
		$$
		\Phi^i=\sum_{k=1}^d \int_0^t \partial_k b^i\left(X_s\right) U_s^k \mathrm{d} s+\sum_{j=1}^m \sum_{k=1}^d \int_0^t \partial_k \sigma_j^i\left(X_s\right) U_s^k \mathrm{d} W_s^j+\sum_{j=1}^m \sum_{k=1}^d \int_0^{.} \partial_k \sigma_j^i\left(X_s\right) \mathrm{d} V_s^{k, j} .
		$$
		
		We can also show $\Phi^n$ is tight as a $\mathcal{C}_0^{1 / 2-\epsilon}$-valued sequence for any $\epsilon \in(0,1 / 2)$. Indeed, let
		\begin{align}\label{def of hat V}
			\hat{V}^{n, i}_{\cdot}=\sum_{j=1}^m \sum_{k=1}^d \int_0^{\cdot} \partial_k \sigma_j^i(X^n_{\frac{[ns]}{n}}) \mathrm{d} V_s^{n, k, j}.
		\end{align}
		By BDG's inequality, the boundness of the derivatives $\sigma$, \eqref{est of A-1} and \eqref{est of A-2}, we have for any $p>2$ and $0 \leq s<t \leq T$
		$$
		\begin{aligned}
			\E\Big[\big|\hat{V}_t^{n, i}-\hat{V}_s^{n, i}\big|^p\Big] & \leq C \sum_{j=1}^m \sum_{k=1}^d \E\Big[\Big|\int_s^{t} \partial_k \sigma_j^i(X^n_{\frac{[nu]}{n}})n^{2H}\big[X^n_u-X^n_{\frac{[nu]}{n}}-M_u^{n}\big]^{k}\d W^j_u\Big|^p\Big]\nonumber\\&\leq C\sum_{j=1}^m \sum_{k=1}^d\E\Big[\Big(\int_s^{t}\Big|\partial_k \sigma_j^i(X^n_{\frac{[nu]}{n}})n^{2H}\big[X^n_u-X^n_{\frac{[nu]}{n}}-M_u^{n}\big]^{k}\Big|^2\Big)^{p/2}\Big]\nonumber\\&\leq C\sum_{j=1}^m \sum_{k=1}^d\E\Big[\Big(\int_s^{t}\Big|n^{2H}\big[X^n_u-X^n_{\frac{[nu]}{n}}-M_u^{n}\big]^{k}\Big|^2\Big)^{p/2}\Big]\nonumber\\&\leq C|t-s|^{p/2}.\nonumber
		\end{aligned}
		$$
		Then for $t+h, t \in[0, T], h>0$ we obtain
		$$
		\begin{aligned}
			\left\|\Phi_{t+h}^{n, i}-\Phi_t^{n, i}\right\|_{L^p}&\leq  \sum_{k=1}^d\Big\|\int_t^{t+h} \partial_k b^i(X^n_s) U_s^{n, k} \mathrm{d} s\Big\|_{L^p} \\
			&\quad+\sum_{j=1}^m \sum_{k=1}^d\Big\|\int_t^{t+h} \partial_k \sigma_j^i\left(\hat{X}_s\right) U_s^{n, k} \mathrm{d} W_s^j\Big\|_{L^p} +\big\|\hat{V}_{t+h}^{n, i}-\hat{V}_t^{n, i}\big\|_{L^p}\\&\leq C\sum_{k=1}^d\int_t^{t+h}\| U_s^{n, k}\|_{L^p} \mathrm{d} s+C\sum_{j=1}^m \sum_{k=1}^d\Big(\int_t^{t+h}  \|U_s^{n, k}\|^2_{L^p} \mathrm{d} s \Big)^{\frac{1}{2}}\nonumber\\&\quad+Ch^{1/2}\leq Ch^{1/2},\nonumber
		\end{aligned}
		$$
	  where the BDG's inequality, Minkowski's inequality, Lemmas \ref{rate X-X n}, \ref{nabla b} and \ref{remainder term}, and the boundedness of the derivatives of $b,\sigma$ are used.  Hence, Kolmogorov's continuity criterion yields $\E\left[\left\|\Phi^{n, i}\right\|_{ \mathcal{C}_0^{1 / 2-\epsilon}} \right]$
	  $\leq C $ uniformly in $n$, and therefore, by   \cite[Corollary B.3]{FU}, $\Phi^n$ converges in law in $\mathcal{C}_0^{1 / 2-\epsilon}$ to $\Phi$.
		
		Let $\alpha=\frac{1}{2}-H$ and $\gamma=\frac{1}{2}-\epsilon$. By   \cite[Lemma A.3]{FU}, the operator $\mathcal{J}^\alpha$ is continuous from $\mathcal{C}_0^\gamma$ into $\mathcal{C}_0^{\gamma-\alpha}$. Since $\left(U^n, \Phi^n\right)$ converges in law to $(U, \Phi)$ in $\mathcal{C}_0^{\gamma-\alpha} \times \mathcal{C}_0^\gamma$, the continuous mapping theorem implies that $U^n-\mathcal{J}^\alpha \Phi^n$ converges in law to $U-\mathcal{J}^\alpha \Phi$ in $\mathcal{C}_0^{\gamma-\alpha}=\mathcal{C}_0^{H-\epsilon}$. On the other hand, Lemmas \ref{nabla b} and \ref{remainder term} imply $U^n-\mathcal{J}^\alpha \Phi^n$ converges in law to zero. Consequently $U-\mathcal{J}^\alpha \Phi=0$, which is equivalent to \eqref{limit Volterra eq}.
	\end{proof}
	\begin{lem}\label{tight of U n}
		The sequence $U^n$ is tight in $\mathcal{C}_0^{H-\epsilon}$ for any $\epsilon\in (0,H)$.
	\end{lem}
	\begin{proof}
		Define $\hat{U}^{n, i}$ by
		$$
		\hat{U}_t^{n, i}=\int_0^t K(t-s) \nabla b^i(X^n_s) U_s^{n, i} \mathrm{d} s+\sum_{k=1}^d \sum_{j=1}^m \int_0^t K(t-s) U_s^{n, i} \partial_k \sigma_j^i(X^n_s) \mathrm{d} W_s^j .
		$$
		
		In order to show the tightness of $\{\hat{U}^{n, i}\}$ in $\mathcal{C}_0^{H-\epsilon}$. By Theorem B.2 in \cite{FU}, it suffices to show that there exists a uniform bound for $\EE[\|\hat{U}_t^{n, i}\|_{\mathcal{C}_0^{H-\epsilon^{\prime}}}]$ for $\epsilon^{\prime} \in(0, \epsilon)$. Since the derivatives of $b$ and $\sigma$ are bounded, by Lemma \ref{basic lemma}-(iii), (iv) and Lemma \ref{rate X-X n} we have for $p \geq 2$
		\begin{align}
			\EE\Big[&\big|\hat{U}_{t+h}^{n, i}-\hat{U}_t^{n, i}\big|^p\Big]\nonumber\\ &\leq 4^{p-1}\E\Big[\Big|\int_0^t\Big(K(t+h-s)-K(t-s) \Big)\nabla b^i(X^n_s) U_s^{n, i} \mathrm{d} s\Big|^p\Big]\nonumber\\&\quad+4^{p-1}\E\Big[\Big|\int_t^{t+h}K(t+h-s)\nabla b^i(X^n_s) U_s^{n, i} \mathrm{d} s\Big|^p\Big]\nonumber\\&\quad+4^{p-1}\E\Big[\Big|\int_0^t\Big(K(t+h-s)-K(t-s) \Big)\nabla \partial_k \sigma_j^i(X^n_s) U_s^{n, i} \mathrm{d} W^j_s\Big|^p\Big]\nonumber\\&\quad+4^{p-1}\E\Big[\Big|\int_t^{t+h}K(t+h-s)\partial_k \sigma_j^i(X^n_s) U_s^{n, i} \mathrm{d} W^j_s\Big|^p\Big]\nonumber\\&\leq Ch^{Hp}\sup_{t \in[0, T]}\E[|U^{n,i}_t|^p]\leq Ch^{Hp}, \nonumber
		\end{align}
		where $C$ is independent of $t$ and $n$, and hence, Kolmogorov's continuity criterion and Theorem B.2 of \cite{FU} lead to the tightness in $\mathcal{C}_0^{H-\epsilon}$ of $\{\hat{U}^{n, i}\}$.
		
		In light of Lemmas \ref{nabla b} and \ref{remainder term}, it   remains  to show that $\{\int_0^{\cdot} K(\cdot-s) \mathrm{d} \hat{V}_s^{n, i}\}_{n \in \mathbb{N}}$ is tight, where $\hat{V}^{n, i}$ is defined by \eqref{def of hat V}. Similar to the above, by the boundness of the derivatives of $\sigma$, \eqref{def of V}, \eqref{est of A-1}, \eqref{est of A-2} and Lemma \ref{basic lemma}-(iv) we have for any $t+h, t \in[0, T], h>0, p>2$
		\begin{align}
			&\E\Big[\Big|\int_0^{t+h} K(t+h-s) \mathrm{d} \hat{V}_s^{n, i}-\int_0^{t} K(t-s) \mathrm{d} \hat{V}_s^{n, i}\Big|^p\Big]\nonumber\\ &\leq C\E\Big[\Big|\int_{0}^{t}\Big(K(t+h-s)-K(t-s)\Big)n^{2H}\big[X^n_s-X^n_{\frac{[ns]}{n}}-M_s^{n}\big]^{k}\d W_s^j\Big|^p\Big]\nonumber\\&\quad+C\E\Big[\Big|\int_{t}^{t+h}K(t+h-s)n^{2H}\big[X^n_s-X^n_{\frac{[ns]}{n}}-M_s^{n}\big]^{k}\d W_s^j\Big|^p\Big]\nonumber\\&\leq Ch^{Hp}n^{2Hp}\sup_{s\in[0,T]}\E\Big[\Big|\big[X^n_s-X^n_{\frac{[ns]}{n}}-M_s^{n}\big]^{k}\Big|^p\Big]\nonumber\\&\leq Ch^{Hp}\,. 
		\end{align}
		By Kormogorov's continuity criterion and Theorem B.2 of \cite{FU}, we obtain the tightness in $\mathcal{C}_0^{H-\varepsilon}$ of $\{\int_0^{\cdot} K(\cdot-s) \mathrm{d} \hat{V}_s^{n, i}\}_{n \in \mathbb{N}}$. So the tightness of $\mathcal{C}_0^{H-\epsilon}$ of $\{U^{n,i}\}$ is verified.
	\end{proof}
	\begin{lem}\label{unqueness of U}
		The uniqueness in law holds for continuous solution \eqref{limit Volterra eq}.
		%\footnote{Linear
		%	equation, the existence and uniqueness is obvious?}  
	\end{lem}
	\begin{proof}
		By the Yamada-Watanabe argument, it suffices to show that the pathwise uniqueness holds for continuous solutions of \eqref{limit Volterra eq}.
		
		We first show the $L^p$ integrability of the SVE \eqref{limit Volterra eq}.
		Let $\tau_n=\inf \{t||U_t| \geq n\}$. Using the local property of stochastic integration, Lemma \ref{basic lemma}-(i), (ii), Lemma \ref{bound of X}, Fubini's theorem and the boundedness of the derivatives of $\sigma$, we have that
		\begin{align}
			&\E\big[|U_t|^p\mathbb{I}_{\{t<\tau_n\}}\big]\nonumber\\&\leq C\E\Big[\Big|\int_0^t K(t-s) U^k_s\mathbb{I}_{\{s<\tau_n\}} \mathrm{d} s\Big|^p\Big]+C\E\Big[\Big|\int_0^t K(t-s) U^k_s\mathbb{I}_{\{s<\tau_n\}} \mathrm{d} W^j_s\Big|^p\Big]\nonumber\\&\quad+C\E\Big[\Big|\int_0^t K(t-s) \sigma_i^l(X_s) \mathrm{d} B^{i,i',j}_s\Big|^p\Big]\nonumber\\&\leq C_1+C_2\int_{0}^{t}\E\big[|U_s|^p\mathbb{I}_{\{s<\tau_n\}}\big]\d s\nonumber
		\end{align}
		for some $C_1, C_2>0$, independent of $t$. Therefore, Gronwall's lemma yields
		$$
		\E\Big[|U_t|^p \mathbb{I}_{\{t<\tau_n\}}\Big] \leq C_1 e^{C_2 t} \leq C_1 e^{C_2 T},
		$$
		and thus, by Fatou's lemma, we see
		$$
		\E[|U_t|^p]=\E\Big[\liminf _{n \rightarrow \infty}|U_t|^p \mathbb{I}_{\{t<\tau_n\}}\Big] \leq \liminf _{n \rightarrow \infty} \E[|U_t|^p \mathbb{I}_{\{t<\tau_n\}}] \leq C_1 e^{C_2 T},
		$$
		that is, $U_t$ is in $L^p$ for any $t \in[0, T]$.
	\end{proof}
	We now prove  the main theorem.
	\begin{proof}[Proof of Theorem \ref{t.main}]
		By Lemma \ref{tight of U n} and Lemma B.1 in \cite{FU}, or more directly by Lemma \ref{rate X-X n}, $X^n \rightarrow X$ in probability in the uniform topology. Therefore,
		$$
		\left(\left\{\nabla b^i(X^n)\right\}_i,\left\{\partial_k \sigma_j^i(X^n)\right\}_{i j k}\right) \rightarrow\left(\left\{\nabla b^i(X)\right\}_i,\left\{\partial_k \sigma_j^i(X)\right\}_{i j k}\right)
		$$
		in probability in the uniform topology as well. Together with Lemmas \ref{charact of V} and \ref{tight of U n}, we conclude that
		$$
		\left(U^n, V^n,\left\{\nabla b^i(X^n)\right\}_i,\left\{\partial_k \sigma_j^i(X^n)\right\}_{i j k}, Y\right)
		$$
		is tight in $\mathcal{C}_0^{H-\epsilon} \times \mathcal{C}_0 \times \mathcal{D}_{d^2} \times \mathcal{D}_{d^2 m} \times \mathbb{R}$ for any random variable $Y$ on $(\Omega, \mathscr{F}, \PP)$. For any subsequence of this tight sequence, there exists a further subsequence which converges by Prokhorov's theorem (see, e.g., Theorem 5.1 of Billingsley \cite{Bil} for a nonseparable case). Lemmas \ref{limit of U} and \ref{unqueness of U} imply the uniqueness of the limit. Therefore the original sequence itself has to converge. The limit $U$ of $U^n$ is characterized by \eqref{limit Volterra eq} again by Lemma \ref{limit of U}. The convergence of $U^n$ is stable because $Y$ is arbitrary.
	\end{proof}
	\section*{Conflicts of interests}
	The authors declare no conflict of interests.

	{\bf Acknowledgement} S. Liu was supported by China Scholarship Council. Y. Hu was supported by an  NSERC discovery fund and a Centennial  fund of University of Alberta. H. Gao was supported  in part by the NSFC Grant Nos. 12171084, the Jiangsu Provincial Scientific Research Center of Applied Mathematics under Grant No. BK20233002 and the fundamental Research Funds for the Central Universities No. RF1028623037.

	%  	\bibliographystyle{abbrvnat}
	
	% \bibliographystyle{accountm}
	%	\bibliographystyle{abbrv}
	% \bibliographystyle{nar}
 %\bibliographystyle{plain}
 	\bibliographystyle{amsplain}
	
	\bibliography{ref}
	
\end{document}